\documentclass[natbib]{amsart}
\usepackage{amssymb,latexsym,amscd,amsthm,epsfig,amsmath,amssymb}
\usepackage{color}
\usepackage{xcolor}
\usepackage[colorlinks = true]{hyperref}
\hypersetup{
    colorlinks,
    linkcolor={red!50!black},
    citecolor={green!50!black},
    urlcolor={blue!80!black}
}
\usepackage{tikz}
\usepackage{multirow}
\usepackage[english]{babel}
\usepackage{rotating}
\usepackage{polski}
\usepackage{longtable}
\usepackage{geometry}

\usepackage{algorithmic}

\numberwithin{equation}{section}

\usepackage{todonotes}

 \newtheorem{thm}{Theorem}[section]

 \newtheorem{conj}[thm]{Conjecture}
 \newtheorem{problem}{Problem}
 \newtheorem{proposition}[thm]{Proposition}
\newtheorem{lemma}[thm]{Lemma}
\newtheorem{corollary}[thm]{Corollary}

 \theoremstyle{definition}

 \newtheorem{dfn}[thm]{Definition}

 \newtheorem{example}[thm]{Example}
\newtheorem{rmk}[thm]{Remark}
\newtheorem{notation}[thm]{Notation}
\newtheorem{algorithm}[thm]{Algorithm}

\renewenvironment{quote}{%
   \list{}{%
     \it\leftmargin1.0cm   % this is the adjusting screw
     \rightmargin\leftmargin
   }
   \item\relax
}
{\endlist}

 \newcommand{\set}[1]{\left\{#1\right\}}

 \newcommand{\PP}{\mathbb{P}}
 \newcommand{\NN}{\mathbb{N}}

\newcommand{\bvec}[1]{\ensuremath{\mathbf{#1}}}
\newcommand{\bbP}{\mathbb{P}}
\newcommand{\bbG}{\mathbb{G}}
\newcommand{\bbN}{\mathbb{N}}
\newcommand{\bbR}{\mathbb{R}}
\newcommand{\bfV}{\mathbf{V}}
\newcommand{\rk}{\mathrm{R}}
\newcommand{\bfx}{\mathbf{x}}
\newcommand{\bfy}{\mathbf{y}}
\newcommand{\bfd}{\mathbf{d}}

\newcommand{\bfe}{\mathbf{e}}
\newcommand{\bfn}{\mathbf{n}}
\newcommand{\HF}{\mathrm{HF}}
\newcommand{\HS}{\mathrm{HS}}
\newcommand{\calF}{\mathcal{F}}
\newcommand{\calI}{\mathcal{I}}
\newcommand{\calL}{\mathcal{L}}
\newcommand{\calC}{\mathcal{C}}
\newcommand{\calO}{\mathcal{O}}
\newcommand{\downmapsto}{\rotatebox[origin=c]{-90}{$\scriptstyle\mapsto$}\mkern2mu}
\newcommand{\expdim}{{\it exp.\dim}}
\newcommand{\virtdim}{{\it vir.\dim}}
\newcommand{\incV}{\mathfrak{I}}
\newcommand{\Res}{\mathrm{Res}}
\newcommand{\Tr}{\mathrm{Tr}}
\newcommand{\bbC}{\mathbb{C}}
\newcommand{\bbX}{\mathbb{X}}
\newcommand{\VSP}{\mathrm{VSP}}

\def\move-in{\parshape=1.75true in 5true in}
\newcommand\blfootnote[1]{%
  \begingroup
  \renewcommand\thefootnote{}\footnote{#1}%
  \addtocounter{footnote}{-1}%
  \endgroup
}
\usepackage{xypic}

\begin{document}

\title{The Hitchhiker guide to: Secant Varieties and Tensor Decomposition.}

\author[A. Bernardi]{Alessandra Bernardi$^*$}
\address[A. Bernardi]{Dipartimento di Matematica, Universit\`a di Trento, Trento, Italy}
\email{alessandra.bernardi@unitn.it}

\author[E. Carlini]{Enrico Carlini$^*$}
\address[E. Carlini]{Dipartimento di Scienze Matematiche, Politecnico di Torino, Turin, Italy}
\email{enrico.carlini@polito.it}

\author[M.V. Catalisano]{Maria Virginia Catalisano$^*$}
\address[M.V. Catalisano]{Dipartimento di Ingegneria Meccanica, Energetica, Gestionale e dei Trasporti, Universit\`a degli studi di Genova, Genoa, Italy,}
\email{catalisano@dime.unige.it}

\author[A. Gimigliano]{Alessandro Gimigliano$^*$}
\address[A. Gimigliano]{Dipartimento di Matematica, Univ. di Bologna, Bologna, Italy}
\email{Alessandr.Gimigliano@unibo.it}

\author[A. Oneto]{Alessandro Oneto$^*$}
\address[A. Oneto]{Barcelona Graduate School of Mathematics, and Universitat Polit\`ecnica de Catalunya, Barcelona, Spain}
\email{alessandro.oneto@upc.edu; aless.oneto@gmail.com}

\subjclass[2010]{13D40, 14A05, 14A10, 14A15, 14C20, 14M12, 14M15, 14M99, 14N05, 15A69, 15A75}

\maketitle

\centerline {${}^*$ {\it The initial elaboration of this work had Anthony V. Geramita}}

\centerline {{\it as one of the authors. The paper is dedicated to him.}}

%%%%%%%%%%%%%%%%%%%%%%%%%%%%%%%%%%%%%%%%%%%%%%%%%%%%%%%%%%%%%%%%%%%%%%%%%%

\begin{abstract}
We consider here the problem, which is quite classical in Algebraic
geometry, of studying the secant varieties of a projective variety
$X$. The case we concentrate on is when $X$ is a Veronese variety, a
Grassmannian or a Segre variety. Not only these varieties are among
the ones that have been most classically studied, but a strong
motivation in taking them into consideration is the fact that they
parameterize, respectively, symmetric, skew-symmetric and general
tensors, which are decomposable, and their secant varieties give a
stratification of tensors via tensor rank. We collect here most of
the known results and the open problems on this fascinating subject.
\end{abstract}

\blfootnote{{\it Keywords.} Additive decompositions; Secant varieties; Veronese varieties; Segre varieties; Segre-Veronese varieties; Grassmannians; tensor rank; Waring rank; algorithm.}

%%%%%%%%%%%%%%%%%%%%%%%%%%%%%%%%%%%%%%%%%%%%%%%%%%%%%%%%%%%%%%%%%%%%%%%%%%
\tableofcontents
%%%%%%%%%%%%%%%%%%%%%%%%%%%%%%%%%%%%%%%%%%%%%%%%%%%%%%%%%%%%%%%%%%%%%%%%%%

\addtocontents{toc}{\protect\setcounter{tocdepth}{1}}

{\footnotesize
\subsection*{Funding}
	A.B. acknowledges the Narodowe Centrum Nauki, Warsaw Center of Mathematics and Computer Science, Institute of Mathematics of university of Warsaw which financed  the 36th Autumn School in Algebraic Geometry ``Power sum decompositions and apolarity, a geometric approach" September 1st-7th, 2013 {\L}uk\k{e}cin, Poland where a partial version of the present paper was conceived. M.V. C. and A. G. acknowledges financial support from MIUR (Italy). A.O. acknowledges financial support from the Spanish Ministry of Economy and Competitiveness, through the Mar\'ia de Maeztu Programme for Units of Excellence in R$\&$D (MDM-2014-0445).
	}

\addtocontents{toc}{\protect\setcounter{tocdepth}{2}}

%%%%%%%%%%%%%%%%%%%%%%%%%%%%%%%%%%%%%%%%%%%%%%%%%%%%%%%%%%%%%%%%%%%%%%%%%%
%\tableofcontents
%%%%%%%%%%%%%%%%%%%%%%%%%%%%%%%%%%%%%%%%%%%%%%%%%%%%%%%%%%%%%%%%%%%%%%%%%%
\section{Introduction}
\unskip
%%%%%%%%%%%%%%%%%%%%%%%%%%%%%%%%%%%%%%%%%%%%%%%%%%%%%%%%%%%%%%%%%%%%%%%%%%

\subsection{The Classical Problem}

When considering finite dimensional vector spaces over a field
$\Bbbk$ (which for us, will always be algebraically closed and of
characteristic zero, unless stated otherwise), there are three main functors
that come to attention when doing multilinear algebra:
\begin{itemize}
	\item the {tensor product}, denoted by $V_1\otimes \cdots \otimes V_d$;
	\item the {symmetric product}, denoted by $S^dV$; 
	\item the {wedge product}, denoted by $\bigwedge^dV$.
\end{itemize}

These functors are associated with three classically-studied projective varieties in algebraic geometry (see e.g., \cite{Harris}): 
\begin{itemize}
	\item the {Segre variety};
	\item the {Veronese variety};
	\item the {Grassmannian}. 
\end{itemize}

We will address here the problem of studying the higher secant
varieties $\sigma_s(X)$, where $X$ is one of the varieties above. We
have:
\begin{equation}\label{equ:secant} \sigma_s(X) := \overline{\bigcup _{P_1,\ldots,P_s\in X} \langle P_1,\ldots,P_s\rangle}\end{equation}
i.e., $\sigma_s(X)$ is the Zariski closure of the union of the $\PP
^{s-1}$'s, which are $s$-secant to $X$.

The problem of determining the dimensions of the higher secant
varieties of many classically-studied projective varieties (and also
projective varieties in general) is quite classical in algebraic geometry and has a long and interesting
history. By a simple count of parameters, the {expected dimension} of $
\sigma_s(X)$, for $X\subset \PP^N$, is $\min\{s(\dim X)+(s-1),
N\}$. This is always an upper-bound of the actual dimension, and a variety $X$ is said to be
{defective}, or {$s$-defective}, if there is a value $s$ for which
the dimension of $\sigma_s(X)$ is strictly smaller than the expected one;
the difference: 
$$
\delta_s(X) := \min\{s(\dim X)+(s-1), N\} - \dim
\sigma_s(X)
$$ 
is called the {$s$-defectivity} of $X$ (or of $\sigma_s(X)$); a variety $X$ for which some $\delta_s$ is positive is called {defective}.

The first interest in the secant variety $\sigma_2(X)$ of a variety
$X\subset \PP^N$ lies in the fact that if $\sigma_2(X) \neq \PP
^N$, then the projection of $X$ from a generic point of $\PP^N$ into
$\PP^{N-1}$ is an isomorphism. This goes back to the XIX Century with the
discovery of a surface $X \subset \PP ^5$, for which
$\sigma_2(X)$ is a hypersurface, even though its expected dimension
is five. This is the Veronese surface, which is the only surface in $\PP
^5$ with this property. The research on defective varieties has been quite a frequent
subject for classical algebraic geometers, e.g., see the works of
F. Palatini \cite{Pa}, A. Terracini \cite{Te, Te3} and G.%please check the use of initials throughout
 Scorza \cite{Sc1,
Sc2}.

It was then in the $1990$s that two new articles marked a turning point
in the study about these questions and rekindled the interest in
these problems, namely the work of F. Zak and the one by J. Alexander and
A. Hirshowitz.

 Among many other things, like, e.g., proving Hartshorne's conjecture on linear normality, the outstanding paper of F. Zak \cite{Z} studied {Severi varieties}, i.e.,
non-linearly normal smooth $n$-dimensional subvarieties $X\subset \PP^N$, with ${2\over
3}(N-1)=n$. Zak found that all Severi varieties have defective $\sigma_2(X)$, and, by
using invariant theory, classified all of them as follows.
\begin{thm} 
Over an algebraically-closed field of characteristic zero, each Severi variety is
projectively equivalent to one of the following four projective
varieties:
\begin{itemize}
	\item Veronese surface $\nu_2(\bbP^2) \subset \bbP^5$;
	\item Segre variety $\nu_{1,1}(\PP^2 \times \PP^2) \subset \PP^8$;
	\item Grassmann variety $Gr(1,5)\subset \PP^{14}$;
	\item Cartan variety ${\Bbb E}^{16} \subset \PP^{26}$.
\end{itemize}
\end{thm}
Moreover, later in the paper, also Scorza varieties are classified,
which are maximal with respect to defectivity and which generalize
the result on Severi varieties.

The other significant work is the one done by J. Alexander and
A. Hirschowitz; see \cite{AH95} and Theorem \ref{corollAH} below. Although
not directly addressed to the study of secant varieties, they confirmed
the conjecture that, apart from the quadratic Veronese varieties and
a few well-known exceptions, all the Veronese varieties have
higher secant varieties of the expected dimension. In a sense, this
result completed a project that was underway for over $100$ years (see
\cite{Pa,Te,War}).

%Among the important work in that period we'd like also to mention
%the paper of R. Ehrenborg and G.-C. Rota \cite{ER}.
%\alessandratodo{Ci sono alcuni errori in questo articolo, tra cui la congettura tra simmetrici e antisimmetrici chiaramente sbagliata... lo vogliamo davvero menzionare?}
%%%%%%%%%%%%%%%%%%%%%%%%%%%%%%%%%%%%%%%%%%%%%%%%%%%%%%%%%%%%%%%%%%%%%%%%%%

\subsection{Secant Varieties and Tensor Decomposition}
%\alessandratodo{Modificati i primi due paragrafi}
{Tensors} are multidimensional arrays of numbers and play an important role in numerous research areas including computational complexity, signal processing for telecommunications~\cite{deLC} and scientific data analysis \cite{SBG}.
As specific examples, we can quote the complexity of matrix multiplication~\cite{Str}, the P versus NP complexity problem 
 \cite{Vali}, the study of entanglement in quantum physics \cite{entanglement, bernardi2012algebraic}, matchgates in computer science \cite{Vali},
the study of phylogenetic invariants \cite{ar}, independent component analysis~\cite{Comon}, blind identification in signal processing \cite{SGB}, branching structure in diffusion images~\cite{ScSe} and other multilinear data analysis techniques in bioinformatics and spectroscopy \cite{CoJu}. Looking at this literature shows how knowledge about this subject used to 
be quite scattered and suffered a bit from
the fact that the same type of problem can be considered in different
areas using a different language.

In particular, {tensor decomposition} is nowadays an intensively-studied argument by many algebraic geometers and by more applied communities. Its main problem is the decomposition of a tensor with a given structure as a linear combination of {decomposable tensors} of the same structure called {rank-one tensors}. 
%The minimum number of decomposable tensors needed in the decomposition of a given tensor is called the {rank} of the tensor.
%A strong motivation for studying the secant varieties described
%above is their connection with the problem of how to minimally
%represent certain kinds of tensors as a sum of decomposable ones.
%This is a problem with a long history; for a taste of the results ant the several approaches 
%developed in the last twenty years of the last century see \cite{BCS},
%\cite {Wa},\cite {Gr},\cite{St}, \cite{Li},\cite {E}). These questions are also
% of interest in many applications, as in Electrical
%Engineering (Antenna Array Processing \cite {ACCF}, \cite{DM} and
%Telecommunications \cite{Ch}, \cite {dLC}); in Statistics
%(cumulant tensors, see \cite{McC}), or in Data Analysis (
%Independent Component Analysis \cite{Co1}, \cite{JS}). For
%other applications see also \cite{Co2}, \cite{CR},
%\cite{LMV}, \cite{SBG}.
To be more precise:
let $V_1, \ldots , V_d$ be $\Bbbk$-vector spaces of dimensions $n_1+1,\ldots,n_d+1$, respectively, and let
$
{V} = V_1^*\otimes \cdots \otimes V_d^* \simeq (V_1\otimes\ldots\otimes V_d)^*.
$
We call a {decomposable}, or {rank-one}, {tensor} an element of the type $v^*_1\otimes\cdots\otimes v^*_d \in \bfV$. If $T \in {V}$, one can ask: 
\begin{quote}
	{What is the minimal length of an expression of $T$ as a sum of decomposable tensors?}
\end{quote}

We call such an expression a {tensor decomposition} of $T$, and the answer to this question is usually referred to as the {tensor rank} of $T$. Note that, since ${V}$ is a finite-dimensional vector space of dimension $\prod _{i=1}^d \dim_\Bbbk V_i$,
which has a basis of decomposable tensors, it is quite trivial to see that every $T \in {V}$ can be written as the sum of finitely many decomposable tensors. Other natural questions to ask are:
\begin{quote}
{What is the rank of a generic tensor in ${V}$? What is the dimension of the closure of the set of all tensors of tensor rank $\leq r$?}
\end{quote}

Note that it is convenient to work up to scalar multiplication, i.e., in the projective space $\bbP(\bfV)$, and the latter questions are indeed meant to be considered in the {Zariski topology} of $\bbP(\bfV)$. This is the natural topology used in algebraic geometry, and it is defined such that closed subsets are zero loci of (homogeneous) polynomials and open subsets are always dense. In this terminology, an element of a family is said to be {generic} in that family if it lies in a proper Zariski open subset of the family. Hence, saying that a property holds for a generic tensor in $\bbP({V})$ means that it holds on a proper Zariski subset of $\bbP({V})$.

In the case $d=2$, tensors correspond to ordinary matrices, and the notion of tensor rank coincides with the usual one of the rank of matrices. Hence, the generic rank is the maximum one, and it is the same with respect to rows or to columns. When considering multidimensional tensors, we can check that in general, all these usual properties for tensor rank fail to hold; e.g., for
$(2\times 2\times 2)$-tensors, the generic tensor rank is two, but the
maximal one is three, and of course, it cannot be the dimension of the
space of ``row vectors'' in whatever direction.

 It is well known that studying the dimensions of the secant varieties to
Segre varieties gives a first idea of the stratification of ${V}$, or equivalently of $\PP ({V})$,
 with respect to tensor rank. In fact, the {Segre variety} $\nu_{1,\ldots,1}(\bbP^{n_1}\times\ldots\times\bbP^{n_d})$ can be seen as the projective variety in $\bbP(\bfV)$, which parametrizes rank-one tensors, and consequently, the generic point of $\sigma_s(\nu_{1,\ldots,1}(\bbP^{n_1}\times\cdots\times\bbP^{n_d}))$ parametrizes a tensor of tensor rank equal to $s$ (e.g., see \cite {CGG1,CGG2}).

If $V_1 = \cdots = V_d = V$ of dimension $n+1$, one can just consider {symmetric} or
{skew-symmetric} tensors. In the first case, we study the $S^dV^*$, which corresponds to the space of homogeneous polynomials in $n+1$ variables. Again, we have a notion of {symmetric decomposable tensors}, i.e, elements of the type $(v^*)^d \in S^dV^*$, which correspond to powers of linear forms. These are parametrized by the {Veronese variety} $\nu_d(\bbP^n)\subset\bbP(S^dV^*)$. In the skew-symmetric case, we consider $\bigwedge^dV^*$, whose {skew-symmetric decomposable tensors} are the elements of the form $v^*_1\wedge\ldots\wedge v^*_d \in \bigwedge^dV^*$. These are parametrized by the {Grassmannian} $Gr(d,n+1)$ in its Pl\"ucker embedding. Hence, we get a notion of {symmetric-rank} and of {$\wedge$-rank} for which one can ask the same questions as in the case of arbitrary tensors. Once again, these are translated into algebraic geometry problems on secant varieties of Veronese varieties and Grassmannians. 

Notice that actually, Veronese varieties 
%and Grassmannians, 
embedded
in a projective space corresponding to $\PP (S^dV^*)$ 
%and to $\PP(\bigwedge^d V^*)$, 
can be thought of as sections of the Segre variety in 
%\alessandratodo{NO! La Gassmanniana NO: i tensori di skew-symmetric-rank 1 NON hanno rank 1!!!}
 $\PP ((V^*)^{\otimes d})$ defined by the (linear) equations given by the
symmetry relations.% or, respectively, by the skew-symmetry ones.

Since the case of symmetric tensors is the one that has been classically considered more in depth, due to the fact that symmetric tensors correspond to homogeneous polynomials, we start from analyzing secant varieties of Veronese varieties in Section \ref{sec:SymmetricDec_Veronese}. Then, we pass to secant varieties of Segre varieties in Section \ref{sec:TensorDec_Segre}. Then, Section \ref{sec:OtherDec} is dedicated to varieties that parametrize other types of structured tensors, such as Grassmannians, which parametrize skew-symmetric tensors, Segre--Veronese varieties, which parametrize decomposable partially-symmetric tensors, Chow varieties, which parametrize homogeneous polynomials, which factorize as product of linear forms, varieties of powers, which parametrize homogeneous polynomials, which are pure $k$-th powers in the space of degree $kd$, or varieties that parametrize homogeneous polynomials with a certain prescribed factorization structure. In Section \ref{sec:ByondDim}, we will consider other problems related to these kinds of questions, e.g., what is known about maximal ranks, how to find the actual value of (or bounds on) the rank of a given tensor, how to determine the number of minimal decompositions of a tensor, what is known about the equations of the secant varieties that we are considering or what kind of problems we meet when treating this problem over $\mathbb{R}$, a case that is of course very interesting for applications.
%%%%%%%%%%%%%%%%%%%%%%%%%%%%%%%%%%%%%%%%%%%%%%%%%%%%%%%%%%%%%%%%%%%%%%%%%%

\section{Symmetric Tensors and Veronese Varieties}\label{sec:SymmetricDec_Veronese}

%\alessandratodo{Tutta questa intro l'ho aggiunta io}
A symmetric tensor $T$ is an element of the space $S^{d}V^*$, where $V^*$ is an $(n+1)$-dimensional $\Bbbk$-vector space and $\Bbbk$ is an algebraically-closed field.
%; the minimum integer $r$ such that
%$t$ can be written as the sum of $r$ elements of the type
%$v^{\otimes d}\in S^{d}V$ is called the {symmetric-rank
%}, or also the {Waring rank}), of $t$.
%
%\alessandratodo{Qui avevate messo che $T=[t]$ ma prima avete usato che $T$ era vettoriale, e poi mi pare che sia molto pi\`u usato $T$ per il vettoriale e $t$ per il proiettivo.}
%
%From now on we will indicate with $T$ the projective class of a
%symmetric tensor $t\in S^{d}V$, i.e., if $t\in S^{d}V$ then $T=[t]\in
%\PP (S^{d}V)$. We will write that an element $T\in \PP (S^{d}V)$ has
%symmetric-rank equal to $r$ meaning that there exists a tensor $t\in
%S^{d}V$ such that $T=[t]$ and $srk(t)=r$.
%
%
% In most applications it turns out that the knowledge of the symmetric-rank
% is quite useful, e.g., the symmetric-rank of a symmetric tensor extends the
%Singular Value Decomposition (SVD) problem for symmetric matrices
%(see \cite{GVL}).
%
It is quite immediate to see that we can associate a degree $d$ homogeneous polynomial in $\Bbbk[x_0,\ldots,x_n]$ with any symmetric tensor in $S^dV^*$. 

In this section, we address the problem of symmetric tensor decomposition.

\begin{quote} 
 {What is the smallest integer $r$ such that a given symmetric tensor $T \in S^dV^*$ can be written as a sum of $r$ symmetric decomposable tensors, i.e., as a sum of $r$ elements of the type $(v^*)^{\otimes d}\in~(V^*)^{\otimes d}$?}
\end{quote}

We call the answer to the latter question the {\emph{symmetric rank}} of $T$. Equivalently, %^ please confirm if the bold neceassary.

\begin{quote} 
 {What is the smallest integer $r$ such that a given homogeneous polynomial $F\in S^dV^*$ (a $(n+1)$-ary $d$-ic%define if appropriate/check the convention
, in classical language) can be written as a sum of $r$ $d$-th powers of linear forms?}
\end{quote}

We call the answer to the latter question the {Waring rank}, or simply {rank}, of $F$; denoted $\rk_{\mathrm{sym}}(F)$. Whenever it will be relevant to recall the base field, it will be denoted by $\rk_{\rm sym}^{\Bbbk}(F)$. Since, as we have said, the space of symmetric tensors of a given format can be naturally seen as the space of homogeneous polynomials of a certain degree, we will use both names for the rank.

The name ``Waring rank'' comes from an old problem in number theory regarding expressions of integers as sums of powers; we will explain it in Section \ref{waring}.

The first naive remark is that there are ${n+d \choose d}$ coefficients $a_{i_0,\cdots ,i_n}$ needed to write: 
$$F=\sum a_{i_0,\cdots ,i_n}x_0^{i_0}\cdots x_n^{i_n},$$ 
and $r(n+1)$ coefficients $b_{i,j} $ to write the same $F$ as: 
$$F=\sum_{i=1}^r (b_{i,0}x_0+ \cdots + b_{i,n}x_n)^d.$$ 
Therefore, for a general polynomial, the answer to the question should be that $r$ has to be at least such that $r(n+1)\geq {n+d \choose d}$. Then, the minimal value for which the previous inequality holds is $\left\lceil \frac{1}{n+1} {n+d\choose d}\right\rceil$. For $n=2$ and $d=2$, we know that this bound does not give the correct answer because a regular quadratic form in three variables cannot be written as a sum of two squares. On the other hand, a straightforward inspection shows that for binary cubics, i.e., $d=3$ and $n=1$, the generic rank is as expected. Therefore, the answer cannot be too simple.

The most important general result on this problem has been obtained by J. Alexander and A. Hirschowitz, in $1995$; see \cite{AH95}. It says that the generic rank is as expected for forms of degree $d\geq 3$ in $n\geq 1$ variables except for a small number of peculiar pairs $(n, d)$; see Theorem \ref{corollAH}.

What about {non-generic} forms? As in the case of binary cubics, there are special forms that require a larger $r$, and these cases are still being investigated. Other presentations of this topic from different points of view can be found in \cite{Ge, cgo, IaKa, Land}.

As anticipated in the Introduction, we introduce {Veronese varieties}, which parametrize homogeneous polynomials of symmetric-rank-one, i.e., powers of linear forms; see Section \ref{Verosection}. Then, in order to study the symmetric-rank of a generic form, we will use the concept of secant varieties as defined in \eqref{equ:secant}. %\alessandratodo{Visto che qui trattiamo tutti i tipi di rango, sarei per evitare confusioni e chiamare sempre ogni rango col suo nome giusto: symmetric-rank per i polinomi, skew-symmetric-rank per Grassmanniani e rank per Segre. Poi possiamo usare abbreviazioni se sono lunghi da scrivere volta per volta ma eviterei confusioni (soprattuto se vogliamo anche solo enunciare la Comon Conj da qualche parte. Suggestion: sym-rank (o s-rank) $R_s(T)$ o $R_{\circ}(T)$, wedge-rank (o w-rank) $R_w(T)$ o $R_{\wedge}(T)$ and rank $R(T)$.} 
%in Section \ref{Secantsection} the concept of secant varieties. 
In fact, the order of the first secant that fills the ambient space will give the symmetric-rank of a generic form. The dimensions of secant varieties to Veronese varieties were completely classified by J. Alexander and A. Hirschowitz in \cite{AH95} (Theorem \ref{corollAH}). We will briefly review their proof since it provides a very important constructive method to compute dimensions of secant varieties that can be extended also to other kinds of varieties parameterizing different structured tensors. In order to do that, we need to introduce apolarity theory (Section \ref{Apolaritysection}) and the so-called {Horace method} (Sections \ref{Horacesection} and \ref{Horacediffsection}).

The second part of this section will be dedicated to a more algorithmic approach to these problems, and we will focus on the problem of computing the symmetric-rank of a {given} homogeneous polynomial. 

In the particular case of binary forms, there is a very well-known and classical result firstly obtained by J. J. Sylvester in the XIX Century. We will show a more modern reformulation of the same algorithm presented by G. Comas and M. Seiguer in \cite{CS} and a more efficient one presented in~\cite{bgi}; see Section \ref{Sylvestersection}. In Section \ref{beyondsection}, we will tackle the more general case of the computation of the symmetric-rank of any homogeneous polynomial, and we will show the only theoretical algorithm (to our knowledge) that is able to do so, which was developed by J. Brachat, P. Comon, B. Mourrain and E. Tsigaridas in \cite{BCMT} with its reformulation \cite{taufer,bertau}. 

The last subsection of this section is dedicated to an overview of open problems.

%It is a very classical algebraic problem
%(inspired by a number theory problem posed by Waring in 1770, see
%\cite{War}), to determine which is the minimum integer $r$ such
%that a form of degree $d$ in $n+1$ variables can be written as a sum
%of $r$ $d$-th powers of linear forms. The so called ``~Big Waring
%Problem"~, is equivalent to determining the symmetric symmetric-rank of a
%generic symmetric tensor $t$.

%%%%%%%%%%%%%%%%%%%%%%%%%%%%%%%%%%%%%%%%%%%%%%%%%%%%%%%%%%%%%%%%%%%%%%%%%%
\subsection{On Dimensions of Secant Varieties of Veronese Varieties}
%\alessandratodo{Questa subsection \`e molto ampliata}

This section is entirely devoted to computing the symmetric-rank of a generic form, i.e., to the computation of the generic symmetric-rank. As anticipated, we approach the problem by computing dimensions of secant varieties of Veronese varieties. Recall that, in algebraic geometry, we say that a property holds for a generic form of degree $d$ if it holds on a Zariski open, hence dense, subset of~$\bbP(S^dV^*)$.

%%%%%%%%%%%%%%%%%%%%%%%%%%%%%%%%%%%%%%%%%%%%%%%%%%%%%%%%%%%%%%%%%%%%%%%%%%
\subsubsection{Waring problem for forms}\label{waring}
The problem that we are presenting here takes its name from an old question in number theory. In $1770$, E. Waring in \cite{War} stated (without proofs) that:
\begin{quote}
{``Every natural number can be written as sum of at most $9$ positive cubes, 
Every natural number can be written as sum of at most $19$ biquadratics.''}
\end{quote}

Moreover, he believed that: 
\begin{quote}
{``For all integers $d\geq 2$, there exists a number $g(d)$ such that each positive integer $n\in \mathbb{Z}^{+}$ can be written as sum of the $d$-th powers of $g(d)$ many positive integers, i.e., $n=a_{1}^{d}+\cdots +a_{g(d)}^{d}$ with $a_{i}\geq 0$.''}
\end{quote}

E. Waring's belief was shown to be true by D. Hilbert in $1909$, who proved that such a $g(d)$ indeed exists for every $d\geq 2$. In fact, we know from the famous four-squares Lagrange theorem (1770) that $g(2) = 4$, and more recently, it has been proven that $g(3) = 9$ and $g(4) = 19$. However, the exact number for higher powers is not yet known in general. In \cite{Dav39:Waring}, H. Davenport proved that any {sufficiently large} integer can be written as a sum of $16$ fourth powers. As a consequence, for any integer $d \geq 2$, a new number $G(d)$ has been defined, as the least number of $d$-th powers of positive integers to write any {sufficiently large} positive integer as their sum. Previously, C. F. Gauss proved that any integer congruent to seven modulo eight can be written as a sum of four squares, establishing that $G(2) = g(2) = 4$. Again, the exact value $G(d)$ for higher powers is not known in general.

This fascinating problem of number theory was then formulated for homogeneous polynomials as follows.

Let $\Bbbk$ be an algebraically-closed field of characteristic zero.
We will work over the projective space $\mathbb{P}^{n}=\mathbb{P}V$ where $V$ is an $(n+1)$-dimensional vector space over $\Bbbk$. 
We consider the polynomial ring $S=\Bbbk[x_{0},\ldots,x_{n}]$ with the graded structure
$S =\bigoplus_{d\geq0}S_{d}$, where $S_{d}=\langle x_{0}^{d},x_{0}^{d-1}x_{1},\ldots,x_{n}^{d}\rangle $ is the vector space of homogeneous polynomials, or {forms}, of degree $d$, which, as we said, can be also seen as the space $S^dV$ of symmetric tensors of order $d$ over $V$. 
%It is a well known fact that $\dim_{K}(S_{d})=\left( \begin{array}{c} d+n\\ n \end{array}\right)$. 
In geometric language, those vector spaces $S_{d}$ are called {complete linear systems of
hypersurfaces of degree $d$ in $\mathbb{P}^n$}. Sometimes, we will write $\mathbb{P}S_{d}$ in order to mean the projectivization of $S_{d}$, namely $\mathbb{P}S_{d}$ will be a $\mathbb{P}^{{n+d\choose d}-1}$ whose elements are classes of forms of degree $d$ modulo scalar multiplication, i.e., $[F]\in \mathbb{P}S_{d}$ with $F\in S_{d}$.

In analogy to the Waring problem for integer numbers, the so-called {little Waring problem for forms} is the following.
\begin{problem}[Little Waring problem]\label{lwp}
 Find the minimum $s\in \mathbb Z$ such that {all} forms $F\in S_d$ can be written as the sum of at most $s$ $d$-th powers of linear forms.
\end{problem}
The answer to the latter question is analogous to the number $g(d)$ in the Waring problem for integers. At the same time, we can define an analogous number $G(d)$, which considers decomposition in sums of powers of all numbers, but finitely many. In particular, the {big Waring problem for forms} can be formulated as follows.
\begin{problem}[Big Waring problem]\label{bwp}
Find the minimum $s\in \mathbb Z$ such that the {generic} form
$F\in S_d$ can be written as a sum of at most $s$ $d$-th powers of linear
forms.
\end{problem}
In order to know which elements of
$S_d$ can be written as a sum of $s$ $d$-th powers of linear forms,
we study the image of the map:
\begin{equation}\label{phi}\phi_{d,s} : \underbrace{S_{1}\times\cdots\times S_{1}}_{s}\longrightarrow S_{d}, \; \;
\phi_{d,s}(L_{1},\ldots,L_{s})=L_{1}^{d}+\cdots+L_{s}^{d}.
\end{equation} 
In terms of maps $\phi_{d,s}$, the little Waring problem (Problem \ref{lwp}) is to find the smallest $s$, such that $\mathrm{Im}(\phi_{d,s}) = S_d$. Analogously, to solve the big Waring problem (Problem \ref{bwp}), we require $\overline{\mathrm{Im}(\phi_{d,s})} = S_d$, which is equivalent to finding the minimal $s$ such that $\dim(\mathrm{Im}(\phi_{d,s})) = \dim S_d$.

The map $\phi_{d,s}$ can be viewed as a polynomial map between affine spaces:
$$\phi_{d,s} : \mathbb{A}^{s(n+1)} \longrightarrow \mathbb{A}^{N}, \quad \text{ with } N={n+d\choose n}.$$
In order to know the dimension of the image of such a map, we look
at its differential at a general point $P$ of the domain:
$$d\phi_{d,s}|_{P}:T_{P}(\mathbb{A}^{s(n+1)})\longrightarrow T_{\phi_{d,s}(P)}(\mathbb{A}^{N}).$$
Let $P=(L_{1},\ldots,L_{s})\in \mathbb{A}^{s(n+1)}$ and
$v=(M_{1},\ldots,M_{s}) \in T_{P}(\mathbb{A}^{s(n+1)}) \simeq
\mathbb{A}^{s(n+1)}$, where $L_{i},M_{i}\in S_1$ for $i=1,\ldots ,s$. Let us consider the
following parameterizations $t \longmapsto
(L_{1}+M_{1}t,\ldots,L_{s}+M_{s}t)$ of a line ${\mathcal{C}}$ passing through $P$ whose tangent vector at $P$ is $M$. The image of $\mathcal C$ via $\phi_{d,s}$ is
$\phi_{d,s}
(L_{1}+M_{1}t,\ldots,L_{s}+M_{s}t)=\sum_{i=1}^{s}
(L_{i}+M_{i}t)^{d}$. The tangent vector to $\phi_{d,s}(\mathcal{C})$ in
$\phi_{d,s}(P)$ is: 
\begin{equation}\label{eq:tangent}
\left.\frac{d}{dt}\right|_{t=0} \left( \sum_{i=1}^{s}
(L_{i}+M_{i}t)^{d} \right) = 
\sum_{i=1}^{s}\left.\frac{d}{dt}\right|_{t=0}(L_{i}+M_{i}t)^{d} = \sum_{i=1}^{s}dL_{i}^{d-1}M_{i}.
\end{equation}
 Now, as
$v=(M_{1},\ldots,M_{s})$ varies in $\mathbb{A}^{s(n+1)}$, the tangent
vectors that we get span $\langle L_{1}^{d-1}S_{1},\ldots,L_{s}^{d-1}S_{1}\rangle $. Therefore, we just proved the following.
\begin{proposition}
 Let $L_{1},\ldots,L_{s}$ be linear forms in $S=\Bbbk[x_{0},\ldots,x_{n}]$, where $L_{i}=a_{i,{0}}x_{0}+\cdots+a_{i,{n}}x_{n}$, and consider the map: $$\phi_{d,s}: \underbrace{S_{1}\times\cdots\times S_{1}}_{s} \longrightarrow S_{d}, \; \; \phi_{d,s} (L_{1},\ldots,L_{s})=L_{1}^{d}+\cdots+L_{s}^{d};$$
 then:
 $${\rm rk}(d\phi_{d,s})|_{(L_{1},\ldots,L_{s})}=\dim_\Bbbk\langle L_{1}^{d-1}S_{1},\ldots,L_{s}^{d-1}S_{1}\rangle .$$
\end{proposition}
It is very interesting to see how the problem of determining the latter dimension has been solved, because the solution involves many algebraic and geometric tools.

%%%%%%%%%%%%%%%%%%%%%%%%%%%%%%%%%%%%%%%%%%%%%%%%%%%%%%%%%%%%%%%%%%%%%%%%%%

\subsubsection{Veronese Varieties}\label{Verosection}
%\alessandratodo{Questa parte \`e molto espansa (la vostra \`e sotto \%)}

The first geometric objects that are related to our problem are the Veronese varieties.
We recall that a {\emph{Veronese variety}} can be viewed as (is projectively equivalent to) the image of the following $d$-ple%define if appropriate/check the convention
 embedding of $\PP^n$, where all degree $d$ monomials in $n+1$ variables appear in lexicographic order: % please confirm if the bold necessary.
\begin{equation}\label{vero1}
\begin{array}{rccc}
 \nu_d : & \mathbb{P}^n & \hookrightarrow & \mathbb{P}^{{n+d \choose d} -1} \\
 & $[$u_0: \ldots : u_n] & \mapsto & [u_0^d: u_0^{d-1}u_1:
u_0^{d-1}u_2: \ldots : u_n^d].\\
\end{array}
\end{equation}
With a slight abuse of notation, we can describe the Veronese map as follows:
\begin{equation}\label{vero2}
\begin{array}{rcccl}\nu_d:&\mathbb{P}
S_1= (\mathbb{P}^n)^{*}&\hookrightarrow& \mathbb{P}^{}
S_d&=\left(\mathbb{P}^{{n+d\choose d} -1}\right)^{*}\\
&[L]&\mapsto&[L^d]&
\end{array}.
\end{equation}
Let $X_{n,d} := \nu_d(\bbP^n)$ denote a Veronese variety.

Clearly, ``$\nu_d$ as defined in \eqref{vero1}'' and ``$\nu_d$ as defined in \eqref{vero2}'' are not the same map; indeed, from \eqref{vero2},
\begin{align*}
\nu_d([L])& =\nu_d\left([u_0x_0+ \cdots + u_nx_n]\right)= [L^d] = \\
& = \left[u_0^d: du_0^{d-1}u_1:
{d\choose 2}u_0^{d-1}u_2: \ldots : u_n^d\right]\in\bbP S_d.
\end{align*}
However, the two images are projectively equivalent. In order to see that, it is enough to consider the monomial basis of $S_d$ given by: 
$$\left\{ {d \choose \alpha} \bfx^\alpha ~|~\alpha = (\alpha_0,\ldots,\alpha_n) \in \bbN^{n+1}, |\alpha| = d\right\}.$$

Given a set of variables $x_0,\ldots,x_n$, we let $\bfx^\alpha$ denote the monomial $x_0^{\alpha_0}\cdots x_n^{\alpha_n}$, for any $\alpha \in \bbN^{n+1}$. Moreover, we write $|\alpha| = \alpha_0+\ldots+\alpha_n$ for its degree. Furthermore, if $|\alpha| = d$, we use the standard notation ${d \choose \alpha}$ for the multi-nomial coefficient $\frac{d!}{\alpha_0!\cdots\alpha_n!}$.

Therefore, we can view the Veronese variety either as the variety that
parametrizes $d$-th powers of linear forms or as the one parameterizing completely decomposable symmetric tensors.

\begin{example}[Twisted cubic]\label{ex1} Let $V=\Bbbk^2$ and $d=3$, then:
$$\begin{array}{rcccl}\nu_3:&\mathbb{P}^{1}&\hookrightarrow& \mathbb{P}^{3}\\
&[a_0:a_1]&\mapsto&[a_0^3:a_0^2a_1:a_0a_1^2:a_1^3]&\end{array}.$$
If we take $\{z_0, \ldots , z_3\}$ to be homogeneous coordinates in $\mathbb{P}^3$, then the Veronese curve in $\mathbb{P}^3$ (classically known as {twisted cubic}) is given by the solutions of the following system of equations:
$$\left\{ 
\begin{array}{l}
z_0z_2-z_1^2=0 
\\z_0z_3-z_1z_2=0 
\\z_1z_3-z_2^2=0
\end{array}\right. .$$
Observe that those equations can be obtained as the vanishing of all the maximal minors of the following matrix:
\begin{equation}\label{babycat}\left( \begin{array}{ccc} z_0 &z_1 & z_2 \\ z_1 & z_2 & z_3 \end{array}\right).\end{equation}
Notice that the matrix \eqref{babycat} can be obtained also as the defining matrix of the linear map:
$$S^2V^* \rightarrow S^1V, \; \;
\partial^2_{x_i} \mapsto \partial^2_{x_i}(F)$$
where $F=\sum_{i=0}^3 {d \choose i}^{-1}z_ix_0^{3-i}x_1^{i}$ and $\partial_{x_i}:=\frac{\partial}{\partial x_i}$.

Another equivalent way to obtain \eqref{babycat} is to use the so-called {flattenings}. We give here an intuitive idea about flattenings, which works only for this specific example.

 Write the $2\times 2 \times 2$ tensor by putting in position $ijk$ the variable $z_{i+j+k}$. This is an element of $V^*\otimes V^* \otimes V^*$. There is an obvious isomorphism among the space of $2\times 2 \times 2$ tensors $V^*\otimes V^* \otimes V^*$ and the space of $4\times 2$ matrices $(V^* \otimes V^*) \otimes V^*$. Intuitively, this can be done by {slicing} the $2\times 2 \times 2$ tensor, keeping fixed the third index. This is one of the three obvious possible flattenings of a $2\times 2 \times 2$ tensor: the other two flattenings are obtained by considering as fixed the first or the second index. Now, after having written all the possible three flattenings of the tensor, one could remove the redundant repeated columns and compute all maximal minors of the three matrices obtained by this process, and they will give the same ideal.
\end{example}

The phenomenon described in Example \ref{ex1} is a general fact. Indeed, Veronese varieties are always defined by $2\times 2$ minors of matrices constructed as \eqref{babycat}, which are usually called {catalecticant matrices}.
 
\begin{dfn}\label{CatalecticantMatrices} Let $F\in S_d$ be a homogeneous polynomial of degree $d$ in the polynomial ring $S = \Bbbk[x_0,\ldots,x_n]$. For any $i = 0,\ldots,d$, the \emph{$(i,d-i)$-th catalecticant matrix} associated to $F$ is the matrix representing the following linear maps in the standard monomial basis, i.e.,
$$
\begin{array}{rccc}
Cat_{i,d-i}(F) : & S_{i}^* & \longrightarrow & S_{d-i}, \\
& \partial^{i}_{\bfx^\alpha} & \mapsto & \partial^{i}_{\bfx^\alpha}(F), \\
\end{array}
$$
where, for any $\alpha \in \bbN^{n+1}$ with $|\alpha| = d-i$, we denote $\partial^{d-i}_{\bfx^\alpha} := \frac{\partial^{d-i}}{\partial x_0^{\alpha_0}\cdots\partial x_n^{\alpha_n}}$.

Let $\{z_\alpha ~|~ \alpha \in \bbN^{n+1}, ~|\alpha| = d\}$ be the set of coordinates on $\bbP S^dV$, where $V$ is $(n+1)$-dimensional. The \emph{$(i,d-i)$-th catalecticant matrix} of $V$ is the ${n+i \choose n} \times {n+d-i \choose n}$ matrix whose rows are labeled by $\mathcal{B}_i = \{\beta \in \bbN^{n+1} ~|~ |\beta| = i\}$ and columns are labeled by $\mathcal{B}_{d-i} = \{\beta \in \bbN^{n+1} ~|~ |\beta| = d-i\}$, given by:
$$
	Cat_{i,d-i}(V) = 
	\begin{pmatrix}
		z_{\beta_1 + \beta_2}
	\end{pmatrix}_{\substack{\beta_1 \in \mathcal{B}_i \\ \beta_2 \in \mathcal{B}_{d-i}}}.
$$
\end{dfn}
\begin{rmk}
Clearly, the catalecticant matrix representing $Cat_{d-i,i}(F)$ is the transpose of $Cat_{i,d-i}(F)$. Moreover, the most possible square catalecticant matrix is $Cat_{\lfloor d/2\rfloor ,\lceil d/2\rceil }(F)$ (and its transpose).
\end{rmk}
 Let us describe briefly how to compute the ideal of any Veronese variety.

\begin{dfn} A hypermatrix $A=(a_{i_{1}, \ldots , i_{d}})_{0\leq i_{j}\leq n, \, j=1, \ldots , d}$ is said to be {symmetric}, or {completely symmetric}, if $a_{i_{1}, \ldots , i_{d}}=a_{i_{\sigma(1)}, \ldots , i_{\sigma(d)}}$ for all $\sigma\in {\mathfrak{S}}_{d}$, where ${\mathfrak{S}}_{d}$ is the permutation group of $\{1, \ldots ,d\}$.
\end{dfn}

\begin{dfn} Let $H\subset V^{\otimes d}$ be the ${n+d\choose d}$-dimensional subspace of completely symmetric tensors of $V^{\otimes d}$, i.e., $H$ is isomorphic to the symmetric algebra $S^dV$ or the space of homogeneous polynomials of degree $d$ in $n+1$ variables. 
Let ${S}$ be a ring of coordinates of $\mathbb{P}^ {{n+d\choose d}-1}=\mathbb{P}H$ obtained as the quotient $S=\widetilde{S}/I$ where $\widetilde{S}=\Bbbk[x_{i_{1}, \ldots , i_{d}}]_{0\leq i_{j}\leq n, \, j=1, \ldots ,d}$ and $I$ is the ideal generated by all: $$x_{i_{1}, \ldots , i_{d}}-x_{i_{\sigma(1)}, \ldots , i_{\sigma(d)}}, \forall\; \sigma \in {\mathfrak{S}}_{d}.$$
The hypermatrix $(\overline{x}_{i_{1}, \ldots , i_{d}})_{0\leq i_{j}\leq n, \, j=1, \ldots , d}$, whose entries are the generators of $S$, is said to be a {\emph{generic symmetric~hypermatrix}}.
\end{dfn}

Let $A=(x_{i_{1}, \ldots , i_{d}})_{0\leq i_{j}\leq n, \, j=1, \ldots , d} $ be a generic symmetric hypermatrix, then it is a known result that the ideal of any Veronese variety is generated in degree two by the $2 \times 2$ minors of a generic symmetric hypermatrix, i.e., 
\begin{equation}\label{veronese}
	I(\nu_d(\mathbb{P}^n))=I_{2}(A) := (2 \times 2 \text{ minors of } A) \subset \tilde{S}.
\end{equation}
See \cite{Wa} for the set theoretical point of view. In \cite{Pu}, the author proved that the ideal of the Veronese variety is generated by the two-minors of a particular catalecticant matrix. In his PhD thesis \cite{parolin2004varieta}, A. Parolin showed that the ideal generated by the two-minors of that catalecticant matrix is actually $I_{2}(A)$, where $A$ is a generic symmetric hypermatrix.

\subsubsection{Secant Varieties}\label{Secantsection}
Now, we recall the basics on {secant varieties}.

\begin{dfn} 
Let $X\subset\mathbb{P}^{N}$ be a projective variety
of dimension $n$. We define the \emph{$s$-th secant variety} $\sigma_s(X)$ of $X$ as the closure of the union of all linear spaces spanned by $s$ points lying on $X$, i.e.,
$$\sigma_{s}(X):= \overline{\bigcup_{P_1,\ldots ,P_s\in X} \langle P_1,\ldots ,P_s\rangle} \subset \bbP^N. $$
For any $\calF \subset \bbP^n$, $\langle \calF \rangle$ denotes the linear span of $\calF$, i.e., the smallest projective linear space containing~$\calF$.

\end{dfn}
\begin{rmk}
The closure in the definition of secant varieties is necessary. Indeed, let $L_1,L_2\in S_1$ be two homogeneous linear forms. The polynomial $L_1^{d-1}L_2$ is clearly in $\sigma_2(\nu_d(\mathbb{P}(V)))$ since we can write: 
\begin{equation}\label{equation:remark}
L_1^{d-1}L_2=\lim_{t\rightarrow 0}\frac{1}{t}\left((L_1+tL_2)^d-L_1^d\right);
\end{equation}
however, if $d>2$, there are no $M_1,M_2\in S_1$ such that $L_1^{d-1}L_2=M_1^d+M_2^d$. This computation represents a very standard concept of basic calculus: {tangent lines are the limit of secant lines}. Indeed, by \eqref{eq:tangent}, the left-hand side of \eqref{equation:remark} is a point on the tangent line to the Veronese variety at $[L_1^d]$, while the elements inside the limit on the right-hand side of \eqref{equation:remark} are lines secant to the Veronese variety at $[L_1^d]$ and another moving point; see Figure \ref{figure:tangent_asLimit}.
\begin{figure}[h]
\centering
	\includegraphics[scale=0.35]{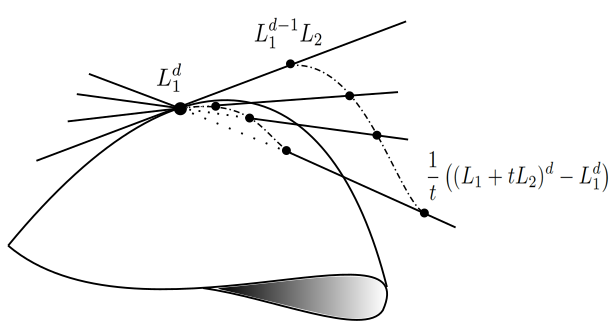}
	\caption{Representation of \eqref{equation:remark}.}
	\label{figure:tangent_asLimit}
\end{figure}
\end{rmk}
From this definition, it is evident that the generic element of $\sigma_{s}(X)$ is an element of some $\langle P_1,\ldots ,P_s\rangle$, with $P_i \in X$; hence, it is a linear combination of $s$ elements of $X$. This is why secant varieties are used to model problems concerning {additive decompositions}, which motivates the following general definition.
\begin{dfn}
	Let $X\subset\bbP^N$ be a projective variety. For any $P \in \bbP^N$, we define the \emph{$X$-rank} of $P$ as
	$$
		\rk_X(P) = \min\{s ~|~ P \in \langle P_1,\ldots,P_s \rangle, \text{ for } P_1,\ldots,P_s \in X\},
	$$
	and we define the \emph{border $X$-rank} of $P$ as
	$$
		\underline{\rk}_X(P) = \min\{s ~|~ P \in \sigma_s(X)\}.
	$$	
\end{dfn}
If $X$ is a non-degenerate variety, i.e., it is not contained in a proper linear subspace of the ambient space, we obtain a chain of inclusions
$$
	X \subset \sigma_2(X) \subset \ldots \subset \sigma_s(X) = \bbP^N.
$$
\begin{dfn}\label{linearization constant} 
The smallest $s\in \mathbb Z$ such that $\sigma_{s}(X)=\mathbb{P}^N$ is called the \emph{generic $X$-rank}. This is the $X$-rank of the generic point of the ambient space. %In some texts it is also called the {essential rank}, see eg. \cite{ER}.\alessandratodo{io toglierei anche questo riferimento all'essential rank... lo usa qualcun altro oltre a loro?}
\end{dfn}
The generic $X$-rank of $X$ is an invariant of the embedded variety $X$.

As we described in \eqref{vero2}, the image of the $d$-uple Veronese embedding of $\mathbb{P}^n = \bbP S_1$ can be viewed as the subvariety of $\PP S_d$ made by all forms, which can be written as $d$-th powers of linear forms. From this point of view, the generic rank $s$ of the Veronese variety is the minimum integer such that the generic form of degree $d$ in $n+1$ variables can be written as a sum of $s$ powers of linear forms. In other~words,
\begin{quote}
	{the answer to the Big Waring problem (Problem \ref{bwp}) is the generic rank with respect to the $d$-uple Veronese embedding in $\bbP S_d$.}
\end{quote}

This is the reason why we want to study the problem of determining the dimension of
$s$-th secant varieties of an $n$-dimensional projective variety $X\subset \mathbb{P}^N$.

Let $X^{s}:=\underbrace{X\times \cdots
\times X}_{s}$, $X_{0}\subset X$ be the open subset of regular points of $X$ and:
$$U_{s}(X) := \left\{ (P_1,\ldots ,P_{s})\in X^{s} ~|~ \substack{\forall \, i, \;P_i\in X_{0}, \text{ and } \\ 
\text{ the } P_i \text{'s are independent}}\right\}.$$
Therefore, for all $(P_1,\ldots ,P_s)\in U_{s}(X)$, since the $P_i$'s are linearly independent, the linear span
$H=\langle P_1,\ldots ,P_s\rangle $ is a $\mathbb{P}^ {s-1}$. Consider the following incidence variety:
 $$\incV^{s}(X)=\{ (Q,H)\in \mathbb{P}^ N \times U_{s}(X) \; | \; Q\in H \}.$$
If $s\le N+1$, the dimension of that incidence variety is:
$$\dim(\incV^{s}(X))=n(s-1)+n+s-1.$$
With this definition, we can consider the projection on the first factor:
$$\pi_1:\incV^{s}(X)\rightarrow \mathbb{P}^ N;$$
 the $s$-th secant variety of $X$ is just the closure of the image of this map, i.e.,
$$\sigma_{s}(X)=\overline{{\rm Im}(\pi_1:\incV^{s}(X)\rightarrow \mathbb{P}^ N)}.$$
Now, if $\dim(X)=n$, it is clear that, while $\dim(\incV^{s}(X))=ns+s-1$,
the dimension of $\sigma_{s}(X)$ can be smaller: it suffices that the
generic fiber of $\pi_1$ has positive dimension to impose
$\dim(\sigma_{s}(X))<n(s-1)+n+s-1$. Therefore, it is a general fact that, if $X\subset \mathbb{P}^ N$ and
$\dim(X)=n$, then,
$$\dim(\sigma_{s}(X))\leq \min \{N,sn+s-1 \}.$$

\begin{dfn}
A projective variety $X\subset \mathbb{P}^ N$ of dimension $n$ is said to
be \emph{$s$-defective} if $\dim(\sigma_{s}(X))<\min\{ N,sn+s-1 \}$. If so, we call \emph{$s$-th defect} of $X$ the difference: $$\delta_{s}(X):= \min\{
N,sn+s-1 \} -\dim(\sigma_{s}(X)).$$ Moreover, if $X$ is $s$-defective, then $\sigma_s(X)$ is said to be \emph{defective}. If $\sigma_s(X)$ is not defective, i.e., $\delta_s(X)=0$, then it is said to be \emph{regular} or of \emph{expected dimension}.
\end{dfn}

Alexander--Hirschowitz Theorem (\cite{AH95}) tells us that the dimension of
the $s$-th secant varieties to Veronese varieties is not always the
expected one; moreover, they exhibit the list of all the defective cases.

\begin{thm} [Alexander--Hirschowitz Theorem]\label{corollAH}
Let $X_{n,d}=\nu_{d}( \mathbb{P}^n)$, for $d\geq 2$, be a Veronese variety. Then:
 $$\dim(\sigma_s(X))=\min\left\{ {n+d \choose d } -1, sn+s-1\right\}$$
 except for the following cases:
 \begin{enumerate}
 \item[{\rm (1)}] $d=2$, $n\geq 2$, $s\leq n$, where $\dim(\sigma_s(X))=\min\left\{ {n+2 \choose 2} -1, 2n+1-{s \choose 2}\right\}$;
 \item[{\rm (2)}] $d=3$, $n=4$, $s=7$, where $\delta_s=1$;
 \item[{\rm (3)}] $d=4$, $n=2$, $s=5$, where $\delta_s=1$;
 \item[{\rm (4)}] $d=4$, $n=3$, $s=9$, where $\delta_s=1$;
 \item[{\rm (5)}] $d=4$, $n=4$, $s=14$, where $\delta_s=1$.
 \end{enumerate}
\end{thm}

Due to the importance of this theorem, we firstly give a historical review, then we will give the main steps of the idea of the proof. For this purpose, we will need to introduce many mathematical tools (apolarity in Section \ref{Apolaritysection} and fat points together with the Horace method in Section \ref{fat:points:section}) and some other excursuses on a very interesting and famous conjecture (the so-called SHGH%define if appropriate
 conjecture; see Conjecture \ref{conjSegre} and Conjecture \ref{conjGimi}) related to the techniques used in the proof of this theorem.

\label{history}The following historical review can also be found in \cite{bo}. 

The quadric cases ($d = 2$) are classical. The first non-trivial exceptional case $d=4$ and $n=2$ was known already by Clebsch in 1860 \cite{Cl}. He thought of the quartic as a quadric of quadrics and found that $\sigma_5(\nu_4(\mathbb{P}^2))\subsetneq \mathbb{P}^{14}$, whose dimension was not the expected one. Moreover, he found the condition that the elements of $\sigma_5(\nu_4(\mathbb{P}^2))$ have to satisfy, i.e., he found the equation of the hypersurface $\sigma_5(\nu_4(\mathbb{P}^2)) \subsetneq \mathbb{P}^{14}$: that condition was the vanishing of a $6\times 6$ determinant of a certain {catalecticant~matrix}.

To our knowledge, the first list of all exceptional cases was described by Richmond in \cite{Ri}, who showed all the defectivities, case by case, without finding any general method to describe all of them. It is remarkable that he could describe also the most difficult case of general quartics of $\mathbb{P}^4$. The same problem, but from a more geometric point of view, was at the same time studied and solved by Palatini in 1902--1903; see \cite{Pa1, Pa2}. In particular, Palatini studied the general problem, proved the defectivity of the space of cubics in $\mathbb{P}^4$ and studied the case of $n=2$. He was also able to list all the defective cases.

The first work where the problem was treated in general is due to Campbell (in 1891; therefore, his work preceded those of Palatini, but in Palatini's papers, there is no evidence of knowledge of Campbell's work), who in \cite{Cam}, found almost all the defective cases (except the last one) with very interesting, but not always correct arguments (the fact that the Campbell argument was wrong for $n = 3$ was claimed also in~\cite{Te3} in 1915).

 His approach is very close to the infinitesimal one of Terracini, who introduced in \cite{Te} a very simple and elegant argument (today known as Terracini's lemmas, the first of which will be displayed here as Lemma \ref{Terracini}), which offered a completely new point of view in the field. Terracini showed again the case of $n=2$ in \cite{Te}. In \cite{Te2}, he proved that the exceptional case of cubics in $\mathbb{P}^4$ can be solved by considering that the rational quartic through seven given points in $\mathbb{P}^4$ is the singular locus of its secant variety, which is a cubic hypersurface. In \cite{Te3}, Terracini finally proved the case $n=3$ (in 2001, Ro\'e, Zappal\`a and Baggio revised Terracini's argument, and they where able to present a rigorous proof for the case $n = 3$; see \cite{rzb}). 

In 1931, Bronowski \cite{br} tried to tackle the problem checking if a linear system has a vanishing Jacobian by a numerical criterion, but his argument was incomplete.

In 1985, Hirschowitz (\cite{Hir1}) proved again the cases $n=2,3$, and he introduced for the first time in the study of this problem the use of zero-dimensional schemes, which is the key point towards a complete solution of the problem (this will be the idea that we will follow in these notes). Alexander used this new and powerful idea of Hirschowitz, and in \cite{Ale}, he proved the theorem for $d \geq 5$.

Finally, in \cite{AH0,AH95} (1992--1995), J. Alexander and A. Hirschowitz joined forces to complete the proof of Theorem \ref{corollAH}. After this result, simplifications of the proof followed \cite{AH3, Chandler}.

After this historical excursus, we can now review the main steps of the proof of the Alexander--Hirschowitz theorem. As already mentioned, one of the main ingredients to prove is Terracini's lemma (see \cite{Te} or \cite{AAd}), which gives an extremely powerful technique to compute the dimension of any secant variety.

%\begin{dfn} 
%If $V$ is a vector space, then denote the {tangent space} to a point $\hat x\in M \subset V$ with $\hat T_{\hat x} M \subset V$.
% For a projective variety $X \subset \mathbb{P}(V)$ and $P\in X$, define the {affine tangent space to $X$ at $P$} as $\hat T_{P}X:= \hat T_{\hat P}\hat X$ where $\hat X$ is the affine cone over $X$. We will refer to the {tangent space of $X\subset \mathbb{P}^N$ in $P\in X$} as $T_PX:=\mathbb{P}(\hat T_{P}X)$.
%\end{dfn}

\begin{lemma} [Terracini's Lemma] \label{Terracini} 
Let $X$ be an irreducible non-degenerate variety in $\mathbb{P}^ N$, and let $P_1,\ldots ,P_s$ be $s$ generic points on $X$. Then, the tangent space to $\sigma_{s}(X)$ at a generic point $Q\in \langle P_1,\ldots ,P_s\rangle $ is the linear span in $\mathbb{P}^ N$ of the tangent spaces $T_{P_i}(X)$ to $X$ at $P_i$, $i=1,\ldots ,s$, i.e.,
$$T_{Q}(\sigma_{s}(X)) = \langle T_{P_1}(X),\ldots , T_{P_s}(X)\rangle.$$
\end{lemma}

This ``lemma'' (we believe it is very reductive to call it a ``lemma'') can be proven in many ways (for example, without any assumption on the characteristic of $\Bbbk$, or following Zak's book \cite{Z}). Here, we present a proof ``made by hand''.

\begin{proof} We give this proof in the case of $\Bbbk=\mathbb{C}$, even though it works in general for any algebraically-closed field of characteristic zero. 

We have already used the notation $X^{s}$ for $X\times \cdots \times X$ taken $s$ times. Suppose that $\dim(X)=n$. Let us consider the following incidence variety,
$$\incV=\left\{ (P;P_{1},\ldots ,P_{s}) \in \mathbb{P}^N \times X^{s} ~~\Big|~~  \let\scriptstyle\textstyle
\substack{P \in \langle P_{1},\ldots ,P_{s}\rangle \\
P_{1},\ldots ,P_{s} \in X}\right\}\subset \mathbb{P}^N\times X^{s},$$
and the two following projections,
$$\pi_{1}:\incV \rightarrow \sigma_{s}(X)\ , \quad \pi_{2}:\incV \rightarrow X^{s}.$$ 
The dimension of $X^{s}$ is clearly $sn$. If $(P_{1},\ldots , P_{s})\in X^{s}$, the fiber $\pi_{2}^{-1}((P_{1},\ldots ,P_{s}))$ is generically a $\mathbb{P}^ {s-1}$, $s<N$. Then, $\dim (\incV)=sn+s-1$. 

If the generic fiber of $\pi_{1}$ is finite, then $\sigma_{s}(X)$ is regular. i.e., it has the expected dimension; otherwise, it is defective with a value of the defect that is exactly the dimension of the generic fiber.

Let $(P_{1},\ldots , P_{s})\in X^{s}$ and suppose that each $P_{i}\in X \subset \mathbb{P}^ N$ has coordinates $P_{i}=[a_{i,0}:\ldots:a_{i,N}]$, for $i=1,\ldots ,s$. In an affine neighborhood $U_i$ of $P_{i}$, for any $i$, the variety $X$ can be locally parametrized with some rational functions $f_{i,j}:\Bbbk^{n+1} \rightarrow \Bbbk$, with $j=0,\ldots ,N$, that are zero at the origin. Hence, we~write:
$$X \supset U_i : \; \left\{ \begin{array}{l} 
x_{0}= a_{i,0}+ f_{i,0}(u_{i,0},\ldots ,u_{i,n})\\
\quad \vdots \\
x_{N}=a_{i,N}+f_{i,N}(u_{i,0}, \ldots ,u_{i,n})
\end{array} \right. .$$
Now, we need a parametrization $\varphi$ for $\sigma_{s}(X)$. Consider the subspace spanned by $s$ points of $X$, i.e.,
$$\langle (a_{1,0}+f_{1,0},\ldots , a_{1,N}+f_{1,N}),\ldots ,(a_{s,0}+f_{s,0},\ldots ,a_{s,N}+f_{s,N})\rangle,$$
where for simplicity of notation, we omit the dependence of the $f_{i,j}$ on the variables $u_{i,j}$; thus, an element of this subspace is of the form: 
$$\lambda_{1}(a_{1,0}+f_{1,0},\ldots , a_{1,N}+f_{1,N})+\cdots +\lambda_{s}(a_{s,0}+f_{s,0},\ldots , a_{s,N}+f_{s,N}),$$
for some $\lambda_{1}, \ldots ,\lambda_{s}\in \Bbbk$. We can assume $\lambda_{1}=1$. Therefore, a parametrization of the $s$-th secant variety to $X$ in an affine neighborhood of the point $P_1+\lambda_2P_2+\ldots+\lambda_sP_s$ is given by: 
\begin{align*}
(a_{1,0}+f_{1,0},\ldots,&a_{1,N}+f_{1,N})+ \\
& + (\lambda_{2}+t_{2})(a_{2,1}-a_{1,0}+f_{2,1}-f_{1,0},\ldots ,a_{2,N}-a_{1,N}+f_{2,N}-f_{1,N})+ \\ 
& + \cdots + \\
& + (\lambda_{s}+t_{s})(a_{s,1}-a_{1,0}+f_{s,1}-f_{1,0},\ldots ,a_{s,N}-a_{1,N}+f_{s,N}-f_{1,N}),
\end{align*}
for some parameters $t_{2},\ldots ,t_{s}$. Therefore, in coordinates, the 
parametrization of $\sigma_s(X)$ that we are looking for is the map $\varphi: \Bbbk^{s(n+1)+s-1} \rightarrow \Bbbk^{N+1}$ given by:
	\begin{center}
	$ (u_{1,0},\ldots,u_{1,n},u_{2,0},\ldots,u_{2,n},\ldots,u_{s,0},\ldots,u_{s,n},t_{2},\ldots ,t_{s})$
	
	${\LARGE \downmapsto}$
	
	$
	 (\ldots , a_{1,j}+f_{1,j}+ (\lambda_{2}+t_{2})(a_{2,j}-a_{1,j}+f_{2,j}-f_{1,j})+\cdots +(\lambda_{s}+t_{s})(a_{s,j}-a_{1,j}+f_{s,j}-f_{1,j}) ,\ldots),
	$
	\end{center}
	where for simplicity, we have written only the $j$-th element of the image.
Therefore, we are able to write the Jacobian of $\varphi$. We are writing it in three blocks: the first one is $(N+1)\times (n+1)$; the second one is $(N+1)\times (s-1)(n+1)$; and the third one is $(N+1)\times (s-1)$:
$$J_{\underline 0}(\varphi)=\left( \begin{array}{ccccc} (1-\lambda_{2}-\cdots -\lambda_{s})\frac{\partial f_{1,j}}{\partial u_{1,k}} & \Big| & \lambda_{i}\frac{\partial f_{i,j}}{\partial u_{i,k}} & \Big| & a_{i,j}- a_{1,j}\end{array} \right),$$
with $i=2,\ldots ,s$; $j=0, \ldots ,N$ and $k=0,\ldots ,n$. Now, the first block is a basis of the (affine) tangent space to $X$ at $P_{1}$, and in the second block, we can find the bases for the tangent spaces to $X$ at $P_{2},\ldots ,P_{s}$; the rows of:
$$\left(\begin{array}{ccc}
\frac{\partial f_{i,0}}{\partial u_{i,0}} &\cdots & \frac{\partial f_{i,0}}{\partial a_{i,N}}\\
\vdots &&\vdots \\
\frac{\partial f_{i,N}}{\partial a_{i,0}} &\cdots & \frac{\partial f_{i,N}}{\partial a_{i,N}}
\end{array}\right)$$
give a basis for the (affine) tangent space of $X$ at $P_i$.
\end{proof}

The importance of Terracini's lemma to compute the dimension of any secant variety is extremely evident. One of the main ideas of Alexander and Hirshowitz in order to tackle the specific case of Veronese variety was to take advantage of the fact that Veronese varieties are embedded in the projective space of homogeneous polynomials. They firstly moved the problem from computing the dimension of a vector space (the tangent space to a secant variety) to the computation of the dimension of its dual (see Section \ref{Apolaritysection} for the precise notion of duality used in this context). Secondly, their punchline was to identify such a dual space with a certain degree part of a zero-dimensional scheme, whose Hilbert function can be computed by induction (almost always). We will be more clear on the whole technique in the sequel. Now we need to use the language of schemes.

\begin{rmk} {Schemes} are locally-ringed spaces isomorphic to the spectrum of a commutative ring. Of course, this is not the right place to give a complete introduction to schemes. The reader interested in studying schemes can find the fundamental material in \cite{EH, hartshorne, mumford}. In any case, it is worth noting that we will always use only zero-dimensional schemes, i.e., ``points''; therefore, for our purpose, it is sufficient to think of zero-dimensional schemes as points with a certain structure given by the vanishing of the polynomial equations appearing in the defining ideal. For example, a homogeneous ideal $I$ contained in $\Bbbk[x,y,z]$, which is defined by the forms vanishing on a degree $d$ plane curve $\mathcal{C}$ and on a tangent line to $\mathcal{C}$ at one of its smooth points $P$, represents a zero-dimensional subscheme of the plane supported at $P$ and of {length} two, since the degree of intersection among the curve and the tangent line is two at $P$ (schemes of this kind are sometimes called {jets}).
\end{rmk}

\begin{dfn}
 A \emph{fat point} $Z\subset \mathbb{P}^n$ is a zero-dimensional scheme, whose defining ideal is of the form $\wp^m$, where $\wp$ is the ideal of a simple point and $m$ is a positive integer. In this case, we also say that $Z$ is a {$m$-fat point}, and we usually denote it as $mP$. We call the \emph{scheme of fat points} a union of fat points $m_1P_1+\cdots+m_sP_s$, i.e., the zero-dimensional scheme defined by the ideal $\wp_1^{m_1}\cap\cdots\cap\wp_s^{m_s}$, where $\wp_i$ is the prime ideal defining the point $P_i$, and the $m_i$'s are positive integers.
\end{dfn}

\begin{rmk}
	In the same notation as the latter definition, it is easy to show that $F \in \wp^m$ if and only if $\partial(F)(P) = 0$, for any partial differential $\partial$ of order $\leq m-1$. In other words, the hypersurfaces ``vanishing'' at the $m$-fat point $mP$ are the hypersurfaces that are passing through $P$ with multiplicity $m$, i.e., are singular at $P$ of order $m$. 
\end{rmk}

\begin{corollary}\label{corTerracini} Let $(X,{\mathcal L})$ be an integral, polarized
scheme. If ${\mathcal L}$ embeds $X$ as a closed scheme in $\mathbb{P}^ N$,
then:
$$ \dim (\sigma_{s}(X)) = N - \dim (h^0({\mathcal I}_
{Z,X}\otimes {\mathcal L})),$$ where $Z$ is the union of s%should it be italics?
 generic two-fat
points in $X$.
\end{corollary}

\begin{proof} By Terracini's lemma, we have that, for generic points $P_1,\ldots ,P_s\in X$, $$\dim (\sigma_{s}(X)) =\dim (\langle T_{P_1}(X),\ldots ,T_{P_s}(X)\rangle).$$ Since $X$ is
embedded in $ \mathbb{P}^{} (H^0(X,{\mathcal L})^*)$ of dimension $N$, we can view the
elements of $ H^0(X,{\mathcal L})$ as hyperplanes in $\mathbb{P}^ N$. The
hyperplanes that contain a space $T_{P_i}(X)$ correspond to
elements in $H^0({\mathcal I}_ {2P_i,X}\otimes {\mathcal L})$, since they
intersect $X$ in a subscheme containing the first infinitesimal
neighborhood of $P_i$. Hence, the hyperplanes of $\mathbb{P}^N$ containing
 $\langle T_{P_1}(X),\ldots ,T_{P_s}(X)\rangle $ are the sections of
$H^0({\mathcal I}_{Z,X}\otimes {\mathcal L})$, where $Z$ is the scheme
union of the first infinitesimal neighborhoods in $X$ of the
points $P_i$'s. 
\end{proof}

\begin{rmk} A hyperplane $H$ contains the tangent space to a non-degenerate projective variety $X$ at a smooth point $P$ if and only if the intersection $X\cap H$ has a singular point at $P$. In fact, the tangent space $T_{P}(X)$ to $X$ at $P$ has the same dimension of $X$ and $T_{P}(X\cap H)= H\cap T_{P}(X)$. Moreover, $P$ is singular in $H\cap X$ if and only if $\dim (T_{P}(X\cap H))\geq \dim (X\cap H)=\dim (X)-1$, and this happens if and only if $H \supset T_{P}(X)$.
\end{rmk}

\begin{example}[The Veronese surface of $\bbP^5$ is defective]\label{VeroSurfDef} Consider the Veronese surface $X_{2,2} = \nu_2(\PP^2)$ in $\mathbb{P}^ 5$. We want to show that it is two-defective, with $\delta_2=1$. In other words, since the expected dimension of $\sigma_2(X_{2,2})$ is $2\cdot 2 + 1$, i.e., we expect that $\sigma_2(X_{2,2})$ fills the ambient space, we want to prove that it is actually a hypersurface. This will imply that actually, it is not possible to write a generic ternary quadric as a sum of two squares, as expected by counting parameters, but at least three squares are necessary instead. 

Let $P$ be a general point on the linear span $\langle R,Q\rangle $ of two general points $R,Q\in X$; hence, $P \in \sigma_{2}(X_{2,2})$. By Terracini's lemma, $T_{P}(\sigma_{2}(X_{2,2}))=\langle T_{R}(X_{2,2}),T_{Q}(X_{2,2})\rangle $. The expected dimension for $\sigma_{2}(X_{2,2})$ is five, so $\dim (T_{P}(\sigma_{2}(X_{2,2})))<5$ if and only if there exists a hyperplane $H$ containing $T_{P}(\sigma_{2}(X_{2,2}))$. The previous remark tells us that this happens if and only if there exists a hyperplane $H$ such that $H\cap X_{2,2}$ is singular at $R, Q$. Now, $X_{2,2}$ is the image of $\mathbb{P}^2$ via the map defined by the complete linear system of quadrics; hence, $X_{2,2} \cap H$ is the image of a plane conic. Let $R', Q'$ be the pre-images via $\nu_{2}$ of $R,Q$ respectively. Then, the double line defined by $R',Q'$ is a conic, which is singular at $R',Q'$. Since the double line $\langle R',Q'\rangle $ is the only plane conic that is singular at $R',Q'$, we can say that $\dim (T_{P}(\sigma_{2}(X_{2,2})))=4<5$; hence, $\sigma_2(X_{2,2})$ is defective with defect equal to zero. 

Since the two-Veronese surface is defined by the complete linear system of quadrics, Corollary \ref{corTerracini} allows us to rephrase the defectivity of $\sigma_{2}(X_{2,2})$ in terms of the {number of conditions} imposed by two-fat points to forms of degree two; i.e., we say that

{two two-fat points of $\mathbb{P}^2$ do not impose independent conditions on ternary quadrics.}

As we have recalled above, imposing the vanishing at the two-fat point means to impose the annihilation of all partial derivatives of first order. In $\mathbb{P}^2$, these are three linear conditions on the space of quadrics. Since we are considering a scheme of two two-fat points, we have six linear conditions to impose on the six-dimensional linear space of ternary quadrics; in this sense, we expect to have no plane cubic passing through two two-fat points. However, since the double line is a conic passing doubly thorough the two two-fat points, we have that the six linear conditions are not independent. We will come back in the next sections on this relation between the conditions imposed by a scheme of fat points and the defectiveness of secant varieties.
\end{example}

Corollary \ref{corTerracini} can be generalized to non-complete linear systems on $X$.

\begin{rmk} Let $D$ be any divisor of an irreducible projective variety $X$. With $|D|$, we indicate the complete linear system defined by $D$. Let $V\subset |D|$ be a linear system. We use the notation:
$$V(m_{1}P_{1}, \ldots , m_{s}P_{s})$$
for the subsystem of divisors of $V$ passing through the fixed points $P_{1}, \ldots , P_{s}$ with multiplicities at least $m_{1}, \ldots , m_{s}$ respectively.
\end{rmk}

When the multiplicities $m_{i}$ are equal to two, for $i=1,\ldots ,s$, since a two-fat point in $\bbP^n$ gives $n+1$ linear conditions, in general, we expect that, if $\dim (X)=n$, then:
$$\expdim (V(2P_{1}, \ldots , 2P_{s}))= \dim (V)- s(n+1).$$
Suppose that $V$ is associated with a morphism $\varphi_{V}:X_{0}\rightarrow \mathbb{P}^ {r}$ (if $\dim (V)=r$), which is an embedding on a dense open set $X_{0}\subset X$. We will consider the variety $\overline{\varphi_{V}(X_{0}) }$.

The problem of computing $\dim (V(2P_{1}, \ldots , 2P_{s}))$ is equivalent to that one of computing the dimension of the $s$-th secant variety to $\overline{\varphi_{V}(X_{0}) }$.

\begin{proposition}\label{Terrsistlin} Let $X$ be an integral scheme and $V$ be a linear system on $X$ such that the rational function $\varphi_{V}:X\dashrightarrow \mathbb{P}^ r$ associated with $V$ is an embedding on a dense open subset $X_{0}$ of X%should it be italics? please check the conventions for the math notations throughout
. Then, $\sigma_{s}\left(\overline{\varphi_{V}(X_{0})}\right)$ is defective if and only if for general points, we have $P_{1}, \ldots , P_{s}\in X$:
$$\dim (V(2P_{1}, \ldots , 2P_{s}))>\min \{-1, r-s(n+1)\} .$$ 
\end{proposition}
\noindent This statement can be reformulated via apolarity, as we will see in the next section.

%%%%%%%%%%%%%%%%%%%%%%%%%%%%%%%%%%%%%%%%%%%%%%%%%%%%%%%%%%%%%%%%%%%%%%%%%%
\subsubsection{Apolarity}\label{Apolaritysection}

This section is an exposition of inverse systems techniques, and it follows \cite{Ge1}.

As already anticipated at the end of the proof of Terracini's lemma, the whole Alexander and Hirshowitz technique to compute the dimensions of secant varieties of Veronese varieties is based on the computation of the dual space to the tangent space to $\sigma_s(\nu_d(\mathbb{P}^n))$ at a generic point. Such a duality is the apolarity action that we are going to define. 

\begin{dfn}[Apolarity action]\label{ApolarityAction}
Let $S=\Bbbk[x_{0},\ldots,x_{n}]$ and
$R=\Bbbk[y_{0},\ldots,y_{n}]$ be polynomial rings and consider the action of $R_1$ on $S_1$ and of $R_1$ on $S_1$ defined by:
$$y_{i} \circ x_{j}= \left(\frac{\partial}{\partial x_{i}}\right)
 (x_{j})=\left\{ \begin{array}{ll}0, & \textrm{if } i\neq j\\1, & \textrm{if } i=j \end{array} \right. ;$$
i.e., we view the polynomials of $R_1$ as ``partial derivative operators'' on $S_1$.
 \end{dfn}
 
Now, we extend this action to the whole rings $R$ and $S$ by linearity and using properties
of differentiation. Hence, we get the {apolarity action:}
\begin{equation*}
\circ: R_{i} \times S_{j} \longrightarrow S_{j-i}
\end{equation*}
where:
$$\bfy^\alpha \circ \bfx^\beta=
\begin{cases}
  \prod_{i=1}^n \frac{(\beta_i)!}{(\beta_i-\alpha_i)!}\bfx^{\beta - \alpha}, & \text{ if }\alpha \leq \beta; \\
  0, & \text{otherwise.} \\
\end{cases}
$$ 
for $\alpha, \beta \in \mathbb{N}^{n+1}$, $\alpha = (\alpha_0,...,\alpha_n)$, $\beta = (\beta_0,\ldots,\beta_n)$, where we use the notation $\alpha
 \leq \beta$ if and only if $a_i\leq b_i$ for all
$i=0,\ldots ,n$, which is equivalent to the condition that $x^\alpha$ divides $x^\beta$ in $S$.
\begin{rmk} Here, are some basic remarks on apolarity action:
\begin{itemize}
\item the apolar action of $R$ on $S$ makes $S$ a (non-finitely generated) $R$-module (but the converse is not true);
 \item the apolar action of $R$ on $S$ lowers the degree; in particular, given $F \in S_d$, the $(i,d-i)$-th catalecticant matrix (see Definition \ref{CatalecticantMatrices}) is the matrix of the following linear map induced by the apolar action
   $$
   	Cat_{i,d-i}(F) : R_i \longrightarrow S_{d-i}, \quad G \mapsto G \circ F;
   $$
 \item the apolarity action induces a non-singular $\Bbbk$-bilinear pairing:
 \begin{equation}\label{pairing}
 R_{j} \times S_{j} \longrightarrow \Bbbk, \quad\forall~ j \in \bbN,
 \end{equation}
that induces two bilinear maps (Let $V \times W \longrightarrow \Bbbk$ be a $\Bbbk$-bilinear parity given by
 $v \times w \longrightarrow v \circ w$. It induces two $\Bbbk$-bilinear maps:
 (1) $\phi : V \longrightarrow Hom_{\Bbbk}(W,\Bbbk)$ such that
  $\phi (v) := \phi_{v}$ and $\phi_{v}(w)=v \circ w$ and $\chi :
  W \longrightarrow Hom_{\Bbbk}(V,\Bbbk)$ such that $\chi (w) :
  = \chi_{w}$ and $\chi_{w}(v)=v \circ w$; (2) $V \times W \longrightarrow \Bbbk$
  is not singular iff for all the bases
   $\{w_{1},\ldots,w_{n}\}$ of $W$, the matrix $(b_{ij}=v_{i} \circ w_{j})$ is invertible.).
\end{itemize}
\end{rmk}
\begin{dfn}
 Let $I$ be a homogeneous ideal of $R$. The \emph{inverse system} $I^{-1}$ of $I$
 is the $R$-submodule of $S$ containing all the elements of $S$, which are annihilated by $I$ via the apolarity action.
\end{dfn}

\begin{rmk}Here are some basic remarks about inverse systems:
\begin{itemize}
\item if $I=(G_{1},\ldots,G_{t})\subset R$ and $F \in S$, then:
 $$F \in I^{-1} \Longleftrightarrow G_{1} \circ F = \cdots = G_{t} \circ
 F = 0,$$ finding all such $F$'s amounts to finding all the polynomial solutions for the differential equations defined by the $G_{i}$'s, so one can notice that determining $I^{-1}$ is equivalent to solving (with
polynomial solutions) a finite set of differential equations;
\item $I^{-1}$ is a graded submodule of $S$, but it is not
necessarily multiplicatively closed; hence in general, $I^{-1}$ is not
an ideal of $S$. 
\end{itemize}
\end{rmk}

Now, we need to recall a few facts on {Hilbert functions} and {Hilbert series}.

Let $X \subset \mathbb{P}^n$ be a closed subscheme whose
defining homogeneous ideal is $I:=I(X) \subset S=\Bbbk[x_0,\ldots,x_n]$. Let $A= S
/I$ be the homogeneous coordinate ring of $X$, and $A_{d}$ will be its degree $d$~component.

\begin{dfn}\label{def:HF}
The \emph{Hilbert function} of the scheme $X$ is the numeric function:
$$\HF(X,\cdot):\mathbb{N}\rightarrow \mathbb{N};$$
$$\HF(X,d)= \dim_{\Bbbk}(A_{d}) = \dim_{\Bbbk}(S_{d})-\dim_{\Bbbk}(I_{d}).$$
The \emph{Hilbert series} of $X$ is the generating power series:
$$
	\HS(X;z) = \sum_{d \in \bbN} \HF(X,d) t^d \in \Bbbk[[z]].
$$
\end{dfn}

In the following, the importance of inverse systems for a particular choice of the ideal $I$ will be given by the following result.

\begin{proposition} The dimension of the part of degree $d$ of the
inverse system of an ideal $I\subset R$ is the Hilbert function of
$R/I$ in degree $d$:
\begin{equation}\label{inv syst e H}
 \dim_\Bbbk(I^{-1})_{d}=\mathrm{codim}_\Bbbk(I_{d})=\HF(R/I,d). 
\end{equation}
\end{proposition}

\begin{rmk}\label{rmk: inverse systems properties}
If $V \times W \rightarrow \Bbbk$ is a non-degenerate bilinear form and 
$U$ is a subspace of $V$, then with $U^{\bot}$, we denote the subspace of
$W$ given by: $$U^{\bot}= \{w \in W \, | \, v \circ w =0 \;
\forall \; v \in U \}.$$ With this definition, we observe that:
\begin{itemize} 
\item if we consider the bilinear map in (\ref{pairing}) and an ideal $I\subset R$, then:
 
$(I^{-1})_{d}\cong I_{d}^{\bot}. \;\;$
\item moreover, if $I$ is a monomial ideal, then:

$I_{d}^{\bot}=\langle \textrm{monomials of $R_{d}$ that are not in
$I_{d}$}\rangle$;
\item for any two ideals $I,J\subset R$: $(I \cap J)^{-1}=I^{-1}+J^{-1}.$
\end{itemize}
\end{rmk}
If $I=\wp_{1}^{m_{1}+1} \cap \cdots \cap \wp_{s}^{m_{s}+1}\subset
R = \Bbbk[y_{0},\ldots,y_{n}]$ is the defining ideal of the scheme of fat points $m_1P_1+\cdots+m_sP_s \in \bbP^n$, where $P_{i}=[p_{i_{0}}:p_{i_{1}}:\ldots:p_{i{n}}]\in \bbP^n$, and if
$L_{P_{i}}=p_{i_{0}}x_{0}+p_{i_{1}}x_{1}+\cdots+p_{i_{n}}x_{n} \in
S=\Bbbk[x_{0},\ldots,x_{n}]$, then:
  $$(I^{-1})_{d}= 
  \begin{cases}
  S_{d}, & \textrm{ for } d \leq \max\{m_{i}\}, \\
 L_{P_{1}}^{d-m_{1}}S_{m_{1}}+\cdots+L_{P_{s}}^{d-m_{s}}S_{m_{s}}, &
 \textrm{ for } d \geq \max\{m_{i}+1\},
 \end{cases}$$
and also: 
\begin{align}\label{Hilb fat points}
\HF(R/I,d) & =\dim_\Bbbk(I^{-1})_{d} = \nonumber \\
& = \begin{cases}
 \dim_\Bbbk S_{d}, & \textrm{ for } d \leq \max\{m_{i}\}
 \\ 
 \dim_\Bbbk \langle L_{P_{1}}^{d-m_{1}}S_{m_{1}},\ldots,L_{P_{s}}^{d-m_{s}}S_{m_{s}}\rangle, &
 \textrm{ for } d \geq \max\{m_{i}+1\} 
 \end{cases}
\end{align}
This last result gives the following link between the Hilbert function of a
set of fat points and ideals generated by sums of powers of linear
forms.

\begin{proposition}\label{prop: inverse system fat points}
Let $I=\wp_{1}^{m_{1}+1} \cap \cdots \cap
\wp_{s}^{m_{s}+1}\subset R = \Bbbk[y_{0},\ldots ,y_{n}]$, then the inverse system $(I^{-1})_{d}\subset S_{d}=\Bbbk[x_{0},\ldots ,x_{n}]_{d}$ is the $d$-th graded part
of the ideal $(L_{P_{1}}^{d-m_{1}},\ldots,L_{P_{s}}^{d-m_{s}})\subset S$,
for $d\geq \max \{ m_{i}+1 , \, i=1,\ldots ,s\}$.
 \end{proposition}

Finally, the link between the big Waring problem (Problem \ref{bwp}) and inverse systems is clear. If in~(\ref{Hilb fat points}), all the $m_{i}$'s are equal to one, the dimension of the vector space $\langle L_{P_{1}}^{d-1}S_{1},\ldots ,L_{P_{s}}^{d-1}S_{1}\rangle$ is at the same time the Hilbert function of the inverse system of a scheme of $s$ double fat points and the rank of the differential of the application $\phi$ defined in (\ref{phi}). 

\begin{proposition}\label{apolaritythm}
 Let $L_{1},\ldots,L_{s}$ be linear forms of $S = \Bbbk[x_{0},\ldots,x_{n}]$ such that:
 $$L_{i}=a_{i_{0}}x_{0}+\cdots+a_{i_{n}}x_{n},$$
and let $P_{1},\ldots,P_{s}\in \mathbb{P}^ n$ such that
 $P_{i}=[a_{i_{0}},\ldots,a_{i_{n}}].$
Let $\wp_{i} \subset R = \Bbbk[y_{0},\ldots,y_{n}]$ be the prime ideal
associated with $P_{i}$, for $i=1,\ldots,s$, and let:
 $$\phi_{s,d}: \underbrace{S_{1}\times\cdots\times S_{1}}_{s} \longrightarrow S_{d}$$
 with
 $\phi_{s,d} (L_{1},\ldots,L_{s})=L_{1}^{d}+\cdots+L_{s}^{d}.$ Then,
 $$\rk(d\phi_{s,d})|_{(L_{1},\ldots,L_{s})}=\dim_\Bbbk\langle L_{1}^{d-1}S_{1},\ldots,L_{s}^{d-1}S_{1}\rangle.$$ Moreover, by \eqref{inv syst e H}, we have:
 $$\dim_\Bbbk (\langle L_{1}^{d-1}S_{1},\ldots,L_{s}^{d-1}S_{1}\rangle)=\HF\left( \frac{R}{\wp_{1}^{2}\cap\cdots\cap\wp_{s}^{2}},d \right).$$
\end{proposition}

Now, it is quite easy to see that: 
$$\langle T_{P_1}X_{n,d}, \ldots ,T_{P_s}X_{n,d} \rangle=\PP\langle L_{1}^{d-1}S_{1},\ldots,L_{s}^{d-1}S_{1}\rangle.$$
Therefore, putting together Terracini's lemma and Proposition \ref{apolaritythm}, if we assume the $L_i$'s (hence, the $P_i$'s) to be generic, we get:
\begin{align}\label{eq: dimensions secants = HF points}
\dim(\sigma_s(&X_{n,d})) = \dim \langle T_{P_1}X_{n,d}, \ldots ,T_{P_s}X_{n,d} \rangle = \nonumber \\
& = \dim_\Bbbk\langle L_{1}^{d-1}S_{1},\ldots,L_{s}^{d-1}S_{1}\rangle-1 = \HF(R/(\wp_1^2 \cap \cdots \cap \wp_s^2),d)-1.
\end{align}
\begin{example}
 Let $P\in \mathbb{P}^ n$, $\wp \subset S$ be its representative prime ideal and $f\in S$. Then, the order of all partial derivatives of $f$ vanishing in $P$ is almost $t$ if and only if
 $ f \in \wp^{t+1}$, i.e., $P$ is a singular point of $V(f)$ of multiplicity greater than or equal to $t+1$. Therefore,
 \begin{equation}\label{Hilbertuno} 
 \HF(S/ \wp^{t},d)=
 \begin{cases}
 	{d+n \choose n} & \text{ if } d < t; \\
 	{t-1+n \choose n} & \text{ if } d \geq t.
 \end{cases}
\end{equation}
It is easy to conclude that one $t$-fat point of $ \mathbb{P}^ n$ has the same Hilbert function of
${t-1+n\choose n}$ generic distinct points of $ \mathbb{P}^ n$. Therefore, $\dim(X_{n,d})=\HF(S/ \wp^{2},d)-1=(n+1)-1$. This reflects the fact that Veronese varieties are never one-defective, or, equivalently, since $X_{n,d} = \sigma_1(X_{n,d})$, that Veronese varieties are never defective: they always have the expected dimension $1\cdot n+1-1$.
\end{example}
\begin{example}\label{example: defectiveness Veronese surface}
Let $P_{1},P_{2}$ be two points of $\mathbb{P}^{2}$, $\wp_{i}\subset S=\Bbbk[x_{0},x_{1},x_{2}]$ their associated prime ideals and $m_{1}=m_{2}=2$, so that $I= \wp_{1}^{2} \cap \wp_{2}^{2}$. Is the Hilbert function of $I$ equal to the Hilbert function of six points of $\mathbb{P}^ 2$ in general position? No; indeed, the Hilbert series of six general points of $\mathbb{P}^ 2$ is $1+3z+6\sum_{i \geq 2}z^i$. This means that $I$ should not contain conics, but this is clearly false because the double line through $P_{1}$ and $P_{2}$ is contained in $I$. By~\eqref{eq: dimensions secants = HF points}, this implies that $\sigma_2(\nu_2(\mathbb{P}^2)) \subset \bbP^5$ is defective, i.e., it is a hypersurface, while it was expected to fill all the ambient space.
\end{example}
 
\begin{rmk}[Fr\"oberg--Iarrobino's conjecture]\label{conj:FI}
	Ideals generated by powers of linear forms are usually called {power ideals}. Besides the connection with fat points and secant varieties, they are related to several areas of algebra, geometry and combinatorics; see \cite{ArPo10:PowerIdeals}. Of particular interest is their Hilbert function and Hilbert series. In \cite{Fro}, Fr\"oberg gave a lexicographic inequality for the Hilbert series of homogeneous ideals in terms of their number of variables, number of generators and their degrees. That is, if $I = (G_1,\ldots,G_s) \subset S=\Bbbk[x_0,\ldots,x_n]$ with $\deg(G_i) = d_i$, for $i = 1,\ldots,s$,
	\begin{equation}\label{eq: Froberg inequality}
		\HS(S/I;z) \succeq_{\rm Lex} \left\lceil \frac{\prod_{i=1}^s (1-z^{d_i})}{(1-z)^{n+1}} \right\rceil,
	\end{equation}
	where $\left\lceil \cdot \right\rceil$ denotes the truncation of the power series at the first non-positive term. Fr\"oberg conjectured that equality holds generically, i.e., it holds on a non-empty Zariski open subset of $\bbP S_{d_1}\times\ldots\times \bbP S_{d_s}$. By semicontinuity, fixing all the numeric parameters $(n;d_1,\ldots,d_s)$, it is enough to exhibit {one} ideal for which the equality holds in order to prove the conjecture for those parameters. In \cite{Ia97III:InvSys} (Main Conjecture 0.6), Iarrobino suggested to look to power ideals and asserted that, except for a list of cases, their Hilbert series coincides with the right-hand-side of \eqref{eq: Froberg inequality}. By \eqref{Hilb fat points}, such a conjecture can be translated as a conjecture on the Hilbert function of schemes of fat points. This is usually referred to as the {Fr\"oberg--Iarrobino conjecture}; for a detailed exposition on this geometric interpretation of Fr\"oberg and Iarrobino's conjectures, we refer to \cite{Ch05:GeomInterpretationFrobIarrobino}. As we will see in the next section, computing the Hilbert series of schemes of fat points is a very difficult and largely open problem.
\end{rmk}
 
Back to our problem of giving the outline of the proof of Alexander and Hirshowitz Theorem (Theorem \ref{corollAH}): Proposition \ref{apolaritythm} clearly shows that the computation of $T_Q(\sigma_s(\nu_d(\mathbb{P}^n)))$ relies on the knowledge of the Hilbert function of schemes of double fat points. Computing the Hilbert function of fat points is in general a very hard problem. In $\mathbb{P}^2$, there is an extremely interesting and still open conjecture (the SHGH conjecture). The interplay with such a conjecture with the secant varieties is strong, and we deserve to spend a few words on that conjecture and related aspects.

\subsection{Fat Points in the Plane and SHGH Conjecture}\label{fat:points:section}
The general problem of determining if a set of generic points $P_1,\ldots ,P_s$ in the plane, each with a structure of $m_i$-fat point, has the expected Hilbert function is still an open one. There is only a conjecture due first to B. Segre in 1961 \cite{Se}, then rephrased by B. Harbourne in 1986 \cite{Har}, A. Gimigliano in 1987~\cite{Gi87}, A. Hirschowitz in 1989 \cite{Hir} and others. It describes how the elements of a sublinear system of a linear system $\mathcal{L}$ formed by all divisors in $\mathcal{L}$ having multiplicity at least $m_i$ at the points $P_1, \ldots , P_s$, look when the linear system does not have the expected dimension, i.e., the sublinear system depends on fewer parameters than expected. For the sake of completeness, we present the different formulations of the same conjecture, but the fact that they are all equivalent is not a trivial fact; see~\cite{Ci01, CHMR, CM, CM2}. 

Our brief presentation is taken from \cite{Ci01, CHMR}, which we suggest as excellent and very instructive deepening on this topic.

Let $X$ be a smooth, irreducible, projective, complex variety of dimension $n$. Let $\mathcal{L}$ be a complete linear system of divisors on $X$. Fix $P_1,\ldots ,P_s$ distinct points on $X$ and $m_1,\ldots ,m_s$ positive integers. We denote by $\mathcal{L}(-\sum_{i=1}^s m_iP_i)$ the sublinear system of $\mathcal{L}$ formed by all divisors in $\mathcal{L}$ having multiplicity at least $m_i$ at $P_i$, $i = 1,\ldots ,s$. Since a point of multiplicity $m$ imposes ${m+n-1 \choose n}$ conditions on the divisors of $\mathcal{L}$, it makes sense to define the {expected dimension} of $\mathcal{L}(-\sum_{i=1}^s m_iP_i)$ as:
$$\expdim\left(\mathcal{L}\left(-\sum_{i=1}^s m_iP_i\right)\right):=\max
\left\{ \dim (\mathcal{L})- \sum_{i=1}^s{m_i+n-1 \choose n},-1\right\}.$$
 If $\mathcal{L}(-\sum_{i=1}^s m_iP_i)$ is a linear system whose dimension is not the expected one, it is said to be a {special linear system}.
Classifying special systems is equivalent to determining the Hilbert function of the zero-dimensional subscheme of $\mathbb{P}^n$ given $s$ general fat points of given multiplicities.

A first reduction of this problem is to consider particular varieties $X$ and linear systems $\mathcal{L}$ on them. From this point of view, the first obvious choice is to take $X = \mathbb{P}^n$ and $\mathcal{L} = \mathcal{L}_{n,d} := |\mathcal{O}_{\mathbb{P}^n }(d)|$, the system of all hypersurfaces of degree $d$ in $\mathbb{P}^n$. In this language, $\mathcal{L}_{n,d}(-\sum_{i=1}^s m_iP_i)$ are the hypersurfaces of degree $d$ in $n+1$ variables passing through $P_1, \ldots , P_s$ with multiplicities $m_1, \ldots , m_s$, respectively.

The {SHGH conjecture} describes how the elements of $\mathcal{L}_{2,d}(-\sum_{i=1}^s m_iP_i)$ look when not having the expected dimension; here are two formulations of this. 

\begin{conj}[Segre, 1961 \cite{Se}]\label{conjSegre}
If $\mathcal{L}_{2,d}(-\sum_{i=1}^s m_iP_i)$ is a special linear system, then there is a fixed double component for all curves through the scheme of fat points defined by $\wp_1^{m_{1}}\cap \cdots \cap \wp_s^{m_{s}}$.
\end{conj}

\begin{conj}[Gimigliano, 1987 \cite{Gi87, Gi}]\label{conjGimi}
Consider $\mathcal{L}_{2,d}(-\sum_{i=1}^s m_iP_i).$ Then, one has the following possibilities:
\begin{enumerate}
\item the system is non-special, and its general member is irreducible; 
\item the system is non-special; its general member is non-reduced, reducible; 
its fixed components are all rational curves, except for at most one (this 
may occur only if the system is zero-dimensional); and the general member 
of its movable part is either irreducible or composed of rational curves in 
a pencil; 
\item the system is non-special of dimension zero and consists of a unique multiple 
elliptic curve; 
\item the system is special, and it has some multiple rational curve as a fixed 
component. 
\end{enumerate}
\end{conj}

This problem is related to the question of what self-intersections occur
for reduced irreducible curves
on the surface $X_s$ obtained by blowing up the projective plane at the $s$ points.
Blowing up the points introduces rational curves (infinitely many 
when $s>8$) of self-intersection $-1$.
Each curve $\calC \subset X_s$ corresponds to a curve $D_\calC\subset \mathbb{P}^2$ of some degree $d$
vanishing to orders $m_i$ at the $s$ points:
$$\mathbb{P}^2 \dasharrow X_s, \quad D_\calC \mapsto \calC, $$
and the self-intersection $\calC^2$ is $d^2-m_1^2-\cdots-m_s^2$ if $D_\calC\in \mathcal{L}_{2,d}(-\sum_{i=1}^s m_iP_i)$. 

\begin{example}
 An example of a curve $D_\calC$ corresponding to a curve $\calC$ such that $\calC^2=-1$ on $X_s$
is the line through two of the points; in this case,
$d=1$, $m_1=m_2=1$ and $m_i=0$ for $i>2$, so we have 
$d^2-m_1^2-\cdots-m_s^2=-1$.
\end{example}
According to the SHGH conjecture, these $(-1)$-curves should be the 
only reduced irreducible curves of negative self-intersection,
but proving that there are no others turns out to be itself very hard
and is still open.

\begin{dfn}\label{sls} 
Let $P_{1},\ldots ,P_{s}$ be $s$ points of $\mathbb{P}^ n$ in general position.
The \emph{expected dimension} of $\mathcal{L}(-\sum_{i=1}^s m_iP_i)$ is: 
$$\expdim\left(\mathcal{L}\left(-\sum_{i=1}^s m_iP_i\right)\right):=\max
\left\{\virtdim\left(\mathcal{L}\left(-\sum_{i=1}^s m_iP_i\right) \right),-1\right\},$$
where:
$$\virtdim\left(\mathcal{L}\left(-\sum_{i=1}^s m_iP_i\right) \right):={n+d\choose d}-1-\sum_{i=1}^s{m_i+n-1\choose n},$$
is the \emph{virtual dimension} of $\mathcal{L}(-\sum_{i=1}^s m_iP_i)$. 
\end{dfn}
Consider the blow-up $\pi: \widetilde{\mathbb{P}}^2 \dasharrow \mathbb{P}^2$ of the plane at the points $P_1, \ldots , P_s$. Let $E_1,\ldots ,E_s$ be the exceptional divisors corresponding to the blown-up points $P_1, \ldots , P_s$, and let $H$ be the pull-back of a general line of $\mathbb{P}^2$ via $\pi$. The strict transform of the system $\mathcal{L} := \mathcal{L}_{2,d}(\sum_{i=1}^sm_iP_i)$ is the system $\widetilde{ \mathcal{L} }= |dH - \sum_{i=1}^sm_iE_i|$. 
Consider two linear systems $\mathcal{L} := \mathcal{L}_{2,d}(\sum_{i=1}^sm_iP_i)$ and $\mathcal{L}' := \mathcal{L}_{2,d}(\sum_{i=1}^sm'_iP_i)$. Their intersection product is defined by using the intersection product of their strict transforms on $\widetilde{\mathbb{P}}^2$,~i.e.,
$$\mathcal{L}\cdot \mathcal{L}'=\widetilde{\mathcal{L}}\cdot\widetilde{\mathcal{L}'}=dd'-\sum_{i=1}^sm_im'_i.$$
Furthermore, consider the anticanonical
class $-K := -K_{\widetilde{\mathbb{P}}^2}$ of $\widetilde{\mathbb{P}}^2$ corresponding to the linear system $\mathcal{L}_{2,d}(-\sum^s_{i=1} P_i$), which, by abusing notation, we also denote by $-K$.
The adjunction formula tells us that the arithmetic genus $p_a(\widetilde{\mathcal{L}})$ of a curve
in $\widetilde{\mathcal{L}}$ is:
$$p_a(\widetilde{\mathcal{L}})=\frac{\mathcal{L}\cdot(\mathcal{L}+K)}{2}+1={d-1 \choose 2}\sum_{i=1}^s{m_i \choose 2},$$
which one defines to be the{ geometric genus }of $\mathcal{L}$, denoted $g_{\mathcal{L}}$.

This is the classical Clebsch formula. Then, Riemann--Roch says that:
\begin{align*}
\dim(\mathcal{L}) = \dim(\widetilde{\mathcal{L}}) &= \mathcal{L}\cdot(\mathcal{L}-K) +h^1(\widetilde{\mathbb{P}}^2,\widetilde{L})-h^2(\widetilde{\mathbb{P}}^2,\widetilde{L}) =\\
& = \mathcal{L}^2 - g_{\mathcal{L}} + 1 + h^1(\mathbb{P}^2,\widetilde{\mathcal{L}}) = \virtdim(\mathcal{L}) + h^1(\mathbb{P}^2,\widetilde{\mathcal{L}})
\end{align*}
because clearly, $h^2(\widetilde{\mathbb{P}}^2, \widetilde{\mathcal{L}}) = 0$. Hence,
$$\mathcal{L}~\hbox{is non-special if and only if } h^0 (\widetilde{\mathbb{P}}^2 , \widetilde{\mathcal{L}}) \cdot h^1 (\widetilde{\mathbb{P}}^2 , \widetilde{\mathcal{L}}) = 0 .$$

Now, we can see how, in this setting, special systems can naturally arise. Let us look for an irreducible curve $\calC$ on $\widetilde{\mathbb{P}}^2$, corresponding to a linear system $\mathcal{L}$ on $\mathbb{P}^2$, which is expected to exist, but, for example, its double is not expected to exist. It translates into the following set of inequalities:
$$
\begin{cases}
\virtdim(\mathcal{L})\geq 0;\\
g_{\mathcal{L}}\geq 0;\\
\virtdim(2\mathcal{L})\leq -1;
\end{cases}
$$
which is equivalent to:
$$
\begin{cases}
\calC^2 - \calC \cdot K\geq 0;\\
\calC^2 + \calC \cdot K\geq -2;\\
2\calC^2 -\calC \cdot K \leq 0;
\end{cases}.$$
and it has the only solution:
$$\calC^2 = \calC \cdot K = -1,$$
which makes all the above inequalities equalities. Accordingly, $\calC$ is a rational curve, i.e., a curve of genus zero, with self-intersection $-1$, i.e., a $(-1)$-curve. A famous theorem of Castelnuovo's (see \cite{bov} (p. 27)) says that these are the only curves that can be contracted to smooth points via a birational morphism of the surface on which they lie to another surface. By abusing terminology, the curve $\Gamma \subset \mathbb{P}^2$ corresponding to $\calC$ is also called a {$(-1)$-curve}.

More generally, one has special linear systems in the following situation. Let $\mathcal{L}$ be a linear system on $\mathbb{P}^2$, which is not empty; let $\calC$ be a $(-1)$-curve on $\mathbb{P}^2$ corresponding to a curve $\Gamma$ on $\mathbb{P}^2$, such that $\widetilde{\mathcal{L}} \cdot \calC = -N < 0$. Then, $\calC$ (respectively, $\Gamma$) splits off with multiplicity $N$ as a fixed component from all curves of $\widetilde{\mathcal{L}}$ (respectively, $\mathcal{L}$), and one has:
$$\widetilde{\mathcal{L}}=N\calC+\widetilde{\mathcal{M}}, \quad {\rm (respectively, } \mathcal{L}=N\Gamma+\mathcal{M}{\rm )},$$
where $\widetilde{\mathcal{M}}$ (respectively, $\mathcal{M}$) is the residual linear system. Then, one computes:
$$\dim(\mathcal{L}) = \dim(\mathcal{M}) \geq \virtdim(\mathcal{M}) = \virtdim(\mathcal{L}) + {N\choose 2},$$
and therefore, if $N\geq 2$, then $\mathcal{L}$ is special.

\begin{example} 
One immediately finds examples of special systems of this type by starting from the $(-1)$-curves of the previous example. For instance, consider $\mathcal{L} := \mathcal{L}_{2,2d}(- \sum_{i=1}^5 dP_i)$, which is not empty, consisting of the conic $\mathcal{L}_{2,2}(\sum^d_{i=1} P_i)$ counted $d$ times, though it has virtual dimension $-{d\choose 2}$.
\end{example}

Even more generally, consider a linear system $\mathcal{L}$ on $\mathbb{P}^2$, which is not empty, $\calC_1,\ldots ,\calC_k$ some $(-1)$-curves on 
$\widetilde{\mathbb{P}}^2$ corresponding to curves $\Gamma_1,\ldots ,\Gamma_k$ on $\mathbb{P}^2$, such that $\widetilde{\mathcal{L}}\cdot \calC_i = -N_i < 0$, $i = 1,\ldots ,k$. Then, for $i = 1,\ldots,k$,
$$ \mathcal{L}=\sum_{i=1}^k N_i \Gamma_i + \mathcal{M}, \quad 
\widetilde{\mathcal{L}} =
\sum^k_{i=1}
N_i \calC_i + \widetilde{\mathcal{M}}, \quad \text{ and } \quad \widetilde{\mathcal{M}}\cdot \calC_i =0.$$ 

As before, $\mathcal{L}$ is special as soon as there is an $i = 1,\ldots ,k$ such that $N_i \geq 2$. Furthermore, $\calC_i \cdot \calC_j =\delta_{i,j}$, because the union of two meeting $(-1)$-curves moves, according to the Riemann--Roch theorem, in a linear system of positive dimension on $\widetilde{\mathbb{P}}^2$, and therefore, it cannot be fixed for $\widetilde{\mathcal{L}}$. In this situation, the reducible curve $\calC := \sum^k_{i=1} \calC_i$ (respectively, $\Gamma := \sum^k_{i=1} N_i\Gamma_i$) is called a $(-1)$-configuration on $\widetilde{\mathbb{P}}^2$ (respectively, on $\mathbb{P}^2$).

\begin{example}
Consider $\mathcal{L} := \mathcal{L}_{2,d}(-m_0P_0 -\sum^ s_{i=1} m_iP_i)$, with $m_0 + m_i = d + N_i$,
$N_i \geq 1$. Let $\Gamma_i$ be the line joining $P_0$, $P_i$. It splits off $N_i$ times from $\mathcal{L}$. Hence: $$\mathcal{L} = \sum^s_{i=1} N_i\Gamma_i + \mathcal{L}_{2,d -\sum_{i=1}^s N_i} \left(-\left(m_0 -\sum^s_{i=1} N_i\right)P_0 -\sum^s_{i=1}\left(m_i -N_i\right)P_i\right).$$ If we require the latter system to have non-negative virtual dimension, e.g., $d\geq\sum_{i=1}^sm_i$, if $m_0 = d$ and some $N_i > 1$, we have as many special systems as we want.
\end{example}

\begin{dfn}
A linear system $\mathcal{L}$ on $\mathbb{P}^2$ is \emph{$(-1)$-reducible} if $\widetilde{\mathcal{L}} = \sum^ k_{i=1} N_i\calC_i +\widetilde{\mathcal{M}}$, where $\calC = \sum^ k_{i=1} \calC_i$ is a $(-1)$-configuration, $\widetilde{\mathcal{M}} \cdot \calC_i = 0$, for all $i = 1,\ldots , k$ and $\virtdim(\mathcal{M}) \geq 0$.
The system $\mathcal{L}$ is called \emph{$(-1)$-special} if, in addition, there is an $i = 1,\ldots , k$ such that $N_i > 1$.
\end{dfn}

\begin{conj}[Harbourne, 1986 \cite{Har}, Hirschowitz, 1989 \cite{Hir}]\label{conjHar} A linear system of plane curves $\mathcal{L}_{2,d}(-\sum_{i=1}^s m_iP_i)$ with general multiple base points is special if and only if it is $(-1)$-special, i.e., it contains some multiple rational curve of self-intersection $-1$ in the base locus. 
\end{conj}

No special system has been discovered except $(-1)$-special systems.

Eventually, we signal a concise version of the conjecture (see \cite{Gi} (Conjecture~3.3)), which involves only a numerical condition. 

\begin{conj}\label{conjGimeasy} 
A linear system of plane curves $\mathcal{L}_{2,d}(-\sum_{i=1}^s m_iP_i)$ with general multiple base points and such that $m_1\geq m_2\geq \ldots \geq m_s \geq 0$ and $d \geq m_1+m_2+m_3$ is always non-special.
\end{conj}

The idea of this conjecture comes from Conjecture \ref{conjHar} and by working on the surface $X =\tilde{\PP^2}$, which is the blow up of $\PP^2$ at the points $P_i$; in this way, the linear system $\mathcal{L}_{2,d}(-\sum_{i=1}^s m_iP_i)$ corresponds to the linear system $\tilde{\mathcal{L}} = dE_0 - m_1E_1 - \ldots-E_s$ on $X$, where $(E_0,E_1,\ldots , E_s)$ is a basis for ${\rm Pic}(X)$, and $E_0$ is the strict transform of a generic line of $\PP ^2$, while the divisors $E_1,\ldots, E_s$ are the exceptional divisors on $P_1,\ldots, P_s$. If we assume that the only special linear systems $\mathcal{L}_{2,d}(-\sum_{i=1}^s m_iP_i)$ are those that contain a fixed multiple ($-1$)-curve, this would be the same for $\tilde{\mathcal{L}}$ in ${\rm Pic}(X)$, but this implies that either we have $m_s < -1$, or we can apply Cremona transforms until the fixed multiple ($-1$)-curve becomes of type $-m'_iE'_i$ in ${\rm Pic}(X)$, where the $E'_i$'s are exceptional divisors in a new basis for ${\rm Pic}(X)$. Our conditions in Conjecture \ref{conjGimeasy} prevent these possibilities, since the $m_i$ are positive and the condition $d \geq m_1+m_2+m_3$ implies that, by applying a Cremona transform, the degree of a divisor with respect to the new basis cannot decrease (it goes from $d$ to $d'=2d-m_i-m_j-m_k$, if the Cremona transform is based on $P_i$, $P_j$ and $P_k$), hence cannot become of degree zero (as $-m'_iE'_i$ would be).

One could hope to address a weaker version of this problem. Nagata, in connection with his 
negative solution of the fourteenth Hilbert problem, made such a conjecture, 

\begin{conj}[Nagata, 1960 \cite{Na}]\label{conjNagata} The linear system $\mathcal{L}_{2,d}(-\sum_{i=1}^s m_iP_i)$ is empty as soon as $s\geq 10$ and $d\leq \sqrt{s}$.
\end{conj}
%\alessandratodo{Se qualcuno (Sandro?) ha voglia di espandere un po' questa parte forse ci starebbe bene: io avevo scritto proprio il minimo sindacale perch\'e visto quanto detto era necessario concludere enunciando la congettura di Nagata, ma \`e anche vero che questo smilzo paragrafetto per provare a giustificare \`e un po' gramo e qualcosa di pi� preciso, anche se breve, ci starebbe bene... secondo me...}
Conjecture \ref{conjNagata} is weaker 
than Conjecture \ref{conjHar}, yet still open for every non-square $n\geq 10$. Nagata's conjecture does not rule out the occurrence of 
curves of self-intersection less than $-1$, but it does rule out the worst of them.
In particular, Nagata's conjecture asserts that $d^2\geq sm^2$ must hold
when $s\geq10$, where $m=(m_1+\cdots+m_s)/s$. Thus, perhaps there are curves
with $d^2-m_1^2-\cdots-m_s^2<0$, such as the $(-1)$-curves mentioned above,
but $d^2-m_1^2-\cdots-m_s^2$ is (conjecturally) only as negative 
as is allowed by the condition that after averaging the multiplicities $m_i$ 
for $n\geq 10$, one must have $d^2-sm^2\geq 0$. 

\bigskip

Now, we want to find a method to study the Hilbert function of a zero-dimensional scheme. One of the most classical methods is the so-called {Horace method} (\cite{AH95}), which has also been extended with the {Horace differential} technique and led J.~Alexander and A.~Hirschowitz to prove Theorem \ref{corollAH}. We explain these methods in Sections \ref{Horacesection} and \ref{Horacediffsection}, respectively, and we resume in Section \ref{Resume AH Thm} the main steps of the Alexander--Hirschowitz theorem.

%%%%%%%%%%%%%%%%%%%%%%%%%%%%%%%%%%%%%%%%%%%%%%%%%%%%%%%%%%%%%%%%%%%%%%%%%%

\subsubsection{La M\'ethode D'Horace}\label{Horacesection}

In this section, we present the so-called {Horace method}. 
It takes this name from the ancient Roman legend (and a play by Corneille: {Horace}, 1639) about the duel between three Roman brothers, the ``Orazi'', and three brothers from the enemy town of Albalonga, the ``Curiazi''. The winners were to have their town take over the other one. After the first clash among them, two of the Orazi died, while the third remained alive and unscathed, while the Curiazi were all wounded, the first one slightly, the second more severely and the third quite badly. There was no way that the survivor of the Orazi could beat the other three, even if they were injured, but the Roman took to his heels, and the three enemies pursued him; while running, they got separated from each other because they were differently injured and they could run at different speeds. The first to reach the Orazio was the healthiest of the Curiazi, who was easily killed. Then, came the other two who were injured, and it was easy for the Orazio to kill them one by one.

This idea of ``killing'' one member at a time was applied to the three elements in the exact sequence of an ideal sheaf (together with the ideals of a residual scheme and a ``trace'') by A. Hirschowitz in \cite{Hir1} (that is why now, we keep the french version ``Horace'' for Orazi) to compute the postulation of multiple points and count how many conditions they impose.%, not all together that is too difficult, but ``~level by level~".

Even if the following definition extends to any scheme of fat points, since it is the case of our interest, we focus on the scheme of two-fat points.
\begin{dfn}
We say that a scheme $Z$ of $r$ two-fat points, defined by the ideal $I_Z$, \emph{imposes independent conditions on} the space of hypersurfaces of degree $d$ in $n+1$ variable $\mathcal{O}_{\mathbb{P}^n}(d)$ if $\mathrm{codim}_\Bbbk\left(I_Z)_d\right)$ in $S^dV$ is $\min\left\{{n+d \choose d} , r(n+1)\right\}$.
\end{dfn}

This definition, together with the considerations of the previous section and \eqref{eq: dimensions secants = HF points} allows us to reformulate the problem of finding the dimension of secant varieties to Veronese varieties in terms of independent conditions imposed by a zero-dimensional scheme of double fat points to forms of a certain~degree.

\begin{corollary}
 The $s$-th secant variety $\sigma_s(X_{n,d})$ of a Veronese variety has the expected dimension if and only if a scheme of $s$ generic two-fat points in $\mathbb{P}^n$ imposes independent conditions on $\mathcal{O}_{\mathbb{P}^n}(d)$.
\end{corollary}

\begin{example}\label{example: defective Veronese quadrics}
The linear system $\mathcal{L} :=\mathcal{ L}_{n,2}(-\sum_{i=1}^s 2P_i)$ is special if $s\leq n$. Actually, quadrics in $\mathbb{P}^n$ singular at $s$ independent points $P_1 , \ldots , P_s$ are cones with the vertex $\mathbb{P}^{s-1}$ spanned by $P_1 , \ldots , P_s$. Therefore, the system is empty as soon as $ s\geq n + 1$, whereas, if $s\leq n$, one easily computes:
$$\dim(\mathcal{L}) = \virtdim(\mathcal{L}) + {s\choose 2} .$$
Therefore, by \eqref{eq: dimensions secants = HF points}, this equality corresponds to the fact that $\sigma_s(\nu_2(\mathbb{P}^n))$ are defective for all $s\leq n$; see Theorem~\ref{corollAH}~(1).
\end{example}

We can now present how Alexander and Hirschowitz used the Horace method in \cite{AH95} to compute the dimensions of the secant varieties of Veronese varieties.

\begin{dfn}
Let $Z \subset \mathbb{P}^n$ be a scheme of two-fat points whose ideal sheaf is $\calI_{Z}$. Let $H\subset \mathbb{P}^n$ be a hyperplane. We define the following: 
\begin{itemize}
	\item the \emph{trace} of $Z$ with respect to $H$ is the scheme-theoretic intersection: $$\Tr_H(Z) := Z \cap H;$$ 
	\item the \emph{residue} of $Z$ with respect to $H$ is the zero-dimensional scheme defined by the ideal sheaf $\calI_{Z} : \mathcal{O}_{\mathbb{P}^n}(-H)$ and denoted $\Res_H(Z)$.
\end{itemize}
\end{dfn}

\begin{example}
Let $Z = 2P_0 \subset \mathbb{P}^n$ be the two-fat point defined by $\wp^2 = (x_1,\ldots,x_n)^2$, and let $H$ be the hyperplane $\{x_n = 0\}$. Then, the residue $\Res_H(Z)\subset \mathbb{P}^n$ is defined by: 
$$I_{\Res_H(Z)} = \wp^2 : (x_n) = (x_1,\ldots,x_n) = \wp,$$
hence, it is a simple point of $\mathbb{P}^n$; the trace $\Tr_H(Z)\subset H \simeq \mathbb{P}^{n-1}$ is defined by: 
$$I_{\Tr_H(Z)} = (x_1,\ldots,x_n)^2 \otimes \Bbbk[x_0,\ldots,x_n] / (x_n) = (\overline{x}_1,\ldots,\overline{x}_{n-1})^2,$$
where the $\overline{x}_i$'s are the coordinate of the $\bbP^{n-1} \simeq H$, i.e., $\Tr_H(Z)$ is a two-fat point in $\bbP^{n-1}$ with support at $P_0 \in H$. 
\end{example}
The idea now is that it is easier to compute the conditions imposed by the residue and by the trace rather than those imposed by the scheme $Z$; in particular, as we are going to explain in the following, this gives us an inductive argument to prove that a scheme $Z$ imposes independent conditions on hypersurfaces of certain degree. In particular, for any $d$, taking the global sections of the restriction exact sequence:
$$0\rightarrow {\mathcal I}_{\Res_{H}(Z)}(d-1)\rightarrow{\mathcal I}_{Z}(d) \rightarrow {\mathcal I}_{\Tr_H(Z)}(d)\rightarrow 0,$$
we obtain the so-called {Castelnuovo exact sequence:}
\begin{equation}
0\rightarrow (I_{\Res_{H}{(Z)}})_{d-1}\rightarrow (I_Z)_d\rightarrow (I_{\Tr_H(Z)})_d,
\end{equation}
from which we get the inequality:

\begin{equation}\label{castelnuovo}
	\dim_{\Bbbk}(I_{Z})_d \leq \dim_{\Bbbk}(I_{\Res_{H}{(Z)}})_{d-1} + \dim_{\Bbbk}(I_{\Tr_H(Z)})_d.
\end{equation}
Let us assume that the supports of $Z$ are $r$ points such that $t$ of them lie on the hyperplane $H$, i.e., $\Res_H(Z)$ is the union of $r-t$ many two-fat points and $t$ simple points in $\bbP^n$ and $\Tr_H(Z)$ is a scheme of $t$ many two-fat points in $\bbP^{n-1}$ i.e., with the notation of linear systems introduced above,
\begin{align*}
	\dim_\Bbbk(I_{\Res_H(Z)})_{d-1}& = \dim \left(\calL_{n,d-1}\left(-2\sum_{i=1}^{r-t}P_i - \sum_{i=r-t+1}^rP_i\right)\right) + 1; \\
	\dim_\Bbbk(I_{\Tr_H(Z)})_d & = \dim \left(\calL_{n-1,d}\left(-2\sum_{i=1}^{t}P_i\right)\right) + 1.
\end{align*} Assuming that:
\begin{enumerate}
\item\label{item1} $\Res_{H}(Z)$ imposes independent conditions on $\mathcal{O}_{\mathbb{P}^n}(d-1)$, i.e.,
$$\dim_\Bbbk(I_{\Res_{H}(Z)})_{d-1} = \max\left\{{d-1+n \choose n} -(r-t)(n+1)-t,0\right\},$$
\item\label{item2} and $\Tr_H(Z)$ imposes independent conditions on $\mathcal{O}_{\mathbb{P}^{n-1}}(d)$, i.e.,
$$\dim_\Bbbk (I_{\Tr_{H}(Z)})_d = \max\left\{{d+n-1 \choose n+1} -tn,0\right\},$$
\end{enumerate}
then, by \eqref{castelnuovo} and since the expected dimension (Definition \ref{sls}) is always a lower bound for the actual dimension, we conclude the following.

\begin{thm}[Brambilla--Ottaviani \cite{bo}]\label{theorem:bo}
Let $Z$ be a union of $r$ many two-fat points in $\mathbb{P}^{n}$, and let $H\subset \mathbb{P}^{n}$ be a hyperplane such that $t$ of the $r$ points of $Z$ have support on $H$.
Assume that $\Tr_H(Z_r)$ imposes independent conditions on $\mathcal{O}_H(d)$ and that $\Res_{H}{{Z_r}}$ imposes independent conditions on $\mathcal{O}_{\mathbb{P}^n}(d-1)$.
If one of the pairs of the following inequalities occurs:
\begin{enumerate}
\item\label{(i)} $tn\le{{d+n-1}\choose{n-1}}$ and $r(n+1)-tn\le{{d+n-1}\choose{n}}$,
\item\label{(ii)} $tn\ge{{d+n-1}\choose{n-1}}$ and $r(n+1)-tn\ge{{d+n-1}\choose{n}}$,
\end{enumerate}
then $Z$ imposes independent conditions on the system $\mathcal{O}_{\mathbb{P}^n}(d)$.
\end{thm}

%\begin{proof}
%Since $Tr_H(Z_r)$ imposes independent conditions on ${\mathcal{O}}_H(d)$
%we know that $\dim I_{Tr_{H}(Z_{r})}(d)=\max\{{{d+n-1}\choose{n-1}}-tn,0\}$;
%on the other hand, since $\tilde{Z}_r$
%imposes independent conditions on ${\mathcal{O}}_{{\mathbb{P}}^n}(d-1)$
%it follows that
%$\dim I_{\tilde{Z}_r}(d-1) =\max\{{{d-1+n}\choose{n}}-(r-u)(n+1)-t,0\}.$

%Then in case (\ref{(i)}), we get $\dim I_{Z_r}(d)\leq {{d+n}\choose{n}}-r(n+1)$,
%while in case (\ref{(ii)}), we get $\dim I_{Z_r}(d)\leq 0$.
%But since $\dim I_{Z_r}(d)$ is always greater or equal than the expected dimension, we conclude.
%\end{proof}

The technique was used by Alexander and Hirschowitz to compute the dimension of the linear system of hypersurfaces with double base points, and hence, the dimension of secant varieties of Veronese varieties is mainly the Horace method, via induction. 

	The regularity of secant varieties can be proven as described above by induction, but non-regularity cannot. Defective cases have to be treated case by case. We have already seen that the case of secant varieties of Veronese surfaces (Example \ref{example: defectiveness Veronese surface}) and of quadrics (Example \ref{example: defective Veronese quadrics}) are defective, so we cannot take them as the first step of the induction.

%\begin{itemize}
%\item{Induction on the degree}.
%\end{itemize}

\label{first:step:induction}Let us start with $\sigma_s(X_{3,3})\subset \mathbb{P}^{19}$. The expected dimension is $4s-1$. Therefore, we expect that $\sigma_5(X_{3,3})$ fills up the ambient space. Now, let $Z$ be a scheme of five many two-fat points in general position in $\bbP^3$ defined by the ideal $I_Z = \wp_1^2\cap\ldots\cap\wp_5^2$. Since the points are in general position, we may assume that they are the five fundamental points of $\bbP^3$ and perform our computations for this explicit set of points. Then, it is easy to check that:
$$
\HF(S/\wp_1^2\cap \cdots \cap \wp_5^2,3) = 19-\dim_\Bbbk(I_Z)_3=19-0.
$$
 Hence, $\sigma_5(X_{3,3}) = \mathbb{P}^{19}$, as expected. This implies that: 
\begin{equation}\label{sigmasnu3p3}
	\dim (\sigma_s(X_{3,3})) \hbox{ is the expected one for all } s\leq 5.
\end{equation} 
Indeed, as a consequence of the following proposition, if the $s$-th secant variety is regular, so it is the $(s-1)$-th secant variety.

\begin{proposition}\label{succdef} Assume that $X$ is $s$-defective and that $\sigma_{s+1}(X)\neq \mathbb{P}^N$. Then, $X$ is also $(s+1)$-defective.
\end {proposition}

\begin{proof} 
Let $\delta_s$ be the $s$-defect of $X$. By assumptions and by Terracini's lemma, if $P_1,\dots,P_s\in X$ are general points, then the span $T_{P_1,\ldots,P_s} := \langle T_{P_1}X,\ldots,T_{P_s}X\rangle$, which is the tangent space at a general point of $\sigma_s(X)$, has projective dimension $\min(N, sn+s-1)-\delta_s$. Hence, adding one general point $P_{s+1}$, the space $T_{P_1,\dots,P_s, P_{s+1}}$, which is the span of
$T_{P_1,\dots,P_s}$ and $T_{P_{s+1}}X$, has dimension at most $\min\{N,sn+s-1\}-\delta_s+n+1$.
This last number is smaller than $N$, while it is clearly smaller than $(s+1)n+s$. Therefore, $X$ is $(s+1)$-defective.
\end{proof}
%
%\begin{itemize}
%\item Induction on the dimension.
%\end{itemize}

\label{sigma834}In order to perform the induction on the dimension, we would need to study the case of $d=4$, $s=8$ in $\mathbb{P}^3$, i.e., $\sigma_8(X_{3,4})$. We need to compute $\HF_Z(4) = \HF(\Bbbk[x_0, \ldots , x_3]/(\wp_1^2 \cap \cdots \cap \wp_8^2), 4)$. In order to use the Horace lemma, we need to know how many points in the support of scheme $Z$ lie on a given hyperplane. The good news is {upper semicontinuity}\label{upper:semicontinuity}, which allows us to specialize points on a hyperplane. In fact, if the specialized scheme has the expected Hilbert function, then also the general scheme has the expected Hilbert function (as before, this argument cannot be used if the specialized scheme does not have the expected Hilbert function: this is the main reason why induction can be used to prove the regularity of secant varieties, but not the defectiveness). In this case, we choose to specialize four points on $H$, i.e., $Z = 2P_1+\ldots+2P_8$ with $P_1,\ldots,P_4 \in H$. Therefore,
\begin{itemize}
	\item $\Res_H(Z)=P_1+\cdots + P_4+ 2P_5+ \cdots + 2P_8 \subset \bbP^3$;
	\item $\Tr_H(Z_8)=2 \widetilde{P_1} + \cdots +2\widetilde{P_4} \subset H$, where $2\widetilde{P_i}$'s are two-fat points in $\bbP^2$
\end{itemize}

Consider Castelunovo Inequality \eqref{castelnuovo}. Four two-fat points in $\mathbb{P}^3$ impose independent conditions to $\mathcal{O}_{\mathbb{P}^3}(3)$ by \eqref{sigmasnu3p3}, then adding four simple general points imposes independent conditions; therefore, $\Res_{H}{{Z}}$ imposes the independent condition on $\mathcal{O}_{\mathbb{P}^3}(3)$. Again, assuming that the supports of $\Tr_H(Z)$ are the fundamental points of $\bbP^2$, we can check that it imposes the independent condition on $\mathcal{O}_{\mathbb{P}^2}(4)$. Therefore, 
\begin{align*}
	\max\left\{{4+3 \choose 3} - 8 \cdot 4,0\right\} & = 3 = \expdim_\Bbbk (I_Z)_4 \leq \dim_\Bbbk (I_Z)_4 \\
	& \leq \dim_\Bbbk (I_{\Res_H(Z)})_3 + \dim_\Bbbk (I_{\Tr_H(Z)})_4 \\
	& = \max\left\{{3+3 \choose 3} - 4\cdot 4 - 4,0\right\} + \max\left\{ {2+4 \choose 2} - 4 \cdot 3,0\right\} \\
	& = \max\{20 - 16 - 4,0\} + \max\{15-12,0\} = 3.
\end{align*}
In conclusion, we have proven that
$$\sigma_s(X_{3,4}) \hbox{ has the expected dimension for any } s\leq 8.$$

Now, this argument cannot be used to study $\sigma_9(X_{3,4})$ because it is one of the defective cases, but we can still use induction on $d$.

\label{fail:Horace}In order to use induction on the degree $d$, we need a starting case, that is the case of cubics. We have done $\mathbb{P}^3$ already; see \eqref{sigmasnu3p3}. Now, $d=3$, $n=4$, $s=7$ corresponds to a defective case. Therefore, we need to start with $d=3$ and $n=5$. We expect that $\sigma_{10}(X_{5,3})$ fills up the ambient space. Let us try to apply the Horace method as above. The hyperplane $H$ is a $\mathbb{P}^4$; one two-fat point in $\mathbb{P}^4$ has degree five, so we can specialize up to seven points on $H$ (in $\mathbb{P}^4$, there are exactly $35=7\times 5$ cubics), but seven two-fat points in $\mathbb{P}^4$ are defective in degree three; in fact, if $Z = 2P_1+\ldots+2P_7\subset \bbP^4$, then $\dim_\Bbbk\left(I_Z\right)_3=1$. Therefore, if we specialize seven two-fat points on a generic hyperplane $H$, we are ``not using all the room that we have at our disposal'', and \eqref{castelnuovo} does not give the correct upper bound. In other words, if we want to get a zero in the trace term of the Castelunovo exact sequence, we have to ``add one more condition on $H$''; but, to do that, we need a more refined version of the Horace method.

%%\begin{question}\label{Que}
%How can we get one single condition among 10 double points in $\mathbb{P}^5$ where each double point imposes 6 conditions? 
%%\end{question}

%%%%%%%%%%%%%%%%%%%%%%%%%%%%%%%%%%%%%%%%%%%%%%%%%%%%%%%%%%%%%%%%%%%%%%%%%%
\subsubsection{La m\'ethode d'Horace Differentielle}\label{Horacediffsection}

The description we are going to give follows the lines of \cite{BCGI2}. %\alessandratodo{Add reference.}

\begin{dfn}
An ideal $I$ in the algebra of formal functions $\Bbbk
[[{x},y]]$, where ${x} = (x_1,\ldots ,x_{n-1})$, is called a \emph{
vertically-graded} (with respect to $y$) ideal if: 
$$
I = I_0 \oplus I_1y \oplus \cdots \oplus I_{m-1}y^{m-1}\oplus (y^m)
$$
where, for $i = 0,\ldots ,m-1$, $I_i\subset \Bbbk [[{x}]]$ is an
ideal.
\end{dfn}
\begin{dfn}
Let $Q$ be a smooth $n$-dimensional integral scheme, and let $D$ be a
smooth irreducible divisor on $Q$. We say that $Z \subset Q$ is a
\emph{vertically-graded subscheme} of $Q$ with \emph{base} $D$ and \emph{support}
$z\in D$, if $Z$ is a zero-dimensional scheme with support at the point
$z$ such that there is a regular system of parameters $({x},y)$
at $z$ such that $y=0$ is a local equation for $D$ and the ideal of
$Z$ in $\widehat {\mathcal O}_{Q,z} \cong \Bbbk [[{x},y]]$ is
vertically graded.
\end{dfn}
\begin{dfn}
Let $Z\in Q$ be a vertically-graded subscheme with base $D$, and let $p\geq 0$
be a fixed integer. We denote by $\Res^{p}_D(Z)\in Q$ and
$\Tr^{p}_D(Z)\in D$ the closed subschemes defined, respectively, by
the ideals sheaves:
$$ {\mathcal I}_{\Res^{p}_D(Z)} := {\mathcal I}_{Z}+
({\mathcal I}_{Z}:{\mathcal I}^{p+1}_D){\mathcal I}^{p}_D, \qquad \text{ and } \qquad
{\mathcal I}_{\Tr^{p}_D(Z),D} := ({\mathcal I}_{Z}:{\mathcal I}^{p}_D)\otimes
{\mathcal O}_D. $$
\end{dfn}

In $\Res^{p}_D(Z)$, we remove from $Z$ the $(p+1)$-th ``slice'' of $Z$, while in
$\Tr^{p}_D(Z)$, we consider only the $(p+1)$-th ``slice''. Notice
that for $p=0$, this recovers the usual trace $\Tr_D(Z)$ and residual schemes $\Res_D(Z)$.

\begin{example}
Let $Z\subset \mathbb{P}^2$ be a three-fat point defined by $\wp^3$, with support at a point $P \in H$ lying on a plane $H \subset \mathbb{P}^3$. We may assume $\wp = (x_1,x_2)$ and $H = \{x_2 = 0\}$. Then, $3P$ is vertically graded with respect to $H$:

\begin{minipage}{0.43\textwidth}
$$
	\wp^3 = (x_1^3) \oplus (x_1^2) x_2 \oplus (x_1) x_2^2 \oplus (x_2^3),
$$
\end{minipage}
\begin{minipage}{0.5\textwidth}
\begin{center}
\begin{tikzpicture}[line cap=round,line join=round,x=1.0cm,y=1.0cm,scale = 0.5]
\clip(-1.5,-1.) rectangle (4.,3.5);
\draw [line width=1.pt,dash pattern=on 1pt off 2pt] (0.,-1.) -- (0.,3.);
\draw [color = black] (-0.5,3.0) node {$x_2$};
\draw [line width=1.pt,dash pattern=on 1pt off 2pt,domain=-2.:3.] plot(\x,{(-0.-0.*\x)/-1.});
\draw [color = black] (3.5,-0.5) node {$x_1$};
\begin{scriptsize}
\draw [fill=black] (1.,0.) circle (7.5pt);
\draw [fill=black] (0.,1.) circle (7.5pt);
\draw [fill=black] (1.,1.) circle (7.5pt);
\draw [fill=black] (0.,2.) circle (7.5pt);
\draw [fill=black] (0.,0.) circle (7.5pt);
\draw [fill=black] (2.,0.) circle (7.5pt);
\end{scriptsize}
\end{tikzpicture}
\end{center}
{\footnotesize Visualization of a three-fat point in $\bbP^2$: each dot corresponds to a generator of the local ring, which is Artinian.}
\end{minipage}

\vspace{8pt}
Now, we compute all residues (white dots) and traces (black dots) as follows:

\vspace{10pt}
\begin{minipage}{0.7\textwidth}
\noindent {Case} $p=0$. \quad $\Res^0_H(P)\subset\bbP^2$ is defined by: 
$$
	\wp^3 : (x_2) = \wp^2 = (x_1^2) \oplus (x_1) x_2 \oplus (x_2^2),
$$
i.e., it is a two-fat point in $\bbP^2$, while $\Tr^0_H(P)$ is defined by:
$$
	\widetilde{\wp}^3 = \wp^3 \otimes \Bbbk[[x_0,x_1,x_2]]/(x_2) = (\overline{x}_1^3),
$$
where $\overline{x}_0,\overline{x}_1$ are the coordinates on $H$, i.e., it is a three-fat point in $\bbP^1$.
\end{minipage}
\begin{minipage}{0.3\textwidth}
\begin{center}
\begin{tikzpicture}[line cap=round,line join=round,x=1.0cm,y=1.0cm,scale = 0.5]
\clip(-1.5,-1.) rectangle (4.,3.5);
\draw [line width=1.pt,dash pattern=on 1pt off 2pt] (0.,-1.) -- (0.,3.);
\draw [color = black] (-0.5,3.0) node {$x_2$};
\draw [line width=1.pt,dash pattern=on 1pt off 2pt,domain=-2.:3.] plot(\x,{(-0.-0.*\x)/-1.});
\draw [color = black] (3.5,-0.5) node {$x_1$};
\begin{scriptsize}
\draw (1.,0.) circle (7.5pt);
\draw [fill=black] (0.,1.) circle (7.5pt);
\draw [fill=black] (1.,1.) circle (7.5pt);
\draw [fill=black] (0.,2.) circle (7.5pt);
\draw (0.,0.) circle (7.5pt);
\draw (2.,0.) circle (7.5pt);
\end{scriptsize}
\end{tikzpicture}
\end{center}
\end{minipage}

\medskip

\begin{minipage}{0.7\textwidth}
\noindent {Case} $p=1$. \quad $\Res^1_H(3P)\subset\bbP^2$ is a zero-dimensional subscheme of $\bbP^2$ of length four given by: 
\begin{align*}
	\wp^3 + (\wp^3 : (x_2^2)) x_2 & = (x_2^2,x_1x_2,x_1^3) = \\
	 & = (x_1^3) \oplus (x_1)x_2 \oplus (x_2)^2;
\end{align*}
roughly speaking, we ``have removed the second slice'' of $3P$; while, $\Tr^1_H(3P)$ is given by:
$$
	(\wp^3 : (x_2)) \otimes \Bbbk[[x_0,x_1,x_2]] / (x_2) = (\overline{x}_1^2).
$$
\end{minipage}
\begin{minipage}{0.3\textwidth}
\begin{center}
\begin{tikzpicture}[line cap=round,line join=round,x=1.0cm,y=1.0cm,scale = 0.5]
\clip(-1.5,-1.) rectangle (4.,3.5);
\draw [line width=1.pt,dash pattern=on 1pt off 2pt] (0.,-1.) -- (0.,3.);
\draw [color = black] (-0.5,3.0) node {$x_2$};
\draw [line width=1.pt,dash pattern=on 1pt off 2pt,domain=-2.:3.] plot(\x,{(-0.-0.*\x)/-1.});
\draw [color = black] (3.5,-0.5) node {$x_1$};
\begin{scriptsize}
\draw (1.,0.) circle (7.5pt);
\draw [fill=black] (0.,1.) circle (7.5pt);
\draw [fill=black] (1.,1.) circle (7.5pt);
\draw [fill=black] (0.,2.) circle (7.5pt);
\draw (0.,0.) circle (7.5pt);
\draw [fill = black] (2.,1.) circle (7.5pt);
\end{scriptsize}
\end{tikzpicture}
\end{center}
\end{minipage}

\vspace{10pt}
\begin{minipage}{0.7\textwidth}
\noindent {Case} $p=2$. \quad $\Res^2_H(P)\subset\bbP^2$ is a zero-dimensional scheme of length five given by: 
\begin{align*}
	\wp^3 + (\wp^3 : (x_2^3)) x_2^2 & = (x_2^2,x_1^2x_2,x_1^3) = \\
	 & = (x_1^3) \oplus (x_1^2)x_2 \oplus (x_2)^2;
\end{align*}
roughly speaking, we ``have removed the third slice'' of $3P$; while $\Tr^2_H(P)$ is given by:
$$
	(\wp^3 : (x_2^2)) \otimes \Bbbk[[x_0,x_1,x_2]] / (x_2) = (\overline{x}_1).
$$
\end{minipage}
\begin{minipage}{0.3\textwidth}
\begin{center}
\begin{tikzpicture}[line cap=round,line join=round,x=1.0cm,y=1.0cm,scale = 0.5]
\clip(-1.5,-1.) rectangle (4.,3.5);
\draw [line width=1.pt,dash pattern=on 1pt off 2pt] (0.,-1.) -- (0.,3.);
\draw [color = black] (-0.5,3.0) node {$x_2$};
\draw [line width=1.pt,dash pattern=on 1pt off 2pt,domain=-2.:3.] plot(\x,{(-0.-0.*\x)/-1.});
\draw [color = black] (3.5,-0.5) node {$x_1$};
\begin{scriptsize}
\draw [fill=black] (1.,2.) circle (7.5pt);
\draw [fill=black] (0.,1.) circle (7.5pt);
\draw [fill=black] (1.,1.) circle (7.5pt);
\draw [fill=black] (0.,2.) circle (7.5pt);
\draw (0.,0.) circle (7.5pt);
\draw [fill=black] (2.,1.) circle (7.5pt);
\end{scriptsize}
\end{tikzpicture}
\end{center}
\end{minipage}
\end{example}
Finally, let $Z_1,\ldots ,Z_r\in Q$ be vertically-graded subschemes with base $D$
and support $z_i$; let $Z=Z_1\cup \cdots \cup Z_r$, and set ${\bf
p}=(p_1,\ldots ,p_r)\in {\Bbb N}^r$. We write:
$$\Tr^{p}_D(Z):= \Tr^{p_1}_D(Z_1)\cup \cdots \cup \Tr^{p_r}_D(Z_r),
\quad \Res^{p}_D(Z):= \Res^{p_1}_D(Z_1)\cup \cdots \cup
\Res^{p_r}_D(Z_r).$$
We are now ready to formulate the {Horace differential lemma}.
\begin{proposition}[{Horace differential lemma}, {\cite[Proposition 9.1]{AH00}}] 
Let $H$ be a hyperplane in $\mathbb{P}^n$, and let
$W\subset \mathbb{P}^ n$ be a zero-dimensional closed subscheme. Let $Y_1,\ldots ,Y_r,\, Z_1,\ldots ,Z_r$ be zero-dimensional irreducible
subschemes of $\mathbb{P}^n$ such that $Y_i\cong Z_i$, $i=1,\ldots ,r$, $Z_i$
has support on $H$ and is vertically graded with base $H$, and the
supports of $Y=Y_1\cup \cdots \cup Y_r$ and $Z=Z_1\cup \cdots \cup Z_r$ are
generic in their respective Hilbert schemes. Let ${\bf
p}=(p_1,\ldots ,p_r)\in {\Bbb N}^r$. Assume:
\begin{enumerate}
\item $H^0({\mathcal I}_{Tr_HW\cup Tr^{p}_H(Z),H}(d))=0$ and
\item $H^0({\mathcal I}_{Res_HW\cup Res^{p}_H(Z)}(d-1))=0$;
\end{enumerate} 
then,
$$ H^0({\mathcal I}_{W\cup Y}\,(d))=0. $$
\end{proposition}

For two-fat points, the latter result can be rephrased as follows.

\begin{proposition}[{Horace differential lemma for two-fat points}]
Let $H\subset \mathbb{P}^n$ be a hyperplane, $P_1 , \ldots , P_r\in \mathbb{P}^n$ be generic points and $W \subset \mathbb{P}^n$ be a zero-dimensional scheme. Let $Z = 2P_1+\cdots + 2P_r\subset \mathbb{P}^n$, and let $Z' = 2P_1'+\ldots+2P_r'$ such that the $P_i'$'s are generic points on $H$. Let $D_{2,H}(P_i') = 2P_i'\cap H$ be zero-dimensional schemes in $\bbP^n$. Hence, let: 
\begin{align*}
	\overline{Z} &= \Res_H(W) + D_{2,H}(P_1') + \ldots + D_{2,H}(P_r') \subset \bbP^n, \quad \text{ and } \\
	\overline{T} &= \Tr_H(W) + P'_1+\ldots+P'_r \subset H \simeq \bbP^{n-1}.
\end{align*}
Then, if the following two conditions are satisfied:
\begin{align*}
\text{degue}: \quad & \dim_\Bbbk(I_{\overline{Z}})_{d-1}=0; \\
\text{dime}: \quad & \dim_\Bbbk(I_{\overline{T}})_d=0,
\end{align*}
then, $\dim_\Bbbk (I_{W+Z})_d = 0$. 
\end{proposition}
\label{win:diff:horace}Now, with this proposition, we can conclude the computation of $\sigma_{10}(X_{5,3})$. Before Section \ref{Horacediffsection}, we were left with the problem of computing the Hilbert function in degree three of a scheme $Z = 2P_1+ \cdots + 2P_{10}$ of ten two-fat points with generic support in $\mathbb{P}^5$: since a two-fat point in $\bbP^5$ imposes six conditions, the expected dimension of $(I_Z)_3$ is zero. In this case, the ``standard'' Horace method fails, since if we specialize seven points on a generic hyperplane, we lose one condition that we miss at the end of the game. We apply the Horace differential method to this situation. Let $P'_1,\ldots,P'_{8}$ be generic points on a hyperplane $H \simeq \bbP^4 \subset \bbP^5$. Consider:
\begin{align*}
	\overline{Z} & = P'_1 + \ldots + P'_7 + D_{2,H}(P_8) + 2P_9 + 2P_{10} \subset \bbP^5, \quad \text{ and }\\
	\overline{T} & = 2P'_1 + \ldots + 2P'_7 + P'_8 \subset H \simeq \bbP^4.
\end{align*}
Now, \textit{{dime}} is satisfied because we have added on the trace exactly the one condition that we were missing. It is not difficult to prove that \textit{{degue}} is also satisfied: quadrics through $\overline{Z}$ are cones with the vertex the line between $P_9$ and $P_{10}$; hence, the dimension of the corresponding linear system equals the dimension of a linear system of quadrics in $\bbP^4$ passing through a scheme of seven simple points and two two-fat points with generic support. Again, such quadrics in $\bbP^4$ are all cones with the vertex the line passing through the support of the two two-fat points: hence, the dimension of the latter linear system equals the dimension of a linear system of quadrics in $\bbP^3$ passing through a set of eight simple points with general support. This has dimension zero, since the quadrics of $\bbP^3$ are ten. In conclusion, we obtain that the Hilbert function in degree three of a scheme of ten two-fat points in $\bbP^5$ with generic support is the expected one, i.e., by \eqref{eq: dimensions secants = HF points}, we conclude that $\sigma_{10}(X_{5,3})$ fills the ambient space.

\subsubsection{Summary of the Proof of the Alexander--Hirshowitz Theorem}\label{Resume AH Thm}

We finally summarize the main steps of the proof of Alexander--Hirshowitz theorem (Theorem~\ref{corollAH}):
\begin{enumerate}
	\item The dimension of $\sigma_s(X_{n,d})$ is equal to the dimension of its tangent space at a general point $Q$;
	\item By Terracini's lemma (Lemma \ref{Terracini}), if $Q$ is general in $\langle P_1, \ldots, P_s \rangle$, with $P_1, \ldots , P_s\in X$ general points, then: $$\dim(T_Q \sigma_s(X))=\dim(\langle T_{P_1}X, \ldots, T_{P_s}X\rangle);$$
	\item By using the apolarity action (see Definition \ref{ApolarityAction}), one can see that: 
	$$\dim(\langle T_{P_1}X_{n,d}, \ldots, T_{P_s}X_{n,d}\rangle)= \HF(R/(\wp_1^2 \cap \cdots \cap \wp_s^2),d)-1,$$
	where $\wp_1^2 \cap \cdots \cap \wp_s^2$ is the ideal defining the scheme of two-fat points supported by $\PP^n$ corresponding to the $P_i$'s via the $d$-th Veronese embedding;
	\item Non-regular cases, i.e., where the Hilbert function of the scheme of two-fat points is not as expected, have to be analyzed case by case; regular cases can be proven by induction:
	\begin{enumerate}
	\item\label{conclusions:item:nonregularity} The list of non-regular cases corresponds to defective Veronese varieties and is very classical; see Section \ref{Secantsection}, page \pageref{history} and \cite{bo} for the list of all papers where all these cases were investigated. We explained a few of them in Examples \ref{VeroSurfDef}, \ref{example: defectiveness Veronese surface} and \ref{example: defective Veronese quadrics};
	\item\label{conclusions:item:regularity} The proof of the list of non-regular cases classically known is complete and can be proven by a double induction procedure on the degree $d$ and on the dimension $n$ (see Theorem \ref{theorem:bo} and Proposition \ref{succdef}):
	\begin{itemize}
		\item The starting step of the induction for the degree is $d=3$ since quadrics are defective (Example \ref{example: defective Veronese quadrics}):
		\begin{itemize}
			\item The first case to study is therefore $X_{3,3}$: in order to prove that $\sigma_5(X_{3,3})= \mathbb{P}^{19}$ as expected (see page \pageref{first:step:induction}), the Horace method (Section \ref{Horacesection}) is introduced. 
		\end{itemize}
		\item The starting step of the induction for the dimension is $n=5$:
		\begin{itemize}
			\item $\sigma_8(X_{3,4})$ has the expected dimension thanks to upper semicontinuity (see page~\pageref{upper:semicontinuity}), so also the smallest secant varieties are regular (page \pageref{sigma834});
			\item $\sigma_9(X_{3,4})$ is defective (\cite{Te3,rzb, Hir});
		 	\item Therefore, one has to start with $\sigma_{10}(X_{3,5})$, which cannot be done with the standard Horace method (see page \pageref{fail:Horace}), while it can be solved (see page \pageref{win:diff:horace}) by using the Horace differential method (Section \ref{Horacediffsection}).
\end{itemize}
\end{itemize}
\end{enumerate}
\end{enumerate}

%%%%%%%%%%%%%%%%%%%%%%%%%%%%%%%%%%%%%%%%%%%%%%%%%%%%%%%%%%%%%%%%%%%%%%%%%%
\subsection{Algorithms for the Symmetric-Rank of a Given Polynomial}\label{sec:algorithms}
%\alessandratodo{qui voi avevate: ``Symmetric rank and symmetric border rank (Stratification of the secant varieties)", %che per me \`e una sottosezione, quindi di nuovo le vostre cose sono sotto \%}
The goal of the second part of this section is to {compute the symmetric-rank of a given symmetric tensor}. Here, we have decided to focus on algorithms rather than entering into the details of their proofs, since most of them, especially the more advanced ones, are very technical and even an idea of the proofs would be too dispersive. We believe that a descriptive presentation is more enlightening on the difference among them, the punchline of each one and their weaknesses, rather than a precise proof.

%%%%%%%%%%%%%%%%%%%%%%%%%%%%%%%%%%%%%%%%%%%%%%%%%%%%%%%%%%%%%%%%%%%%%%%%%%
\subsubsection{On Sylvester's Algorithm}\label{Sylvestersection}

In this section, we present the so-called Sylvester's algorithm (Algorithm \ref{algo: sylvester}). It is classically attributed to Sylvester, since he studied the problem of decomposing a homogeneous polynomial of degree $d$ into two variables as a sum of $d$-th powers of linear forms and solved it completely, obtaining that the decomposition is unique for general polynomials of odd degree. %even though the paper where he described precisely this algorithm is apparently not easy to be found 
The first modern and available formulation of this algorithm is due to Comas and Seiguer; see \cite{CS}.%, anyway this algorithm is known as the Sylvester's algorithm since he studied the problem to decompose a homogeneous polynomial of degree $d$ in two variables as a sum of $d$-th powers of linear forms and solved it completely obtaining that the decomposition is unique for general polynomials of odd degree.

Despite the ``age'' of this algorithm, there are modern scientific areas where it is used to describe very advanced tools; see \cite{bernardi2012algebraic} for the measurements of entanglement in quantum physics. The following description follows~\cite{bgi}.

If $V$ is a two-dimensional vector space, there is a well-known
isomorphism between $\bigwedge^{d-r+1}(S^{d}V)$ and $
S^{d-r+1}(S^{r}V)$; see \cite{Mu}. In terms of projective algebraic varieties, this isomorphism allows us to view the $(d-r+1)$-th Veronese embedding of $\mathbb{P}^r \simeq \bbP S^rV$ as the set
of $(r-1)$-dimensional linear subspaces of $\mathbb{P}^{d}$ that are
$r$-secant to the rational normal curve. The description of this
result, via coordinates, was originally given by Iarrobino and Kanev; see \cite{IaKa}. Here, we follow the description appearing in \cite{AB} (Lemma 2.1). We use the notation $G(k,W)$ for the {Grassmannian of $k$-dimensional linear spaces in a vector space $W$} and the notation $\bbG(k,n)$ for the {Grassmannian of $k$-dimensional linear spaces in~$\bbP^n$}.

\begin{lemma}\label{LemmaAB} \label{nu_d(Pn)intG(n-1,n+d-1)} 
Consider the map
$\phi_{r,d-r+1}:\mathbb{P}^{}(S^rV)\to G(d-r+1,S^dV)$ that sends the projective class of $F \in S^rV$ to the
$(d-r+1)$-dimensional subspace of $S^dV$ made by the multiples of $F$, i.e.,
$$
	\phi_{r,d-r+1}([F]) = F\cdot S^{d-r}V \subset S^dV.
$$ 
Then, the following hold:
\begin{enumerate}
	\item the image of $\phi_{r,d-r+1}$, after the Pl\'' ucker embedding of $G(d-r+1,S^dV)$ inside $\bbP(\bigwedge^{d-r+1}S^dV)$, is the $(d-r+1)$-th Veronese embedding of $\bbP S^rV$;
	\item identifying $G(d-r+1,S^dV)$ with $\bbG(r-1,\bbP S^dV^*)$, the above Veronese variety is the set of linear spaces $r$-secant to a rational normal curve $\calC_d \subset \bbP S^dV^*$.
\end{enumerate}
\end{lemma}

For the proof, we follow the constructive lines of \cite{bgi}, which we keep here, even though we take the proof as it is, since it is short and we believe it is constructive and useful.

\begin{proof} 
Let $\{x_0,x_1\}$ be the variables on $V$. Then, write
$F=u_0x_0^r+u_1x_0^{r-1}x_1+\dots+u_rx_1^r \in S^rV$. A basis of
the subspace of $S^dV$ of forms of the type $FH$ is given by:
\begin{equation}\label{rsecante}
\begin{cases}
	x_0^{d-r} F & = u_0x_0^{d}+ \cdots +u_rx_0^{d-r}x_1^r, \\
	x_0^{d-r}x_1 F & = \quad \quad u_0x_0^{d-1}x_1+\cdots +u_rx_0^{d-r-1}x_1^{r+1}, \\
	\quad \vdots & \quad \quad\quad\quad \ddots \\
	x_1^{d-r} F & = \quad \quad \quad \quad\quad u_0x_0^{r}x_1^{d-r}+\cdots +u_rx_1^d.
\end{cases}
\end{equation}
The coordinates of these elements with respect to the standard monomial basis
$\{x_0^{d},x_0^{d-1}x_1,\dots,x_1^{d}\}$ of
$S^dV$ are thus given by the rows of the following $(r+1) \times (d+1)$ matrix:
$$\left(\begin{array}{cccccccc}
u_0&u_1&\dots&u_r&0&\dots&0&0\\
0&u_0&u_1&\dots&u_r&0&\dots&0\\
\vdots&\ddots&\ddots&\ddots&&\ddots&\ddots&\vdots\\
0&\dots&0&u_0&u_1&\dots&u_r&0\\
0&\dots&0&0&u_0&\dots&u_{r-1}&u_r
\end{array}
\right).$$
The standard Pl\"ucker coordinates of the subspace
$\phi_{r,d-r+1}([F])$ are the maximal minors of this matrix. It is
known (see for example \cite{AP05}) that these minors form a basis of
$\Bbbk[u_0,\dots,u_r]_{d-r+1}$, so that the image of $\phi_{r,d-r+1}([F])$ is indeed a
Veronese variety, which proves (1).

To prove (2), we recall some standard facts from \cite{AP05}. Consider 
homogeneous coordinates $z_0,\dots,z_{d}$ in
$\mathbb{P}^{}(S^dV^*)$, corresponding to the dual basis of the basis 
$\{x_0^{d},x_0^{d-1}x_1,\dots,x_1^{d}\}$. Consider
$\calC_{d}\subset\mathbb{P}^{}(S^dV^*)$, the standard rational normal
curve with respect to these coordinates. Then, the image of $[F]$
by $\phi_{r,d-r+1}$ is precisely the $r$-secant space to $\calC_{d}$
spanned by the divisor on $\calC_{d}$ induced by the zeros of $F$.
This completes the proof of (2).
\end{proof}

The rational normal curve $\calC_{d} \subset \mathbb{P}^d$ is the $d$-th Veronese embedding of $\bbP V \simeq \bbP^1$ inside $\mathbb{P}S^{d}V\simeq\bbP^d$. Hence, a symmetric tensor $F\in S^{d}V$ has symmetric-rank $r$ if and only if $r$ is the minimum integer for which there exists a $\mathbb{P}^{r-1}\simeq\mathbb{P} W\subset \mathbb{P}^ {}S^{d}V$
 such that $F\in \mathbb{P} W$ and $\mathbb{P} W$ is $r$-secant to the rational normal curve $\calC_{d}\subset\mathbb{P}^ {}(S^{d}V)$ in $r$ distinct points. Consider the maps:
\begin{equation}\label{star}
\mathbb{P}^ {}(S^rV) \stackrel{\phi_{r, d-r+1}}{\longrightarrow}\bbG (d-r, \mathbb{P}S^dV)
\stackrel{\alpha_{r, d-r+1}}{\simeq} \bbG (r-1, \mathbb{P}^ {} S^dV^{*}).
\end{equation}
Clearly, we can identify $\mathbb{P}S^dV^{*}$ with $\mathbb{P}S^{d}V$; hence, the
Grassmannian $\mathbb{G}(r-1,\mathbb{P}S^dV^*)$ can be identified with $\mathbb{G} (r-1, \mathbb{P}S^{d}V)$. Now, by Lemma \ref{LemmaAB}, a projective subspace $\mathbb{P}W$ of
$\mathbb{P}S^dV^* \simeq \mathbb{P}S^dV \simeq \mathbb{P}^d$ is $r$-secant to $\calC_{d}\subset \mathbb{P}S^{d}V$ in $r$ distinct points
if and only if it belongs to $\mathrm{Im}(\alpha_{r, d-r+1} \circ
\phi_{r,d-r+1})$ and the preimage of $\mathbb{P}^ {}W$ via $\alpha_{r,d-r+1} \circ \phi_{r,d-r+1}$ is a polynomial with $r$ distinct roots. Therefore, a symmetric tensor $F \in S^{d}V$ has symmetric-rank $r$ if and only if $r$ is the minimum integer for which the following two conditions~hold:
\begin{enumerate}
\item $F$ belongs to some $\mathbb{P}^ {}W \in \mathrm{Im}(\alpha_{r, d-r+1} \circ \phi_{r,d-r+1} )\subset \mathbb{G} (r-1, \mathbb{P}S^{d}V)$,
\item there exists a polynomial $F \in S^rV$ that has $r$ distinct roots and such that
$\alpha_{r, d-r+1} ( \phi_{r,d-r+1}([F]) )=\mathbb{P}^ {}(W)$.
\end{enumerate}
Now, let $\bbP U$ be a $(d-r)$-dimensional linear subspace of $\bbP S^dV$. 
%Fix the natural basis $\Sigma=\{t_{0}^{d}, t_{0}^{d-1}t_{1}, \ldots
%, t_{1}^{d}\}$ in $\Bbbk[t_{0}, t_{1}]_{d}$. Let $\mathbb{P}^ {}(U)$ be a
%$(d-r)$-dimensional projective subspace of $\mathbb{P}^ {}(\Bbbk[t_{0},
%t_{1}]_{d})$. 
The proof of Lemma \ref{LemmaAB} shows that $\mathbb{P}U$ belongs to the image of $\phi_{r,d-r+1}$ if and only if there exist $u_{0}, \ldots , u_{r}\in \Bbbk$ such that $U=\langle F_{1}, \ldots ,
F_{d-r+1}\rangle$, where, with respect to the standard monomial basis $\mathcal{B} = \{x_0^d,x_0^{d-1}x_1,\ldots,x_1^d\}$ of $S^dV$, we have:
$$
\begin{cases}
F_{1} & =(u_{0}, u_{1},\ldots , u_{r}, 0, \ldots ,
0), \\
F_{2} & = (0, u_{0}, u_{1},\ldots , u_{r}, 0, \ldots ,
0), \\
\quad\vdots & \quad \quad \quad \vdots \\
F_{d-r+1} & =(0, \ldots , 0, u_{0}, u_{1},\ldots , u_{r}).
\end{cases}
$$
Let $\mathcal{B}^* = \{z_{0}, \ldots , z_{d}\}$ be the dual basis of $\mathcal{B}$ with respect to the apolar pairing. Therefore, there exists a $W\subset S^{d}V$ such that $\mathbb{P} W =\alpha_{r,d-r+1}(\mathbb{P}U)$ if and only if $W=H_{1}\cap
\cdots \cap H_{d-r+1}$, and the $H_{i}$'s are as~follows:
$$
\begin{cases}
	H_1: & u_0z_{0}+\cdots +u_rz_{r} =0; \\
	H_2 : & \quad \quad \quad u_0z_{1} +\cdots +u_rz_{r+1} = 0; \\
	\quad \vdots & \quad \quad \quad \quad \quad \quad \ddots \\
	H_{d-r+1}: & \quad \quad \quad \quad \quad \quad u_0z_{d-r}+\cdots +u_rz_{d} = 0.
\end{cases}
$$
This is sufficient to conclude that $F \in \mathbb{P}^ {}S^{d}V$ belongs
to an $(r-1)$-dimensional projective subspace of $\mathbb{P}^ {}S^{d}V$
that is in the image of $\alpha_{r, d-r+1} \circ \phi_{r,d-r+1}$
defined in \eqref{star} if and only if there exist $H_{1}, \ldots
, H_{d-r+1}$ hyperplanes in $S^{d}V$ as above, such that $F \in
H_{1}\cap \ldots \cap H_{d-r+1}$. Now, given $F \in S^dV$ with coordinates $(a_{0}, \ldots ,a_{d})$ with respect to the dual basis $\mathcal{B}^*$, we have that $F \in
H_{1}\cap \ldots \cap H_{d-r+1}$ if and only if the following linear
system admits a non-trivial solution in the $u_i$'s
$$
\begin{cases}
u_0a_{0}+\cdots +u_ra_{r} & =0 \\
\quad u_0a_{1} +\cdots +u_ra_{r+1} & =0 \\
 \quad \quad \ddots \\
\quad \quad \quad u_0a_{d-r}+\cdots +u_ra_{d} & =0.
\end{cases}
$$
If $d-r+1<r+1$, this system admits an infinite number of solutions. If $r\leq d/2$, it admits a non-trivial solution if and only if all the maximal $(r+1)$-minors of the following catalecticant matrix (see Definition \ref{CatalecticantMatrices}) vanish:
$$
\begin{pmatrix}
a_{0} & \cdots & a_{r}\\
a_{1} & \cdots & a_{r+1}\\
\vdots && \vdots \\
a_{d-r} & \cdots & a_{d}
\end{pmatrix}.
$$
\begin{rmk} 
The dimension of $\sigma_{r}(\calC_{d})$ is never defective, i.e., it is the minimum between $2r-1$
and $d$. Actually, $\sigma_{r}(\calC_{d})\subsetneq \mathbb{P}S^dV$ if and only if $1\leq r < \left\lceil \frac{d+1}{2} \right\rceil$. Moreover, an element $[F] \in \mathbb{P}S^dV$ belongs to $\sigma_{r}(\calC_{d})$ for $1\leq r < \left\lceil \frac{d+1}{2} \right\rceil$, i.e., $\underline{\rk}_{\mathrm{sym}}(F) = r$, if and only if
%the catalecticant matrix 
$Cat_{r,d-r}(F)$ 
% defined in Definition \ref{catalecticant} 
does not have maximal rank. % (cf. Definition \ref{catalecticant} of catalecticant matrix).
These facts are very classical; see, e.g., \cite{Harris}.
\end{rmk}
Therefore, if we consider the monomial basis $\left\{{d \choose i}^{-1} x_0^ix_1^{d-i} ~|~ i = 0,\ldots,d\right\}$ of $S^dV$ and write $F = \sum_{i=0}^d {d \choose i}^{-1} a_i x_0^ix_1^{d-i}$, then we write the $(i,d-i)$-th catalecticant matrix of $F$ as $Cat_{i,d-i}(F) = \left(a_{h+k}\right)_{\substack{h = 0,\ldots,i \\ k = 0,\ldots,d-i}}.$

\vspace{12pt}
\begin{algorithm}[Sylvester's algorithm]\label{algo: sylvester} The algorithm works as follows.

\medskip
\begin{algorithmic}[1]
	\REQUIRE {A binary form $F = \sum_{i=0}^d a_i x_0^ix_1^{d-i} \in S^dV$.}	
	\ENSURE {A minimal Waring decomposition $F = \sum_{i=1}^r\lambda_i L_i(x_0,x_1)^d$.}	
	\STATE {initialize $r \leftarrow 0$;}
	\IF {${\rm rk}~Cat_{r,d-r}(F)$ is maximal} \STATE {increment $r \leftarrow r+1$;} 
	\ENDIF
	\STATE {compute a basis of $Cat_{r,d-r}(F)$;}
	\STATE {take a random element $G \in S^rV^*$ in the kernel of $Cat_{r,d-r}(F)$;}
	\STATE {compute the roots of $G$: denote them $(\alpha_i,\beta_i)$, for $i = 1,\ldots,r$;}
	\IF {the roots are not distinct} \STATE {go to Step $2$;}
	\ELSE \STATE {compute the vector $\lambda = (\lambda_1,\ldots,\lambda_r) \in \Bbbk^r$ such that: 
	$$
	\begin{pmatrix}
		\alpha_1^d & \cdots & \alpha_r^d \\
		\alpha_1^{d-1}\beta_1 & \cdots & \alpha_r^{d-1}\beta_{r} \\
		\alpha_1^{d-2}\beta_1^2 & \cdots & \alpha_r^{d-2}\beta_{r}^2 \\
		\vdots & \vdots & \vdots \\
		\beta_1^d & \cdots & \beta_{r}^d \\		
	\end{pmatrix}
	\begin{pmatrix}
		\lambda_1 \\
		\lambda_2 \\
		\vdots \\
		\lambda_r
	\end{pmatrix} =
	\begin{pmatrix}
		a_0 \\
		\frac{1}{d}a_1 \\
		{d \choose 2}^{-1}a_2 \\
		\vdots \\
		a_d
	\end{pmatrix}
	$$} \ENDIF
	\STATE {construct the set of linear forms $\{L_i = \alpha_i x_0 + \beta_i x_1\} \subset S^1V$;}
	\RETURN {the expression $\sum_{i=1}^r \lambda_i L_i^d$.}
\end{algorithmic}
\end{algorithm}

\begin{example}\label{exSyl} 
Compute the symmetric-rank and a minimal Waring decomposition of the polynomial $$F = 2x_0^4-4x_0^3x_1+30x_0^2x_1^2-28x_0x_1^3+17x_1^4.$$ 
We follow Sylvester's algorithm. The first catalecticant matrix with rank smaller than the maximal is:
$$
Cat_{2,2}(F)=
\begin{pmatrix}
	2&-1&5\\
	-1&5&-7\\
	5&-7&17
\end{pmatrix},$$
in fact, ${\rm rk}~Cat_{2,2}(F)=2$. Now, let $\{y_0,y_1\}$ the dual basis of $V^*$. We get that $\ker Cat_{2,2}(F)=\langle 2y_0^2-y_0y_1-y_1^2 \rangle $. We factorize: 
$$
	2y_0^2-y_0y_1-y_1^2=(-y_0+y_1)(-2y_0-y_1).
$$ 
Hence, we obtain the roots $\{(1,1),(1,-2)\}$. Then, it is direct to check that:
$$
	\begin{pmatrix}
		1 & 1 \\
		1 & -2 \\
		1 & 4 \\
		1 & -8 \\
		1 & 16 \\				
	\end{pmatrix}
	\begin{pmatrix}
		1 \\
		1
	\end{pmatrix} =
	\begin{pmatrix}
		2 \\
		-1 \\
		5 \\
		-7 \\
		17
	\end{pmatrix},
$$
hence, a minimal Waring decomposition is given by: 
$$
	F = (x_0+x_1)^4+(x_0-2x_1)^4.
$$
\end{example}

The following result was proven by Comas and Seiguer in
 \cite{CS}; see also \cite{bgi}. It describes the structure of the {stratification by symmetric-rank} of
symmetric tensors in $S^{d}V$ with $\dim V=2$. This result allows us to improve the classical Sylvester algorithm (see Algorithm \ref{SSRA}).
%We follow the description of \cite{bgi}.

\begin{thm}\label{curve}
Let $\calC_d = \{[F] \in \mathbb{P} S^{d}V ~|~ \rk_{\mathrm{sym}} (F) =1\} = \{[L^d] ~|~ L \in S^1V\}\subset \mathbb{P}^d$ be the rational normal curve of degree $d$ parametrizing decomposable symmetric tensors. Then,
$$
 \forall \, r,\ 2\leq r\leq \left\lceil{d+1\over 2}\right\rceil, \quad \qquad \sigma_r(\calC_d) \smallsetminus \sigma_{r-1}(\calC_d) = \sigma_{r,r}(\calC_{d})\cup \sigma_{r, d-r+2}(\calC_{d}),
 $$
where we write:
$$
	\sigma_{r,s}(\calC_d) := \{[F] \in \sigma_r(\calC_d) ~|~ \rk_{\mathrm{sym}}(F) = s\} \subset \sigma_r(\calC_d).
$$
\end{thm}

\begin{algorithm}[{Sylvester's symmetric (border) rank algorithm, \cite{bgi}}]
\label{SSRA}

The latter theorem allows us to get a simplified version of the Sylvester
algorithm, which computes the symmetric-rank and the symmetric-border rank of a symmetric tensor, without computing any decomposition. Notice that Sylvester's Algorithm \ref{algo: sylvester} for the rank is recursive: it runs for any $r$ from one to the symmetric-rank of the given polynomial, while Theorem \ref{curve} shows that once the symmetric border rank is computed, then the symmetric-rank is either equal to the symmetric border rank or it is $d-r+2$, and this allows us to skip all the recursive process.

\medskip
\begin{algorithmic}[1]
	\REQUIRE {A form $F \in S^dV$, with $\dim V = 2$.}
	\ENSURE {the symmetric-rank $\rk_{\mathrm{sym}}(F)$ and the symmetric-border rank $\underline{\rk}_{\mathrm{sym}}(F)$.}
	\STATE {$r := {\rm rk}~Cat_{\lfloor \frac{d}{2} \rfloor ,\lceil\frac{d}{2}\rceil}(F)$}
	\STATE {$\underline{\rk}_{\mathrm{sym}}(F) = r$;}
	\STATE {choose an element $G \in \ker Cat_{r,d-r}(F)$;}
	\IF {$G$ has distinct roots} \STATE {$\rk_{\mathrm{sym}}(F) = r$} 
	\ELSE \STATE{$\rk_{\mathrm{sym}}(F) = d-r+2$;} 
	\ENDIF
	\RETURN {$\rk_{\mathrm{sym}}(F)$}
\end{algorithmic}

%\textbf{Input}: The projective class $T$ of a symmetric tensor $t\in S^{d}V$ with $\dim (V)=2$\\
%\textbf{Output}: $rk(t)$.
%\begin{enumerate}
%\item Initialize $r=0$;
%\item\label{oursyl2} Increment $r\leftarrow r+1$;
%\item Compute $M_{d-r,r}(t)$'s $(r+1)$-minors; if they
%are not all equal to zero then go to step (\ref{oursyl2}); else, $T\in
%\sigma_r(C_d)$ (notice that this happens for $r\leq
%\lceil\frac{d+1}{2}\rceil$); go to step (\ref{oursyl4}).
%\item\label{oursyl4} Choose a solution $(\overline{u}_{0}, \ldots ,
%\overline{u}_{d})$ of the system $M_{d-r,r}(t)\cdot
%(u_0,\ldots,u_r)^{t} =0$. If the polynomial
%$\overline{u}_0t_0^d+\overline{u}_1t_0^{d-1}t_1+\cdots+\overline{u}_rt_1^r$
%has distinct roots, then $rk (t) = r$, i.e., $T\in
%\sigma_{r,r}(C_d)$, otherwise $rk (t) = d-r+2$, i.e., $T\in
%\sigma_{r,d-r+2}(C_d)$.
%\end{enumerate}}
\end{algorithm}

%\begin{example} 
%In the Example \ref{exSyl} the polynomial $f$ has both symmetric-rank and symmetric border rank equal to 2.
%\end{example}

\begin{example}\label{exSyl2} 
Let $F=5x_0^5x_1$, and let $\{y_0,y_1\}$ be the dual basis to $\{x_0,x_1\}$. The smallest catalecticant without full rank is: 
$$Cat_{2,3}(F)=
\begin{pmatrix}
0&1&0&0\\
1&0&0&0\\
0&0&0&0
\end{pmatrix},$$
which has rank two. Therefore $[F]\in \sigma_2(C_6)$. Now, $\ker~Cat_{2,3}(F)=\langle y_1^2\rangle$, which has a double root. Hence, $[F]\in \sigma_{2,6}(C_6)$.
\end{example}

\begin{rmk}\label{variables}
When a form $F \in \Bbbk[x_0,\ldots ,x_n]$ can be written using less
variables, i.e., $F \in \Bbbk[L_0,\ldots,L_m]$, for $L_j\in \Bbbk[x_0,\ldots
,x_n]_1$, with $m<n$, we say that $F$ has {$m$ essential variables} (in the literature, it is also said that $F$ is $m$-concise). That is, $F \in S^dW$, where $W = \langle L_0,\ldots,L_m \rangle \subset V$. Then, the rank of $[F]$ with respect to $X_{n,d}$ is the same one as the one with respect to $\nu_d(\mathbb{P}W) \subset X_{n,d}$; e.g., see \cite{LS, LT}. As recently clearly described in \cite[Proposition 10]{Ca} and more classically in \cite{IaKa}, the number of essential variables of $F$ coincides with the rank of the first catalecticant matrix $Cat_{1,d-1}(F)$. In particular, when $[F] \in
\sigma_r(X_{n,d}) \subset \mathbb{P} (S^{d}V)$ with $\dim (V) = n+1$, then, if
$r < n+1$, there is a subspace $W\subset V$ with $\dim (W) = r$ such
that $[F] \in \mathbb{P} S^{d}W$, i.e., $F$ can be written with respect to $r$ variables.
\end{rmk}

Let now $V$ be ($n+1$)-dimensional, and consider the following construction:
\begin{equation}\label{Hilb}
\begin{array}{c c c}
 {\rm Hilb}_r(\mathbb{P}^ n) & \stackrel{\phi}{\dashrightarrow} & \mathbb{G}\left({d+n \choose n}-r,S^dV\right) \\
& & \rotatebox[origin=c]{-90}{$\cong$} \\
{\mathbb G}(r-1,\mathbb{P}S^dV^*) & \longleftarrow & {\mathbb G}\left({d+n \choose n}-r-1,\mathbb{P} S^dV\right)
 \end{array}
 \end{equation}
where the map $\phi$ in \eqref{Hilb} sends a zero-dimensional scheme $Z$ 
with $\deg(Z)=r$ to the vector space $(I_Z)_d$ (it is defined in
the open set formed by the schemes $Z$, which impose independent
conditions to forms of degree $d$) and where the last arrow is the identification, which sends a linear space to its perpendicular. 

As in the case $n=1$, the final image from the latter construction
gives the $(r-1)$-spaces, which are $r$-secant to the Veronese
variety in $\mathbb{P}^ N \cong \mathbb{P} (\Bbbk[x_0,\ldots,x_n]_d)^*$. Moreover, each
such space cuts the image of $Z$ via the Veronese embedding.

\begin{notation}\label{PiZ} {From now on, we will always use the notation $\Pi_{Z}$ to indicate
 the projective linear subspace of dimension $r-1$ in $\mathbb{P} S^{d}V$, with $\dim(V)=n+1$,
 generated by the image of a zero-dimensional scheme $Z\subset \mathbb{P}^ n$ of degree $r$ via the Veronese embedding, i.e., $\Pi_Z = \langle \nu_d(Z) \rangle \subset \bbP S^dV$.}
\end{notation}

\begin{thm}\label{secante2} Any $[F] \in \sigma_{2}(X_{n,d})\subset \mathbb{P} S^dV$,
 with $\dim(V)=n+1$ can only have symmetric-rank equal to $1$, $2$ or $d$. More precisely:
$$\sigma_2(X_{n,d}) \smallsetminus X_{n,d} = \sigma_{2,2}(X_{n,d})\cup \sigma_{2,d}(X_{n,d});$$
more precisely, $\sigma_{2,d}(X_{n,d}) = \tau (X_{n,d})\smallsetminus X_{n,d}$, where $\tau(X_{n,d})$ denotes the tangential variety of $X_{n,d}$, i.e., the Zariski closure of the union of the tangent spaces to $X_{n,d}$.
%Here $\sigma_{2,2}(X_{n,d})$ and $ \sigma_{2,d}(X_{n,d})$ are defined in Notation \ref{sigmasq} and $\tau(X_{n,d})$ is defined in Notation \ref{tau}.
\end{thm}

\begin{proof}
This is actually a quite direct consequence of Remark \ref{variables} and of Theorem \ref{curve}, but let us describe the geometry in some detail, following the proof of \cite{bgi}.
 Since $r=2$, every $Z\in {\rm Hilb}_2(\mathbb{P}^ n)$ is the complete intersection of a line and a quadric, so the structure of $I_Z$ is well known, i.e., $I_Z =
(L_0,\ldots,L_{n-2},Q)$, where $L_i$'s are linearly independent linear forms and $Q$ is a quadric in $S^2V \smallsetminus (L_0,\ldots,L_{n-2})_2$.

If $F \in \sigma_2(X_{n,d})$, then we have two possibilities: either
$\rk_{\mathrm{sym}} (F) =2$
%, i.e., $F \in \sigma_2^0(\nu_2(\mathbb{P}^ n))$), 
or $\rk_{\mathrm{sym}} (T)>2$, i.e., $F$ lies on a tangent line $\Pi_Z$ to the Veronese, which is
given by the image of a scheme $Z \subset \bbP V$ of degree $2$, via the maps
\eqref{Hilb}. We can view $F$ in the projective linear space $H
\cong \mathbb{P}^ d$ in $\mathbb{P} (S^dV)$ generated by the rational normal curve
$\calC_d \subset X_{n,d}$, which is the image of the line $\ell$ defined by
the ideal $(L_0,\ldots,L_{n-2})$ in $\mathbb{P}V$, i.e., $\ell \subset \mathbb{P}^n$ is the unique line containing $Z$. Hence, we can apply Theorem \ref{curve} in order to get that $\rk_{\mathrm{sym}} (F)\leq d$. Moreover, by Remark \ref{variables}, we have $\rk_{\mathrm{sym}}(F)=d$.
\end{proof}

\begin{rmk} 
 Let us check that $\sigma_2(X_{n,d})$ is given by the annihilation of the $(3\times 3)$-minors of the
first two catalecticant matrices, $Cat_{1,d-1}(V)$ and $Cat_{2,d-2}(V)$ (see Definition \ref{CatalecticantMatrices}); actually, such minors are the
generators of $I_{\sigma_2(\nu_d(\mathbb{P}^ n))}$; see \cite{K}.

Following the construction above \eqref{Hilb}, we can
notice that the coefficients of the linear spaces defined by the forms $L_i\in V^*$ in the
ideal $I_Z$ are the solutions of a
linear system whose matrix is given by the catalecticant matrix $Cat_{1,d-1}(V)$;
since the space of solutions has dimension $n-1$, we get
${\rm rk}~Cat_{1,d-1}(V) = 2$. When we consider the quadric $Q$ in $I_Z$, instead,
the analogous construction gives that its coefficients are the
solutions of a linear system defined by the catalecticant matrix $Cat_{2,d-2}(V)$, and
the space of solutions give $Q$ and all the quadrics in
$(L_0,\ldots,L_{n-2})_2$, which are ${n \choose 2}+2n-1$, hence: $${\rm rk}~Cat_{2,d-2}(V) = {n+2 \choose 2}-\left({n \choose 2}+2n\right) = 2.$$
\end{rmk}

Therefore, we can write down an algorithm (Algorithm \ref{sigma2Xndalg}) to test if an element $[F] \in \sigma_{2}(X_{n,d})$ has symmetric rank two or $d$.

\vspace{12pt}
\begin{algorithm}[An algorithm to compute the symmetric-rank of an element lying on ${\sigma_{2}(X_{n,d})}$]
\label{sigma2Xndalg} 
 \begin{algorithmic}[1]
 	\REQUIRE {A from $F \in S^dV$, where $\dim V = n+1$.}
 	\ENSURE {If $[F] \in \sigma_2(X_{n,d})$, returns the $\rk_{\mathrm{sym}}(F)$.}
  	\STATE {compute the number of essential variables $m = {\rm rk}~Cat_{1,d-1}(F)$;}
 	\IF {$m = 1$} \PRINT {$F \in X_{n,d}$;} 
 	\ELSIF {$m > 2$} \PRINT {$F \not\in \sigma_2(X_{n,d})$;} 
 	\ELSE \STATE {let $W = (\ker Cat_{1,d-1}(F))^\perp$ and view $F \in S^dW$;} 
 	\ENDIF
 	\RETURN {apply Algorithm \ref{SSRA} to $F$.}
 \end{algorithmic}
%\\
%\\
%\textbf{Input}: The projective class $T$ of a symmetric tensor $t\in S^{d}V$, with $\dim(V)=n+1$;\\
%\textbf{Output}: $T \notin \sigma_{2}(X_{n,d})$, or $T\in
%\sigma_{2,2}(X_{n,d})$, or $ T \in \sigma_{2,d}(X_{n,d})$, or
%$T\in X_{n,d}$.
%\begin{enumerate}
%\item Consider the homogeneous polynomial associated with $t$ as in (\ref{corrisp}) and rewrite it with the minimum possible number of variables
%(methods are described in \cite{Ca} or \cite{Ol}), if this is 1
%then $T\in X_{n,d}$; if it is $>2$ then $T\notin
%\sigma_{2}(X_{n,d})$, otherwise $T$ can be viewed as a point in
%$\mathbb{P}(S^dW) \cong \mathbb{P}^ d \subset \mathbb{P}( S^{d}V)$, and $\dim (W) =2$,
%so go to step \ref{second}.
%\item\label{second} Apply the
%Algorithm \ref{SSRA} to conclude.
%\end{enumerate}}
\end{algorithm}

\begin{example}\label{exSyil2} 
Compute the symmetric-rank of $$F = x_0^3x_2 + 3x_0^2x_1x_2 + 3x_0x_1^2x_2 + x_1^3x_2.$$ 
First of all, note that $(y_0-y_1) \circ F = 0$; in particular, $\ker Cat_{1,3}(F) = \langle y_0-y_1 \rangle$. Hence, $F$ has two essential variables. This can also be seen by noticing that $F = (x_0+x_1)^3x_2$. Therefore, if we write $z_0 = x_0 + x_1$ and $z_1 = x_2$, then $F = z_0^3z_1 \in \Bbbk[z_0,z_1]$. Hence, we can apply Algorithm \ref{algo: sylvester} and \ref{SSRA} to compute the symmetric-rank, symmetric-border rank and a minimal decompositions of $F$. 
%We observe that in the cases where the factorization of a polynomial is too complicate, we could have proceeded in the following way: consider the following Catalecticant matrix:
%$$M=\left(\begin{array}{cccccc}
%0 & 0& 1/4& 0&1/4 &0\\
%0 & 0 &1/4 & 0& 1/4 & 0\\
%1/4 &1/4 &0& 1/4 &0 &0\\
%0& 0& 1/4& 0& 1/4& 0\\
%1/4&1/4&0 &1/4& 0&0\\
%0&0&0&0&0&0\end{array}\right).
%$$
%This has rank 2 hence $[f]\in\sigma_2(X_{2,4})$; this implies that $f$ can be written by using only 2 variables. 
%
In particular, we write: 
$$
Cat_{2,2}(F)=
\begin{pmatrix}
0&1/4&0\\
1/4&0&0\\
0&0&0
\end{pmatrix},
$$
which has rank two, as expected. Again, as in Example \ref{exSyl2}, the kernel of $Cat_{2,2}(F)$ defines a polynomial with a double root. Hence, $\underline{\rk}_{\mathrm{sym}}(F) = 2$ and $\rk_{\mathrm{sym}}(F) = 4$. If we are interested in finding a minimal decomposition of $F$, we have to consider the first catalecticant whose kernel defines a polynomial with simple roots. In this case, we should get to:
$$
	Cat_{0,4}(F) =
	\begin{pmatrix}
	0\\1/4\\0\\0\\0
	\end{pmatrix},$$
whose kernel is $\langle (1,0,0,0,0),(0,0,1,0,0), (0,0,0,1,0), (0,0,0,0,1) \rangle.$ If we let $\{w_0,w_1\}$ be the variables on $W^*$, we take a polynomial in this kernel, as for example $G = w_0^4+w_0^2w_1^2+w_0w_1^3+w_1^4$. Now, if we compute the roots of $G$, we find four complex distinct roots, i.e.,
$$
\begin{array}{c c}
(\alpha_1,\beta_1) = -\frac{1}{6}A - \frac{1}{2}\sqrt{B+C} & (\alpha_2,\beta_2) = -\frac{1}{6}A + \frac{1}{2}\sqrt{B+C}; \\
(\alpha_3,\beta_3) = \frac{1}{6}A - \frac{1}{2}\sqrt{B-C} & (\alpha_4,\beta_4) = \frac{1}{6}A + \frac{1}{2}\sqrt{B-C};
\end{array}
$$
where:
{\small \begin{align*}
A & =\sqrt{\frac{9\left(\frac{1}{18} i\sqrt{257}\sqrt{3} - \frac{43}{54}\right)^{\frac{2}{3}} - 6 \left(\frac{1}{18} i \sqrt{257} \sqrt{3} - \frac{43}{54}\right)^{\frac{1}{3}} + 13}{\frac{1}{18} i \sqrt{257} \sqrt{3} - \frac{43}{54}}^{\frac{1}{3}}}; \\ \\
B & = 
\sqrt{-\left(1/18 i \sqrt{257} \sqrt{3} - \frac{43}{54}\right)^{1/3} -\left( \frac{\frac{13}{9}}{\frac{1}{18} i \sqrt{257} \sqrt{3} - \frac{43}{54}}\right)^{\frac{1}{3}}}; \\ \\
C & = \frac{6}{\sqrt{\frac{9 \left(\frac{1}{18} i \sqrt{257} \sqrt{3} - \frac{43}{54}\right)^{\frac{2}{3}} - 6 \left(\frac{1}{18} i \sqrt{257} \sqrt{3} - \frac{43}{54}\right)^{\frac{1}{3}} + 13}{\frac{1}{18} i \sqrt{257} \sqrt(3) - \frac{43}{54}}^{\frac{1}{3}}}} - \frac{4}{3}
\end{align*}
}
Hence, if we write $L_i = \alpha_iz_0+\beta_iz_1$, for $i = 1,\ldots,4$, we can find suitable $\lambda_i$'s to write a minimal decomposition $F = \sum_{i=1}^4 \lambda_i L_i^4$. 
Observe that any hyperplane through $[F]$ that does not contain the tangent line to $\calC_4$ at $[z_0^4]$ intersects $C_4$ at four distinct points, so we could have chosen also another point in $\langle (1,0,0,0,0),(0,0,1,0,0), (0,0,0,1,0), (0,0,0,0,1) \rangle$, and we would have found another decomposition of~$F$.
\end{example}

Everything that we have done in this section does not use anything more than Sylvester's algorithm for the two-variable case. In the next sections, we see what can be done if we have to deal with more variables and we cannot reduce to the binary case like in Example \ref{exSyil2}.

Sylvester's algorithm allows us to compute the symmetric-rank of any polynomial in two essential variables. It is mainly based on the fact that equations for secant varieties of rational normal curves are well known and that there are only two possibilities for the symmetric-rank of a given binary polynomial with fixed border rank (Theorem \ref{curve}). Moreover, those two cases are easily recognizable by looking at the multiplicity of the roots of a generic polynomial in the kernel of the catalecticant.

The first ideas that were exploited to generalize Sylvester's result to homogeneous polynomials in more than two variables were:
\begin{itemize}
	\item a good understanding of the inverse system (and therefore, of the scheme defined by the kernels of catalecticant matrices and possible extension of catalecticant matrices, namely Hankel matrices); we will go into the details of this idea in Section \ref{beyondbeyond};
	\item a possible classification of the ranks of polynomials with fixed border rank; we will show the few results in this direction in Section \ref{beyondsection}.
\end{itemize}

%%%%%%%%%%%%%%%%%%%%%%%%%%%%%%%%%%%%%%%%%%%%%%%%%%%%%%%%%%%%%%%%%%%%%%%%%%%%%%%%%%%%%%%%%%

\subsubsection{Beyond Sylvester's Algorithm Using Zero-Dimensional Schemes}\label{beyondsection}

We keep following \cite{bgi}. Let us start by considering the case of a homogeneous polynomial with three essential variables.

If $[F]\in \sigma_3(\nu_d(\mathbb{P}^n))\smallsetminus \sigma_2(\nu_d(\mathbb{P}^n))$, then we will need more than two variables, but actually, three are always sufficient. In fact, if $[F] \in \sigma_3(\nu_d(\mathbb{P}^n))$, then there always exists a zero-dimensional scheme $\nu_d(Z)$ of length three contained in $\nu_d(\mathbb{P}^n)$, whose span contains $[F]$; the scheme $Z\subset \mathbb{P}^n$ itself spans a $\mathbb{P}^2$, which can be seen as $\mathbb{P}((L_1,L_2,L_3)_1)$ with $L_i$'s linear forms. Therefore, $F$ can be written in three variables. The following theorem computes the symmetric-rank of any polynomial in $[F] \in \sigma_3(\nu_d(\mathbb{P}^n))\setminus \sigma_2(\nu_d(\mathbb{P}^n))$, and the idea is to classify the symmetric-rank by looking at the structure of the zero-dimensional scheme of length three, whose linear span contains~$[F]$.

\begin{thm}[{\cite[Theorem 37]{bgi}}]\label{sec3Xd} Let $d\geq 3$, $X_{n,d}\subset \mathbb{P}(\Bbbk^{n+1})$. Then,
$$\sigma_{3}(X_{n,3})\smallsetminus
\sigma_{2}(X_{n,3}) =\sigma_{3,3}(X_{n,3})\cup\sigma_{3,4}(X_{n,3})\cup
\sigma_{3,5}(X_{n,3}),$$ 
while, for $d\geq 4$,
$$\sigma_{3}(X_{n,d})\smallsetminus
\sigma_{2}(X_{n,d})=\sigma_{3,3}(X_{n,d})\cup\sigma_{3,d-1}
(X_{n,d})\cup\sigma_{3,d+1}(X_{n,d})\cup \sigma_{3,2d-1}(X_{n,d}).$$
%Here $\sigma_{b,r}( X_{n,d})$ is as in Notation \ref{sigmasq}.
\end{thm}

We do not give here all the details of the proof since they can be found in \cite{bgi}; they are quite technical, but the main idea is the one described above. We like to stress that the relation between the zero-dimensional scheme of length three spanning $F$ and the one computing the symmetric-rank is in many cases dependent on the following Lemma \ref{1}. Probably, it is classically known, but we were not able to find a precise reference.

\begin{lemma}[{\cite{bgi}} (Lemma 11)]\label{1} Let $Z \subset \mathbb{P}^{n}$, $n\geq 2$,
be a zero-dimensional scheme, with $\deg(Z) \leq 2d+1$. A
necessary and sufficient condition for $Z$ to impose independent
conditions on hypersurfaces of degree $d$ is that no line
$\ell \subset \mathbb{P}^ {n}$ is such that $\deg (Z\cap \ell) \geq d+2$.
\end{lemma}

\begin{rmk}\label{rem to lemma 1} \rm Notice that if $\deg (\ell \cap Z)$ is exactly $d+1+k$, then
the dimension of the space of curves of degree $d$ through them
is increased exactly by $k$ with respect to the generic case.
\end{rmk}

It is easy to see that Lemma \ref{1} can be improved as follows; see \cite{bb}.

\begin{lemma}[\cite{bb}]\label{2} Let $Z \subset \mathbb{P}^ {n}$, $n\geq 2$,
be a zero-dimensional scheme, with $\deg(Z) \leq 2d+1$. If $h^1(\mathbb{P}^n, \mathcal{I}_Z(d))>0$, there exists a unique line
$\ell\subset \mathbb{P}^ {n}$ such that $\deg (Z\cap \ell) = d+1+ h^1(\mathbb{P}^n, \mathcal{I}_Z(d))>0$.
\end{lemma}

We can go back to our problem of finding the symmetric-rank of a given tensor.
The classification of symmetric-ranks of the elements in $\sigma_4(X_{n,d})$ can be treated in an analogous way as we did for $\sigma_3(X_{n,d})$, but unfortunately, it requires a very complicated analysis on the schemes of length four. This is done in \cite{BB13a}, but because of the long procedure, we prefer to not present it here.

It is remarkable that $\sigma_4(X_{n,d})$ is the last $s$-th secant variety of Veronesean, where we can use this technique for the classification of the symmetric-rank with respect to zero-dimensional schemes of length $s$, whose span contains the given polynomial we are dealing with; for $s\geq 5$, there is a more intrinsic problem. In fact, there is a famous counterexample due to Buczy\'nska and Buczy\'ski (see \cite{bbjw}) that shows that, in $\sigma_5(X_{4,3})$, there is at least a polynomial for which there does not exist any zero-dimensional scheme of length five on $X_{4,3}$, whose span contains it. The example is the following.
\begin{example}[Buczy\'nska, Buczy\'nski \cite{bbjw, bbm}]\label{Jarek}
%The following polynomial has border rank $\leq 5$ but it doesn't exist any degree 5 zero-dimensional scheme contained in $\nu_3(\mathbb{P}^4)$ whose span contains it:
One can easily check that the following polynomial: 
$$ 
F= x_{0}^{2} x_{2} + 6 x_{1}^{2} x_{3} -3\, (x_{0}+x_{1})^{2}
x_{4}.
$$
can be obtained as $\lim_{\epsilon\rightarrow 0}{1\over 3\epsilon} F_{\epsilon} =F$ where: 
$$ 
F_{\epsilon}= (x_{0}+\epsilon x_{2})^{3} + 6(x_{1}+\epsilon
x_{3})^{3} -3(x_{0}+x_{1}+ \epsilon x_{4})^{3} + 3(x_{0} + 2\,
x_{1})^{3} - (x_{0} + 3 x_{1})^{3}
$$ 
has symmetric-rank five for $\epsilon >0$. Therefore, $[F]\in \sigma_5(\nu_3(\mathbb{P}^4))$.

An explicit computation of $F^{\perp}$ yields the 
Hilbert series for $\HS_{R/F^{\bot}}(z) = 1+5z+5z^2+z^3$. 
Let us prove, by contradiction, that there is no saturated ideal $I\subset F^{\perp}$ defining a zero-dimensional scheme of length $\le 5$.
Suppose on the contrary that $I$ is such an ideal. Then, $\HF_{R/ I}(i) \ge \HF_{R/ F^{\perp}}(i)$ for all $i \in \mathbb{N}$. As $\HF_{R/ I}(i)$ is
an increasing function of $i\in \mathbb{N}$ with $\HF_{R/F^{\perp}} (i) \le \HF_{R/ I} (i) \le 5$, we deduce that 
$\HS_{R/ I}(t) = 1 + 5\sum_{i=1}^\infty z^i$. 
This shows that $I_{1}=\{0\}$
and that $I_{2} = (F^{\perp})_{2}$. As $I$ is saturated, $I_{2}:
(x_{0}, \ldots, x_{4})=I_{1}=\{0\}$, since $\HF_{R/F^{\perp}} (1) =
5$. However, an explicit computation of
$(F^{\bot})_{2}: (x_{0}, \ldots, x_{4})$ gives
$\langle x_{2},x_{3},x_{4}\rangle$. In this way, we obtain a contradiction, so that
there is no saturated ideal of
degree $\le 5$ such that $I\subset F^{\perp}$.
Consequently, the minimal zero-dimensional scheme contained in $X_{4,3}$ whose linear span contains $[F]$ has degree six.
\end{example}

In the best of our knowledge, the two main results that are nowadays available to treat these ``wild'' cases are the following.

\begin{proposition}[\cite{bgi}]\label{brkremark2} Let $X\subset \mathbb{P}^ N$ be a non-degenerate smooth variety.
Let $H_r$ be the irreducible component of the Hilbert scheme of
zero-dimensional schemes of degree $r$ of $X$ containing $r$ distinct
points, and assume that for each $y\in H_r$, the corresponding
subscheme $Y$ of $X$ imposes independent conditions on linear forms.
Then, for each $P\in \sigma_{r}(X)$ $\smallsetminus \sigma_{r}^{0}(X)$,
there exists a zero-dimensional scheme $Z\subset X$ of degree $r$ such
that $P\in \langle Z\rangle \cong \mathbb{P}^ {r-1}$. Conversely, if there exists $Z\in H_{r}$ such that $P\in \langle Z\rangle $, then $P\in \sigma_{r}(X)$.
\end{proposition}

%\begin{proof} Let us consider the map $\phi: H_r \to \Bbb G (r-1,\mathbb{P}^ N)$, $\phi (y)=\langle Y\rangle$;
%$\phi$ is well defined since $\dim \langle Y \rangle=r-1$ for all $y\in H_r$ by
%assumption. Hence $ \phi (H_r)$ is closed in $\Bbb G (r-1,\mathbb{P}^ N)$.
%
%\par Now let ${\mathcal I }\subset \mathbb{P}^ N \times \Bbb G (r-1,\mathbb{P}^ N)$ be the incidence variety,
% and $p$, $q$ its projections on $\mathbb{P}^ N$, $ \Bbb G (r-1,\mathbb{P}^ N)$
%respectively; then, $A:= pq^{-1} (\phi (H_r))$ is closed in $\mathbb{P}^ N$. Moreover, $A$ is irreducible
%since $H_r$ is irreducible, so $\sigma_{r}^{0}(X)$ is dense in $A$.
%Hence $\sigma_{r}(X)=\overline {\sigma_{r}^{0}(X)}=A$.
%\end{proof}

Obviously, five points on a line do not impose independent conditions on cubics in any $\mathbb{P}^n$ for $n\geq 5$; therefore, this could be one reason why the counterexample given in Example \ref {Jarek} is possible. Another reason is the following.
\begin{proposition}[\cite{bbjw}]
Suppose there exist points $P_1, \ldots ,P_r \in X$ that are
linearly degenerate, that is $\dim \langle P_1, \ldots, P_r \rangle < r -1$. Then, the join of
the $r$ tangent stars {\rm (}see {\cite{BGL} (Section 1.4)} for a definition{\rm )} at these points is contained in $\sigma_r(X)$. In the case that 
$X$ is smooth at $P_1, \ldots P_r$, then $\langle T_{P_1}X,\ldots,T_{P_r}X \rangle \subset \sigma_r(X)$.
\end{proposition} % the order citation of reference is wrong, please confirm it.

%%%%%%%%%%%%%%%%%%%%%%%%%%%%%%%%%%%%%%%%%%%%%%%%%%%%%%%%%%%%%%%%%%%%%%%%%%%%%%%%%%%%%%%%%%
\subsubsection{Beyond Sylvester's Algorithm via Apolarity}\label{beyondbeyond}

We have already defined in Section \ref{Apolaritysection} the apolarity action of $S^\bullet V^* \simeq \Bbbk[y_0,\ldots,y_n]$ on $S^\bullet V \simeq \Bbbk[x_0,\ldots,x_n]$ and inverse systems. Now, we introduce the main algebraic tool from apolarity theory to study ranks and minimal Waring decompositions: that is the apolarity lemma; see \cite{IaKa,Ge}. First, we introduce the {apolar ideal} of a polynomial.
\begin{dfn}
	Let $F \in S^dV$ be a homogeneous polynomial. Then, the \emph{apolar ideal} of $F$ is:
	$$
		F^\perp = \{G \in S^\bullet V^* ~|~ G \circ F = 0\}.
	$$
\end{dfn}
\begin{rmk}
The apolar ideal is a homogeneous ideal. Clearly, $F^\perp_i = S^iV^*$, for any $i > d$, namely $A_F = S^\bullet V^*/F^\perp$ is an {Artinian} algebra with {socle degree} equal to $d$. Since $\dim_\Bbbk (A_F)_d = 1$, then it is also a {Gorenstein} algebra. Actually, Macaulay proved that there exists a one-to-one correspondence between graded Artinian Gorenstein algebras with socle degree $d$ and homogeneous polynomials of degree $d$; for details, see \cite[Theorem 8.7]{Ge}.
\end{rmk}
\begin{rmk}
Note that, directly by the definitions, the non-zero homogeneous parts of the apolar ideal of a homogeneous polynomial $F$ coincide with the kernel of its catalecticant matrices, i.e., for $i = 0,\ldots,d$,
$$
	F^\perp_i = \ker(Cat_{i,d-i}(F)).
$$
\end{rmk}
The apolarity lemma tells us that Waring decompositions of a given polynomial correspond to sets of reduced points whose defining ideal is contained in the apolar ideal of the polynomial.
\begin{lemma}[Apolarity lemma]\label{lemma:Apolarity}
 Let $Z=\{[L_1], \ldots , [L_r]\}\subset \mathbb{P}(S^1V)$, then the following are equivalent:
 \begin{enumerate}
 	\item $F=\sum_{i=1}^r \lambda_i L_i^d$, for some $\lambda_1,\ldots,\lambda_r \in \Bbbk$;
 	\item $I(Z)\subseteq F^{\perp}$.
 \end{enumerate}
 If these conditions hold, we say that $Z$ is a set of points {apolar} to $F$.
\end{lemma}
\begin{proof}
The fact that (1) implies (2) follows from the easy fact that, for any $G \in S^dV^*$, we have that $G \circ L^d$ is equal to $d$ times the evaluation of $G$ at the point $[L] \in \PP V$. Conversely, if $I(Z) \subset F^\perp$, then we have that $F \in I(Z)^\perp_d = \langle L_1^d,\ldots,L_r^d \rangle$; see Remark \ref{rmk: inverse systems properties} and Proposition \ref{prop: inverse system fat points}.
\end{proof}
\begin{rmk}[Yet again: Sylvester's algorithm]
With this lemma, we can rephrase Sylvester's algorithm. Consider the binary form $F =\sum_{i=0}^d c_i {d\choose i} x_0^{d-i}x_1^i$. Such an $F$ can be decomposed as the sum of $r$ distinct powers of linear forms if and only if there exists $Q =q_0y_0^r+q_1y_0^{r-1}y_1+ \cdots + q_ry_1^r$ such that:
\begin{equation}\label{piccolacat}\left( \begin{array}{cccc}
c_0 & c_1 & \cdots & c_r\\
c_1 &&\cdots&c_{r+1}\\
\vdots &&&\vdots\\
c_{d-r}&&\cdots&c_d
\end{array}
\right)\left( \begin{array}{c}
q_0\\
q_1\\
\vdots
\\
q_r
\end{array}
\right)=0,
\end{equation}
and $Q=\mu\Pi_{k=1}^r (\beta_k y_0 - \alpha_k y_1)$, for a suitable scalar $\mu \in \Bbbk$, where $[\alpha_i:\beta_i]$'s are different points in $\PP^1$. In this case, there exists a choice of $\lambda_1,\ldots,\lambda_r$ such that $F =\sum_{k=1}^r \lambda_k (\alpha_k x_0 + \beta_k x_1)^d$. This is possible because of the following remarks:
\begin{itemize}
	\item Gorenstein algebras of codimension two are always complete intersections, i.e., $$I(Z)\subset (F^{\perp})=(G_1,G_2);$$
	\item Artinian Gorenstein rings have a symmetric Hilbert function, hence:
	$$
		\deg(G_1) + \deg(G_2) = \deg(F) + 2, \quad \text{say } \deg(G_1) \leq \deg(G_2);
	$$
	\item If $G_1$ is square-free, i.e., has only distinct roots, we take $Q = G_1$ and $\rk_{\rm sym}(F) = \deg(G_1)$; otherwise, the first degree where we get something square-free has to be the degree of $G_2$; in particular, we can take $Q$ to be a generic element in $F^\perp_{\deg(G_2)}$ and $\rk_{\rm sym}(F) = \deg(G_2)$.
\end{itemize} 
\end{rmk}
%Since the Apolarity Lemma is true for any number of variables, what can we say about a possible relation between $I(Z)$, $(F^{\perp})$ and the kernel of the Catalaecticant matrices? Moreover, which is the best Catalecticant matrix to use in order to control this relation? Obviously, by definition, we have that $I(Z)\subset (F^{\perp})$, but in general the equality $(F^{\perp})=\ker (Cat)$ is not true anymore, or better, it doesn't exist any Catalecticant matrix for which this equality holds. In particular this happens when $\rk(F)\geq$ maximum rank $\rk \, Cat$ that a Catalecticant matrix associated with $F$ can have. Therefore if we are in the case of $\rk(F) \leq \max \{\rk \, Cat\}$ we can still use Catalecticant matrices to construct the apolar ideal $(F^{\perp})$ in order to find the linear forms for a decomposition of $F$ (that is what is shown below by the algorithm that we attribute to Iarrobino and Kanev where the main idea is present despite they do not write the explicit algorithm), but out of that range Catalecticant matrices are not sufficient anymore.
%
%A generalization of Sylvester's algorithm to any number of variables that uses this techniques is given by Iarrobino and Kanev (see \cite{IaKa}) only under the hypothesis $\rk(F)=\max\{\rk \, Cat\}$.
By using Apolarity Theory, we can describe the following algorithm (Algorithm \ref{catalg}).

\medskip
\begin{algorithm}[{Iarrobino and Kanev, \cite{IaKa}.}]\label{catalg} 
We attribute the following generalization of Sylvester's algorithm to any number of variables to Iarrobino and Kanev: despite that they do not explicitly write the algorithm, the main idea is presented in \cite{IaKa}. Sometimes, this algorithm is referred to as the {catalecticant method}.

\medskip
\begin{algorithmic}[1]
\REQUIRE {$F \in S^{d}V$, where $\dim V = n+1$.}
\ENSURE {a minimal Waring decomposition.}
\STATE {construct the most square catalecticant of $F$, i.e., $Cat_{m,d-m}(F)$ for $m= \lceil d/2 \rceil;$}
\STATE {compute $\ker Cat_{m,d-m}(F)$;}
\STATE {if the zero-set $Z$ of the polynomials in $\ker Cat_{m,d-m}(F)$ is a reduced set of points, say $\{[L_1],\ldots,[L_r]\}$, then continue, otherwise the algorithm fails;}
\STATE {solve the linear system defined by $F = \sum_{i=1}^s \lambda_iL_i^d$ in the unknowns $\lambda_i$.}
\end{algorithmic}
\end{algorithm}

\begin{example} Compute a Waring decomposition of: 
$$F=3x^4 + 12x^2y^2 + 2y^4 - 12x^2yz + 12xy^2z - 4y^3z + 12x^2z^2 - 12xyz^2 + 6y^2z^2 - 4yz^3 + 2z^4.$$ The most square catalecticant matrix is: 
$$Cat_{2,2}(F)=\left(\begin{array}{rrrrrr}
3&0&0&2&-1&2\\
0&2&-1&0&1&-1\\
0&-1&2&1&-1&0\\
2&0&1&2&-1&1\\
-1&1&-1&-1&1&-1\\
2&-1&0&1&-1&2
\end{array} \right).$$
Now, compute that the rank of $Cat_{2,2}(F)$ is three, and its kernel is: 
\begin{align*}
	\ker(Cat_{2,2}(F)) & = \langle(1,0,0,-1,-1,-1),~(0,1,0,-1,-2,0),~(0,0,1,0,2,1)\rangle =\\
	& \langle y_0^2-y_1^2-y_1y_2-y_2^2,~y_0y_1-y_1^2-2y_1y_2,~y_0y_2+2y_1y_2+y_2^2 \rangle \subset S^2V^*.
\end{align*} 
It is not difficult to see that these three quadrics define a set of reduced points $\{[1:1:0], [1:0:-1], [1:-1:1]\} \subset \bbP V$. Hence, we take $L_1=x_0+x_1$, $L_2=x_0-x_2$ and $L_3=x_0-x_1+x_2$, and, by the apolarity lemma, the polynomial $F$ is a linear combinations of those forms, in particular, $$F=(x_0+x_1)^4+(x_0-x_2)^4+(x_0-x_1+x_2)^4.$$
\end{example}
Clearly, this method works only if $\rk_{\rm sym}(F) = {\rm rank}~Cat_{m,d-m}(F)$, for $m = \left\lceil\frac{d}{2}\right\rceil$. Unfortunately, in many cases, this condition is not always satisfied.

Algorithm \ref{catalg} has been for a long time the only available method to handle the computation of the decomposition of polynomials with more than two variables. In 2013, there was an interesting contribution due to Oeding and Ottaviani (see \cite{OO}), where the authors used vector bundle techniques introduced in \cite{LO} to find non-classical equations of certain secant varieties. In particular, the very interesting part of the paper \cite{OO} is the use of representation theory, which sheds light on the geometric aspects of this algorithm and relates these techniques to more classical results like the Sylvester pentahedral theorem (the decomposition of cubic polynomial in three variables as the sum of five cubes). For the heaviness of the representation theory background needed to understand that algorithm, we have chosen to not present it here. Moreover, we have to point out that \cite{OO} (Algorithm 4) fails whenever the symmetric-rank of the polynomial is too large compared to the rank of a certain matrix constructed with the techniques introduced in \cite{OO}, similarly as happens for the catalecticant method.
%However, we have to point out that \cite[Algorithm 4]{OO} works with the same idea of the one of Iarrobino-Kanev and only if the symmetric-rank of the polynomial is at most the rank of the most possible square Catalecticant matrix.

Nowadays, one of the best ideas to generalize the method of catalecticant matrices is due to Brachat, Comon, Mourrain and Tsidgaridas, who in \cite{BCMT} developed an algorithm (Algorithm \ref{algoBCMT}) that gets rid of the restrictions imposed by the usage of catalecticant matrices. The idea developed in \cite{BCMT} is to use the so-called {Hankel matrix} that in a way encodes all the information of all the catalecticant matrices. The algorithm presented in \cite{BCMT} to compute a Waring decomposition of a form $F \in S^dV$ passes through the computation of an {affine} Waring decomposition of the dehomogenization $f$ of the given form with respect to a suitable variable. Let $S = \Bbbk[x_1,\ldots,x_n]$ be the polynomial ring in $n$ variables over the field $\Bbbk$ corresponding to such dehomogenization.

We first need to introduce the definition of {Hankel operator} associated with any $\Lambda \in S^*$. To do so, we need to use the structure of $S^*$ as the $S$-module, given by: 
$$
	a * \Lambda : S \rightarrow \Bbbk, \quad b \mapsto \Lambda(ab), \quad \text{ for } a \in S, \Lambda \in S^*.
$$
Then, the {Hankel operator} associated with $\Lambda \in S^*$ is the matrix associated with the linear map:
$$
	H_\Lambda: S \rightarrow S^*, \hbox{ such that } a \mapsto a * \Lambda.
$$
Here are some useful facts about Hankel operators.
\begin{proposition} 
	$\ker (H_\Lambda)$ is an ideal.
\end{proposition}
%\begin{proof}
%The proof is straightforward. 
%\end{proof}

%A real new step on the algorithmic part of the tensor decomposition problem from a theoretical point of view has been given in 2010 by J. Brachat, P. Comon, B. Mourrain, E.P. Tsigaridas in \cite{BCMT}. We present here the main steps of their idea in order to understand their Algorithm \ref{algoBCMT}.

Let $I_\Lambda = \ker(H_\Lambda)$ and $A_\Gamma = S / I_\Lambda$.
 
\begin{proposition}\label{prop:BCMT}
If ${\rm rank}(H_\Lambda)=r < \infty$, then {the algebra $A_\Lambda$ is a $\Bbbk$-vector space of dimension $r$}, and there exist {polynomials $l_1, \ldots l_k$ of degree one} and $g_1, \ldots, g_k$ of degree $d_1,\ldots,d_k$, respectively, in $\Bbbk[\partial_1,\ldots,\partial_n]$ such that: $$\Lambda = \sum_{i=1}^k l_i^{d-d_i}g_i.$$ Moreover, $I_\Lambda$ defines the union of affine schemes $Z_1,\ldots,Z_k$ with support on the points $l^*_1,\ldots,l^*_k \in \Bbbk^n$, respectively, and with multiplicity equal to the dimension of the vector space spanned by the inverse system generated by $l_i^{d-d_i}g_i$.
\end{proposition}

The original proof of this proposition can be found in \cite{BCMT}; for a more detailed and expanded presentation, see \cite{taufer,bertau}.

\begin{thm}[{Brachat, Comon, Mourrain, Tsigaridas} \cite{BCMT}] 
An element $\Lambda \in S^*$ can be decomposed as $\Lambda = \sum_{i=1}^r \lambda_il_i^d$ if and only if ${\rm rank} H_\Lambda = r$, and $I_\Lambda$ is a radical ideal.
\end{thm}
%Now, multiplying by a linear (power of) form means to introduce multiplying operators on $A_f$:
Now, we consider the multiplication operators in $A_\Lambda$. Given $a \in A_\Lambda$:
$$M_a: A_\Lambda \rightarrow A_\Lambda,$$
$$b\mapsto a\cdot b,$$
and, 
$$M_a^t: A_\Lambda^* \rightarrow A^*_\Lambda,$$
$$\gamma \mapsto a * \gamma.$$
Now, 
\begin{equation}\label{Maf}
H_{a*\Lambda} := M_a^t \cdot H_\Lambda.
\end{equation}

\begin{thm} If $\dim A_\Lambda < \infty$, then, $\Lambda = \sum_{i=1}^k l_i^{d-d_i}g_i$ and: 
\begin{itemize}
\item the eigenvalues of the operators $M_a$ and $M_a^t$ are given by $\{a(l^*_1), \ldots , a(l^*_r)\}$;
\item the common eigenvectors of the operators $(M^t_{x_i})_{1\leq i \leq n}$ are, up to scalar, the~$l_i$'s.
\end{itemize}
\end{thm}
Therefore, one can recover the $l_i$'s, i.e., the points $l^*_i$'s, by eigenvector computations: take $B$ as a basis of $A_f$, i.e., say $B = \{b_1,\ldots,b_r\}$ with $r={\rm rank} H_\Lambda$, and let $H^B_{a * \Lambda}=M_a^tH_\Lambda^B=H_\Lambda^BM_a$ ($M_a$ is the matrix of the multiplication by $a$ in the basis $B$). The common solutions of the generalized eigenvalue problem: 
$$(H_{a*\Lambda}-\lambda H_\Lambda)v=0,$$
for all $a\in S$ yield the common eigenvectors $H^B_\Lambda v$ of $M_a^t$, that is the evaluations at the points $l_i^*$'s. Therefore, these common eigenvectors $H^B_\Lambda v$ are up to scalar the vectors $[b_i(l^*_i), \ldots , b_r(l^*_i)]$, for $i = 1,\ldots,r$.

If $f=\sum_{i=1}^r \lambda_il_i^d$, then the $Z_i$'s in Proposition \ref{prop:BCMT} are simple, and one eigenvector computation is enough: in particular, for any $a\in S$, $M_a$ is diagonalizable, and the generalized eigenvectors $H^B_\Lambda v$ are, up to scalar, the evaluations at the points $l_i^*$'s.

Now, in order to apply this algebraic tool to our problem of finding a Waring decomposition of a homogeneous polynomial $F \in \Bbbk[x_0,\ldots,x_n]$, we need to consider its dehomogenization $f = F(1,x_1,\ldots,x_n)$ with respect to the variable $x_0$ (with no loss of generality, we may assume that the coefficients with respect to $x_0$ are all non-zero). Then, we associate a {truncated Hankel matrices} as~follows.
\begin{dfn}
Let $B$ be a subset of monomials in $S$. We say that $B$ is
\emph{connected to one} if $\forall \, m\in B$ either $m=1$ or there exists $i\in
\{1,\ldots,n\}$ and $m'\in B$ such that $m=x_{i} m'$.
\end{dfn}
Let $B,B' \subset S_{\leq d}$ be sets of monomials of degree $\leq d$, connected to one.
For any $f = \sum_{\substack{\alpha \in \bbN^{n} \\ |\alpha| \leq d}} c_\alpha {d \choose d-|\alpha|,\alpha_1,\ldots,\alpha_n} \bfx^\alpha \in S_d$, we consider the \emph{Hankel matrix:}
$$ 
{H}^{B,B'}_{f} = (h_{\alpha+\beta})_{\alpha \in B, \beta\in B'},
$$
where $h_{\alpha} = c_{\alpha}$ if
$|\alpha|\le d$, and otherwise, $h_{\alpha}$ is an unknown.
The set of all these new variables is denoted ${h}$. Note that, by this definition, the known parts correspond to the catalecticant matrices of $F$. For simplicity, we write $H^B_f = H^{B,B}_f$. This matrix is also called \emph{quasi-Hankel} \cite{mp-jcomplexity-2000}.

\begin{example}
	Consider $F = -4x_0x_1 + 2x_0x_2 +2x_1x_2 + x_2^2 \in \Bbbk[x_0,x_1,x_2]$. Then, we look at the dehomogenization with respect to $x_0$ given by $f = -4x_1+2x_2+2x_1x_2 + x_2^2 \in \Bbbk[x_1,x_2]$. Then, if we consider the standard monomial basis of $S_{\leq 2}$ given by $B = \{1,x_1,x_2,x_1^2,x_1x_2,x_2^2\}$, then we get:
	$$
		H_f^B = \begin{pmatrix}
			0 & -2 & 1 & 0 & 1 & 1 \\
			-2 & 0 & 1 & h_{(3,0)} & h_{(2,1)} & h_{(1,2)} \\
			1 & 1 & 1 & h_{(2,1)} & h_{(1,2)} & h_{(0,3)} \\
			0 & h_{(3,0)} & h_{(2,1)} & h_{(4,0)} & h_{(3,1)} & h_{(2,2)} \\
			1 & h_{(2,1)} & h_{(1,2)} & h_{(3,1)} & h_{(2,2)} & h_{(1,3)} \\
			1 & h_{(1,2)} & h_{(0,3)} & h_{(2,2)} & h_{(1,3)} & h_{(0,4)} \\
		\end{pmatrix},
	$$
	where the $h$'s are unknowns.
\end{example}
Now, the idea of the algorithm is to find a suitable polynomial $\overline{f}$ whose Hankel matrix extends the one of $f$, has rank equal to the Waring rank of $f$ and the kernel gives a radical ideal. This is done by finding suitable values for the unknown part of the Hankel matrix of $f$. Those $\overline{f}$ are elements whose homogenization is in the following set: 
$$\mathcal{E}^{d,0}_r:=
\left\{[F]\in \mathbb{P}(S^dV) \,|\, \substack{ \scalebox{1}{$\exists L \in S^1V\smallsetminus \{ 0\}, \exists {F'} \in Y^{m,m'}_r$ s.t. $L^{m+m'-d} F' = F$} \\
\scalebox{1}{with $m=\max\{ r, \lceil d/2\rceil\}, m'=\max\{ r-1, \lfloor d/2 \rfloor\}$}}
\right\}$$
where $Y^{i,d-i}_r=\{[F]\in \mathbb{P}(S^dV) \, | \, {\rm rank} Cat_{i,d-i}(F) \leq r\}$.
If $[F] \in \mathcal{E}^{d,0}_r$, we say that $f$ {is the generalized affine decomposition of size $r$}.

\smallskip
Suppose that ${H}^{B,B'}_{f}$ is invertible in $\Bbbk({h})$, then we
define the formal multiplication operators: 
$$
{M}_{i}^{B,B'}({h}) := ({H}^{B,B'}_{f})^{-1} {H}^{B,B'}_{x_{i}f}.
$$
\begin{notation}
	If $B$ is a subset of monomials, then we write $B^+ = B \cup x_1B \cup \ldots \cup x_nB$. Note that, if $B$ is connected to on,e then also $B^+$ is connected to one.
\end{notation}
The key result for the algorithm is the following.
\begin{thm}[{Brachat, Comon, Mourrain, Tsigaridas} \cite{BCMT}] 
If $B$ and $B'$ are sets of monomials connected to one, the coefficients of $f$ are known on $B^+\times B'^+$, and if $H^{B,B'}_{\tilde f}$ is invertible, then $f$ extends uniquely to $S$ if and only if: 
$$M_i^{B,B'}\cdot M_j^{B,B'}=M_j^{B,B'}\cdot M_i^{B,B'}, \quad \text{ for any }1 \leq i < j \leq n.$$
\end{thm}
\begin{algorithm}[{Brachat, Comon, Mourrain, Tsigaridas, \cite{BCMT, taufer, bertau}.}]\label{algoBCMT}
Here is the idea of algorithm presented in \cite{BCMT}. In \cite{taufer,bertau}, a faster and more accurate version can be found.

\medskip
\begin{algorithmic}[1]
\REQUIRE {Any polynomial $f \in S$.}
\ENSURE {An affine Waring decomposition of $f$.}
\STATE {$r \leftarrow 1$;}
\STATE {Compute a set $B$ of monomials of degree $\leq d$ connected to one and with $|B| = r$;}
\STATE {Find parameters ${h}$ such that $\det (H_f^B)\neq 0$ and the operators $M^B_i=(H^B_f)^{-1}H^B_{x_{i}f}$ commute;}
\IF {there is no solution} \STATE {go back to {\small 2} with $r \leftarrow r+1$;} 
\ELSE \STATE {compute the $n\cdot r$ eigenvalues $z_{i,j}$ and the eigenvectors $v_j$ such that $M_jv_j=z_{i,j}v_j$, \\$i=1, \ldots,n$, $j=1, \ldots ,r$, until one finds $r$ different common eigenvectors;}
\ENDIF
\STATE {Solve the linear system $f=\sum_{j=1}^r\lambda_jz_j^d$ in the $\lambda_i$'s, where the $z_j$'s are\\ the eigenvectors found above.}
\end{algorithmic}
\end{algorithm}
\vspace{12pt}

For simplicity, we give the example chosen by the authors of \cite{BCMT}.

\begin{example} 

We look for a decomposition of:
\begin{align*}
F = & -1549440%please check the conventions in these equations and matrices throughout; use commas to separate tens of thousands and greater
\,x_0x_1{x_2}^{3} +2417040\,x_0{x_1}^{2}{x_2}^{2}
  +166320\,{x_0}^{2}x_1{x_2}^{2}-829440\,x_0{x_1}^{3}x_2 \\
 & -5760\,{x_0}^{3}x_1x_2-222480\,{x_0}^{2}{x_1}^{2}x_2
  +38\,{x_0}^{5}-497664\,{x_1}^{5}-1107804\,{x_2}^{5} \\
  & -120\,{x_0}^{4}x_1+180\,{x_0}^{4}x_{{2}}+12720\,{x_0}^{3}{x_1}^{2}
  +8220\,{x_0}^{3}{x_2}^{2}-34560\,{x_0}^{2}{x_1}^{3} \\
  & -59160\,{x_0}^{2}{x_2}^{3}+831840\,x_0{x_1}^{4}+442590\,x_0{x_2}^{4}-5591520\,{x_1}^{4}x_2 \\
  & +7983360\,{x_1}^{3}{x_2}^{2}-9653040\,{x_1}^{2}{x_2}^{3}+5116680\,x_1{x_2}^{4}.
\end{align*}
\begin{enumerate}
\item We form a ${n+d-1 \choose d} \times {n+d-1 \choose d}$ matrix,
 the rows and the columns of which correspond {to the coefficients
 of the polynomial with respect to the expression $f = F(1,x_1,\ldots,x_n) = \sum_{\substack{\alpha \in \bbN^n \\ |\alpha| \leq d}} c_\alpha {d \choose d-|\alpha|,\alpha_1,\ldots,\alpha_n} \bfx^\alpha$.}
 
The whole $21 \times 21$ matrix is the following.
  \begin{center}
%  \begin{sideways}
  $
  { \scriptsize
   \left(
   \begin {array}{c|rrrrrrrrrr} 
     & 1 & x_1 & x_2 & x_1^2 & x_1 x_2 & x_2^2 & x_1^3 & x_1^2x_2 & x_1 x_2^2 & x_2^3\\ \hline
     1 & 38&-24&36&1272&-288&822&-3456&-7416&5544&-5916\\
     x_1 & -24&1272&-288&-3456&-7416&5544& 166368&-41472&80568&-77472\\
     x_2 & 36&-288&822&-7416&5544&-5916& -41472&80568&-77472&88518\\
     x_1^2 & 1272&-3456&-7416&166368&-41472&80568& -497664&-1118304&798336&-965304\\
     x_1 x_2 & -288&-7416&5544&-41472&80568&-77472& -1118304&798336&-965304&1023336\\
     x_2^2 & 822&5544&-5916&80568&-77472&88518&798336&-965304&1023336&-1107804\\
     x_1^3 & -3456&166368&-41472&-497664&-1118304&798336& h_{{6,0,0}}&h_{{5,1,0}}&h_{{4,2,0}}&h_{{3,3,0}}\\
     x_1^2 x_2 & -7416&-41472&80568&-1118304&798336&-965304& h_{{5,1,0}}&h_{{4,2,0}}&h_{{3,3,0}}&h_{{2,4,0}}\\
     x_1 x_2^2 & 5544&80568&-77472&798336&-965304&1023336& h_{{4,2,0}}&h_{{3,3,0}}&h_{{2,4,0}}&h_{{1,5,0}}\\
     x_2^3 & -5916&-77472&88518&-965304&1023336&-1107804& h_{{3,3,0}}&h_{{2,4,0}}&h_{{1,5,0}}&h_{{0,6,0}}
   \end {array}\right)
   }$
%	\end{sideways}
	\end{center}
  Notice that we do not know the elements in some positions of the matrix. 
  {In this case,} we do not know the elements that correspond to monomials with (total) 
  degree higher than five.

\item We extract a principal minor of full rank.
 
 We should re-arrange the rows and the columns of the matrix so that there
 is a principal minor of full rank. We call this minor $\Delta_0$. In
 order to do that, we try to put the matrix in row echelon form, using
 elementary row and column operations.
 
  In our example, the $4 \times 4$ principal minor is of full rank, 
  so there is no need for re-arranging the matrix.
  The matrix $\Delta_0$ is: 
  \begin{displaymath}
   \Delta_0 = 
   \left(
   \begin {array}{rrrr} 
     38&-24&36&1272\\\noalign{}
     -24&1272&-288&-3456\\\noalign{}
     36&-288&822&-7416\\\noalign{}
     1272&-3456&-7416&166368
   \end {array} 
   \right)
  \end{displaymath}
  Notice that the columns of the matrix correspond to the set of monomials
  $\{ 1, x_1, x_2, x_1^2\}$.
 
 \item We compute the ``shifted'' matrix $\Delta_1 = x_1 \Delta_0$.

 The columns of $\Delta_0$ correspond to the set of some monomials, 
 say $\set{ \bvec{x}^{\boldsymbol{\alpha}}}$, where $\boldsymbol{\alpha} \subset \mathbb{N}^{n}$.
 The columns of $\Delta_1$ correspond to the set of monomials
 $\set{x_1 \, \bvec{x}^{\boldsymbol{\alpha}}}$.
 
  The shifted matrix $\Delta_1$ is: 
  \begin{displaymath}
   \Delta_1 = 
   \left(
   \begin {array}{rrrr} 
     -24&1272&-288&-3456\\\noalign{}
     1272&-3456&-7416&166368\\\noalign{}
     -288&-7416&5544&-41472\\\noalign{}
     -3456&166368&-41472&-497664
   \end {array} 
   \right) .
  \end{displaymath}
  Notice that the columns correspond to the monomials
  $\{ x_1, x_1^2, x_1 x_2, x_1^3\}$, 
  which are just the corresponding monomials of the columns of $\Delta_0$,
  i.e., $\{ 1, x_1, x_2, x_1^2\}$, multiplied by $x_1$.
 
 In this specific case, all the elements of the matrices $\Delta_0$ and
 $\Delta_1$ are known. If this is not the case, then we can compute the unknown entries
 of the matrix, using either necessary or sufficient conditions of the quotient algebra,
 e.g., it holds that the ${M}_{x_i}{M}_{x_j} - {M}_{x_j} {M}_{x_i} = 0$, for any $i, j \in \{1, \dots, n\}$.
 % There are other algorithms to extend a moment matrix, e.g., \cite{laurent-ima-2008,laurent-ams-2005,cf-hjm-1991}.
 
\item We solve the equation $(\Delta_1 - \lambda \Delta_0) X = 0$.

 We solve the generalized eigenvalue/eigenvector problem \cite{gvl-book-1996}.
 We normalize the elements of the eigenvectors so that the first element is one, 
 and we read the solutions from the coordinates of the normalized eigenvectors.
 
  The normalized eigenvectors of the generalized eigenvalue problem are:
  \begin{displaymath}
   \left( \begin {array}{r} 
     1 \\ -12\\ -3\\ 144
   \end {array} \right) ,
   \left( \begin {array}{r} 
     1\\ 12\\ -13\\ 144
   \end {array} \right) , 
   \left( \begin {array}{r} 
     1\\ -2\\ 3\\ 4
   \end {array} \right) , 
   \left( \begin {array}{r} 
     1\\ 2\\ 3\\ 4
   \end {array}
   \right).
  \end{displaymath}

  The coordinates of the eigenvectors correspond to the elements of the monomial basis
  $\{ 1, x_1, x_2, x_1^2\}$.
  Thus, we can recover the coefficients of $x_1$ and $x_2$ 
  in the decomposition from the coordinates of the eigenvectors.

  Recall that the coefficients of $x_0$ are considered to be one because of the dehomogenization process. Thus, our polynomial admits a decomposition:
  \begin{align*}
   F = \lambda_1 ( x_0 -12 x_1 -3 x_2)^5 & +
   \lambda_2 ( x_0 + 12 x_1 -13 x_2)^5 +\\
   & \lambda_3 ( x_0 -2 x_1 +3 x_2)^5 +
   \lambda_4 ( x_0 + 2 x_1 + 3 x_2)^5.
  \end{align*}
  It remains to compute $\lambda_i$'s. We can do this easily by solving an
  over-determined linear system, which we know that always has a solution, 
  since the decomposition exists. Doing that, we deduce that 
  $\lambda_1 = 3$, $\lambda_2 = 15$, $\lambda_3 = 15$ and $\lambda_4 = 5$.
\end{enumerate}
\end{example}

\section{Tensor Product and Segre Varieties}\label{sec:TensorDec_Segre}
\unskip
%%%%%%%%%%%%%%%%%%%%%%%%%%%%%%%%%%%%%%%%%%%%%%%%%%%%%%%%%%%%%%%%%%%%%%%%%%
\subsection{Introduction: First Approaches}
As we saw in the Introduction, if we consider 
the space parametrizing $(n_1+1)\times \cdots \times (n_t+1)$-tensors
(up to multiplication by scalars), i.e., the space $\PP^N$, with
$N = \prod_{i=1}^t (n_i + 1)-1$, then additive decomposition problems lead us to study secant varieties of the Segre varieties $X_{n}\subset \PP^N$, ${\bf%is the bolding necessary? please check its use throughout
n}=(n_1,\ldots,n_t)$, which are the image of the Segre embedding of the 
multiprojective spaces $\PP^{n_1}\times \cdots\times \PP^{n_t}$, defined by the map:
$$\nu_{1,\ldots ,1}: \PP^{n_1}\times \cdots\times \PP^{n_t} \rightarrow \PP^N,$$
 $$\nu_{1,\ldots ,1}(P) = (a_{1,0}a_{2,0}\cdots a_{t,0}, \ldots ,a_{1,n_1}\cdots a_{t,n_t}),$$
 where
$P=((a_{1,0},\ldots ,a_{1,n_1}), \ldots,(a_{t,0},\ldots ,a_{t,n_t}))\in \PP^{n_1}\times \cdots\times \PP^{n_t}$, and the products are taken in lexicographical order.
 For example, if $P= ((a_0,a_1),(b_0,b_1,b_2)) \in \PP^{1} \times \PP^{2}$, then we have 
 $\nu_{1 ,1}(P)= (a_0b_0, a_0b_1, a_0b_2, a_1b_0, a_1b_1,a_1b_2) \in \PP^{5}.$

Note that, if $\{ x_{i,0},\ldots ,x_{i,n_i}\}$ are homogeneous coordinates in $\PP^{n_i}$ and $z_{j_1,\ldots,j_t}$, $j_i\in \{0,\ldots n_i\}$ 
are homogeneous coordinates in $\PP^N$, we have
that $X_{n}$ is the variety whose parametric equations are:
$$
z_{j_1,\ldots,j_t} = x_{j_1,1}\cdots x_{j_t,t}; \qquad j_i\in \{1,\ldots n_i\}.
$$ 
Since the use of tensors is ubiquitous in so many applications and
to know a decomposition for a given tensor allows one to ease the computational complexity
when trying to manipulate or study it, this problem has many connections with questions
raised by computer scientists in complexity theory \cite{BCS} and
by biologists and statisticians (e.g., see \cite{GHKM,GSS,ar}).

As it is to be expected with a problem with so much interest in such
varied disciplines, the approaches have been varied; see, e.g.,
\cite{BCS, Land08} for the computational complexity
\mbox{approach, \cite{ar, GSS}} for the biological
statistical approach, \cite{AOP, CGG2, CGG2err, ChCi} for
the classical algebraic geometry \mbox{approach, \cite{LaMa08, LaWe07}} for the representation theory approach,
\cite{Draisma} for a tropical approach and \cite{Fr} for a
multilinear algebra approach. Since the $t=2$ case is easy (it
corresponds to ordinary matrices), we only consider $t\geq 3$.

The first fundamental question about these secant varieties, as we
have seen, is about their dimensions. Despite all the progresses made on
this question, it still remains open; only several partial results
are known.

Notice that the case $t=3$, since it corresponds to the simplest
tensors, which are not matrices, had been widely studied, and many
previous results from several authors are collected in \cite{CGG2}.

We start by mentioning the following result on non-degeneracy; see \cite[Proposition 2.3 and Proposition 3.7]{CGG2}.

\begin{thm}\label{thm:first}
Let $n_1\leq n_2\leq \cdots \leq n_t$, $t\geq 3$. Then, the dimension of the $s$-th secant variety of the Segre variety $X_\bfn$ is as expected, i.e.,

$$\dim \sigma_{s}(X_{n}) = s(n_1+n_2+\cdots+n_t+1)-1$$

if either:
\begin{itemize}
\item $s\leq n_1+1$; \quad or
\item $\max\{ n_t+1,s\}\leq \left[ {{n_1+n_2+\cdots+n_t+1}\over 2}
\right]$. 
\end{itemize}
\end{thm}

In the paper mentioned above, these two results are obtained in two
ways. The first is via combinatorics on monomial
ideals in the multihomogeneous coordinate ring of $\PP^{n_1}\times
\cdots\times \PP^{n_t}$: curiously enough, this corresponds to
understanding possible arrangements of a set of rooks on an
$t$-dimensional chessboard (corresponding to the array representing
the tensor). There is also a reinterpretation of these problems in
terms of code theory and Hamming distance (the so-called Hamming
codes furnish nice examples of non-defective secants varieties to
Segre's of type $\PP^n\times \cdots \times \PP^n$).

Combinatorics turns out to be a nice, but limited tool for those
questions. The second part of Theorem \ref{thm:first} (and many other results that we are going to report) are obtained by the use of inverse systems and the multigraded version of apolarity theory (recall Section \ref{Apolaritysection} for the standard case, and we refer to \cite{CGG5, Galazka, GRV18, BBG} for definitions of multigraded apolarity) or via Terracini's lemma (see Lemma~\ref{Terracini}).

The idea behind these methods is to translate the problem of determining the dimension of
$\sigma_s(X)$, into the problem of determining the
{multihomogeneous Hilbert function} of a scheme
$Z \subset \PP^{n_1}\times \cdots \times \PP^{n_t}$ of $s$ generic two-fat points in multi-degree ${1} = (1,\ldots,1)$. We have that the coordinate ring of the multi-projective space $\PP^{n_1}\times \cdots \times \PP^{n_t}$ is the polynomial ring $S=\Bbbk[x_{1,0},\ldots ,x_{1,n_1},\ldots,x_{t,0},\ldots ,x_{t,n_t}]$, equipped with the multi-degree given by $\deg(x_{i,j}) = \bfe_i = (0,\ldots,\underset{i}{1},\ldots,0)$. Then, the scheme $Z$ is defined by a multi-homogeneous ideal $I = I(Z)$, which inherits the multi-graded structure. Hence, recalling the standard definition of Hilbert function (Definition \ref{def:HF}), we say that the {multi-graded Hilbert function} of $Z$ in multi-degree $\bfd = (d_1,\ldots,d_t)$ is:
$$
	\HF(Z;\bfd) = \dim_\Bbbk(S/I)_\bfd = \dim_\Bbbk S_\bfd - \dim_\Bbbk I_\bfd.
$$

%%%%%%%%%%%%%%%%%%%%%%%%%%%%%%%%%%%%%%%%%%%%%%%%%%%%
\subsection{The Multiprojective Affine Projective Method}\label{MAP}
 We describe here a way to study the dimension of $\sigma_s(X_{n})$ by studying the multi-graded Hilbert function of a scheme of fat points in multiprojective space via a very natural reduction to the Hilbert function of fat points in the standard projective space (of equal dimension). 

\smallskip
We start recalling a direct consequence of Terracini's lemma for any variety.

Let $Y \subset \PP^N$ be a positive dimensional smooth variety, and let $Z\subset Y$ be a scheme of $s$ generic {two-fat points}, i.e., a scheme defined by the ideal sheaf $\mathcal{I}_{Z} = \mathcal{I}^2_{P_1,Y}\cap \cdots \cap \mathcal{I}^2_{P_s,Y}\subset\mathcal{O}_{Y}$, where the $P_i$'s are $s$ generic points of $Y$ defined by the ideal sheaves $\calI_{P_i,Y} \subset \calO_Y$, respectively. Since there is a bijection between hyperplanes of the space $\PP^N$ containing the linear space $\langle T_{P_1}(Y),\dots ,T_{P_s}(Y) \rangle $
and the elements of $H^0(Y,\mathcal{I}_{Z}(1))$, we have the following consequence of the Terracini lemma.
 
\begin{thm}\label{thm:SecantsViaFatPoints}
Let $Y$ be a positive dimensional smooth variety; let $P_1,\ldots,P_s$ be generic points on $Y$; and let $Z\subset Y$ be the scheme defined by $\calI_{P_1,Y}^2 \cap \ldots \cap \calI_{P_s,Y}^2$. Then,
$$ \dim \sigma_s(Y) = \dim \langle T_{P_1}(Y),\dots,T_{P_s}(Y) \rangle = N - \dim_\Bbbk
H^0(Y,\mathcal{I}_{Z}(1)). $$ 
\end{thm}

Now, we apply this result to the case of Segre varieties; we give, e.g., \cite{CGG4} as the main reference. Consider $\PP^\bfn := \PP^{n_1}\times \cdots \times \PP^{n_t}$, and let $X_{\mathbf{n}} \subset \PP^N$ be its Segre embedding given by $\mathcal{O}_{\PP^{n_1}}(1)\otimes\ldots\otimes\mathcal{O}_{\PP^{n_t}}(1)$. By applying Theorem \ref{thm:SecantsViaFatPoints} and since the scheme $Z \subset X_{\mathbf{n}}$ corresponds to a scheme of $s$ generic two-fat points in $X$, which, by a little abuse of notation, we call again $Z$, we get:
$$
\dim \sigma_s(X_{\mathbf{n}}) = \HF(Z, \mathbf{1})-1.
$$
% where, for all ${\mathbf{j}}\in \mathbb{N}^t$, $H(Z,\mathbf{j})$ is the multigraded Hilbert function of $Z$, i.e., 
%$$
%H(Z,\mathbf{j}) = \dim_{\Bbbk} R_{\mathbf{j}} - \dim_k H^0( X, \mathcal{I}_{Z}(\mathbf{j})), 
%$$
%where $R=\Bbbk[x_{0,1},\dots,x_{n_1,1},\ \dots\ ,x_{0,t},\dots,x_{n_t,t}]= \oplus R_{(d_1, \dots, d_t)} $ is the multigraded homogeneous coordinate ring of $X$.

Now, let $n=n_1+ \cdots +n_t$, and consider the birational map: 
$$
\PP^\bfn \dashrightarrow \mathbb{A}^n, 
$$
where: 
\begin{center}
\begin{tabular}{c}
$\big([x_{1,0}:\dots:x_{1,n_1}],\dots,[x_{t,0},\dots,x_{t,n_t}]\big)$ \\
$\rotatebox[origin = c]{-90}{$\mapsto$}$ \\
$ \left({x_{1,1}\over x_{1,0}},{x_{1,2}\over x_{1,0}},\dots,{x_{1,n_1}\over x_{1,0}};
 {x_{2,1}\over x_{2,0}},\dots,{x_{2,n_2}\over x_{2,0}};\dots; {x_{t,1}\over
x_{t,0}},\dots,{x_{t,n_t}\over x_{t,0}}\right).$
\end{tabular}
\end{center}
This map is defined in the open subset of $\PP^\bfn$ given by $\{x_{1,0}x_{2,0} \cdots x_{t,0}\neq 0\}$.

Now, let $\Bbbk[z_0,z_{1,1},\dots,z_{1,n_1}, z_{2,1},\dots,z_{2,n_2}, \dots , z_{t,1}, \dots ,z_{t,n_t}]$ be the coordinate ring of $\PP^n$, and consider the embedding of $\mathbb{A}^n \rightarrow \PP^n$, whose image is the affine chart $\{ z_0\neq 0 \}$. By composing the two maps above, we get:
$$
\varphi: \bbP^\bfn \dashrightarrow \bbP^n, 
$$
with: 
\begin{center}
\begin{tabular}{c}
$\big([x_{1,0}:\dots:x_{1,n_1}],\dots,[x_{t,0},\dots,x_{t,n_t}]\big)$ \\
$\rotatebox[origin = c]{-90}{$\mapsto$}$ \\
$ \left[1:{x_{1,1}\over x_{1,0}}:{x_{1,2}\over x_{1,0}}:\dots:{x_{1,n_1}\over x_{1,0}}:
 {x_{2,1}\over x_{2,0}}:\dots:{x_{2,n_2}\over x_{2,0}}:\dots:{x_{t,1}\over
x_{t,0}}:\dots:{x_{t,n_t}\over x_{t,0}}\right].$
\end{tabular}
\end{center}
Let $Z\subset X$ be a zero-dimensional scheme, which is contained in the affine chart $\{x_{0,1}x_{0,2}\cdots x_{0,t}\neq 0\}$, and let $Z' = \varphi(Z)$. We want to construct 
a scheme $W\subset \PP^n$ such that $\HF(W;d) = \HF(Z;(d_1,\ldots,d_t))$, where $d= d_1+\ldots + d_t$.

Let:
$$Q_0,Q_{1,1},\dots,Q_{1,n_1},Q_{2,1},\dots,Q_{2,n_2},Q_{t,1},\dots,Q_{t,n_t}$$
be the coordinate points of $\PP^{n}$. Consider the linear space $\Pi_i\cong \PP^{n_i-1} \subset \PP^{n}$, where $\Pi_i=\langle Q_{i,1},\dots,Q_{i,n_i}\rangle $. The defining ideal of $\Pi_i$ is:
$$
I(\Pi_i) = (z_0, z_{1,1}, \ldots , z_{1,n_1}; \ldots ; \hat z_{i,1},\ldots ,\hat z_{i,n_i}; 
\ldots ;z_{t,1}, \ldots , z_{t,n_t}) \ . 
$$
Let $W_i$ be the subscheme of $\PP^{n}$ denoted by $(d_i-1)\Pi_i$, i.e., the scheme defined by the ideal $I(\Pi_i)^{d_i-1}$. Since $I(\Pi_i)$ is a prime ideal generated by a regular sequence, the ideal $I(\Pi_i)^{d_i-1}$ is saturated (and even primary for $I(\Pi_i)$). 
Notice that $W_i\cap W_j = \emptyset$, for $i\neq j$. With this construction, we have the following key result.

\begin{thm}\label{thm:multiprojaffproj}
Let $Z,\ Z'$, $W_1, \dots ,W_t$ be as above, and let
$W = Z' + W_1 + \cdots + W_t\subset \PP^n$. Let $I(W) \subset S $ and $I(Z) \subset R $ be the ideals of 
$W$ and $Z$, respectively. Then, we have, for all $(d_1,\dots ,d_t)\in \mathbb{N}^t$: 
$$
\dim_{\Bbbk} I(W)_d = \dim_{\Bbbk} I(Z)_{(d_1,\dots ,d_t)} ,
$$
where $d = d_1 + \dots + d_t$.
\end{thm}

Note that when studying Segre varieties, we are only interested to the case $(d_1,\dots ,d_t) = (1, \ldots , 1)$; but, in the more general case of Segre--Veronese varieties, we will have to look at Theorem \ref{thm:multiprojaffproj} for any multidegree $(d_1,\ldots,d_t)$; see Section~\ref{sec:SegreVeronese}.

Note that the scheme $W$ in $\PP^n$ that we have constructed has two parts: the part $W_1 + \cdots + W_t$ (which we shall call {the part at infinity} and we denote as $W_\infty$) and 
the part $Z^\prime$, which is isomorphic to our original zero-dimensional
scheme $Z \subset X$. Thus, if $Z = \emptyset$ (and hence, $Z^\prime = \emptyset$), we obtain from the theorem that:
$$
\dim_{\Bbbk} I(W_\infty)_d = \dim_{\Bbbk} S_{(d_1, \ldots , d_t)}, \quad d = d_1 + \cdots + d_t.
$$
It follows that:
$$
\HF(W_\infty;d) = {d_1 + \cdots + d_t + n\choose n} - {d_1+n_1\choose n_1}\cdots {d_t+n_t\choose n_t}. 
$$
With this observation made, the following corollary is immediate.

\begin{corollary} 
Let $Z$ and $Z^\prime$ be as above, and write $W = Z^\prime + W_\infty$. 
Then,
$$
\HF(W;d) = \HF(Z;(d_1,\dots ,d_t)) + \HF(W_\infty;d).
$$
\end{corollary} 
Eventually, when $Z$ is given by $s$ generic two-fat points in multi-projective space, we get the~following.
\begin{thm}\label{thm:CGG8}
Let $Z\subset X$ be a generic set of $s$ two-fat points, and let $W\subset \PP^n$ be as in the Theorem \ref{thm:multiprojaffproj}. Then, we~have: 
$$
\dim \sigma_s( X_{\mathbf{n}}) = \HF(Z,(1,\ldots,1))-1 = N - \dim I(W)_t, 
$$
where $N = \Pi_{i=1}^t(1 + n_i) - 1$.
\end{thm}

Therefore, eventually, we can study a projective scheme $W \subset \PP^ n$, which is made of a union of generic two-fat points and of fat linear spaces. Note that, when $n_1=\ldots=n_t=1$, then also $W$ is a scheme of fat~points. 

%%%%%%%%%%%%%%%%%%%%%%%%%%%%%%%%%%%%%%%%%%%%%%%%%%%%%%%%%%%%%%%%%%%%%%%%%%

\subsection{The Balanced Case} One could try to attack the problem starting
with a case, which is in a sense more ``regular'', i.e., the ``balanced''
case of $n_1=\cdots =n_t$, $t\geq 3$. Several partial results are
known, and they lead Abo, Ottaviani and Peterson to propose, in
their lovely paper \cite{AOP}, a conjecture, which states that
there are only a finite number of defective Segre varieties of the
form $\PP^{n}\times \cdots\times \PP^{n}$, and their guess is that
$\sigma_4(\PP^{2}\times \PP^{2}\times\PP^{2})$ and
$\sigma_3(\PP^{1}\times \PP^{1}\times\PP^{1}\times\PP^{1})$ are
actually the only defective cases (as we will see later, this is just
part of an even more hazardous conjecture; see Conjecture \ref{ConjOnSegres}).

In the particular case of $n_1=\cdots =n_t= 1$, the question has been completely solved in
\cite{CGG8}, supporting the above conjecture.

\begin{thm}\label{P1P1P1...}\cite{CGG8}
Let $t, s \in \Bbb N$, $t \geq 3$. Let $X_{1}$ be the Segre embedding of $\PP^1\times \cdots
\times \PP^1$, $(t$-times$)$. The dimension of $\sigma_s(X_{1}) \subset
\PP^N$, with $N = 2^t -1$, is always as expected, i.e., $$
\dim \sigma_s(X) = \min \{ N, s(t+1)-1 \},$$ except for $t = 4$, $s=3$. In this last case,
 $\dim \sigma_3(X) = 13$, instead of $14$.
\end{thm}
The method that has been used to compute the multi-graded
Hilbert function for schemes of two-fat points with generic support in multi-projective spaces
is based first on the procedure described on the multiprojective affine projective method explained in the previous section, which brings to the study of the standard Hilbert function of the schemes $W\subset \PP^t$. Secondly, the problem of determining the dimension of $I(W)_t$ can be attacked by induction, via the powerful tool constituted by the {differential Horace method}, created by Alexander and Hirschowitz; see Section \ref{Horacediffsection}. This is used in \cite{CGG8} together with other ``tricks'', which allow one to ``move'' on a hyperplane some of the conditions imposed by the fat points, analogously as we have described in the examples in Section \ref{Horacediffsection}. These were the key ingredients to prove Theorem \ref{P1P1P1...}.

The only defective secant variety in the theorem above is made by
the second secant variety of $\PP^1\times \PP^1\times \PP^1\times \PP^1
\subset \PP^{15}$, which, instead of forming a hypersurface in
$\PP^{15}$, has codimension two. Via Theorem \ref{thm:CGG8} above, this is geometrically related to a configuration of seven fat points; more precisely, in this case, the scheme $W$ of Theorem~\ref{thm:CGG8} is union of three two-fat points and four three-fat ones (see also Theorem~\ref{thm:multiprojaffproj} for a detailed description of $W$). These always lie on a rational normal curve in $\PP^4$ (see, e.g., Theorem 1.18 of \cite{Harris}) and do not have the expected Hilbert function in degree four, by the result in \cite{CEG}.

\smallskip
 For the general ``balanced'' case $\PP^{n}\times \cdots\times \PP^{n}$, the
following partial result is proven in \cite{AOP}.

\begin{thm} \label{tensorpower}
Let $X_\bfn$ be the Segre embedding of $\PP^{n}\times \cdots\times \PP^{n}$ $(t$ times$)$, $t\ge 3$. Let $s_t$ and $e_t$ be defined by:
$$s_t=\left\lfloor \frac{(n+1)^t}{nt+1}\right\rfloor \hskip 8pt{\rm and}\hskip 8pt e_t\equiv s_t\
mod\ (n+1) \hskip 8pt {\rm with}\hskip 8pt e_t\in\{0,\dots ,n\}.$$
Then:
\begin{itemize}
\item if $s\le s_t-e_t$, then $\sigma_s(X)$ has the expected
dimension;
\item if $s\ge s_t-e_t+n+1$, then $\sigma_s(X)$ fills the ambient
space.
\end{itemize}
\end{thm}

In other words, if $s_t =q (n+1)+r $, with $1 \leq r \leq n$, then $\sigma_s(X_\bfn)$ has the expected dimension both for $s \geq (q+1)(n+1)$ and for $s \leq q(n+1)$, but if $n+1$ divides $s_t$, then $\sigma_s(X)$ has the expected dimension for any $s$.

Other known results in the ``balanced'' case are the following:

\begin{itemize}
\item $\sigma_4(\PP^{2}\times \PP^{2}\times\PP^{2})$ is defective with
defect $\delta_4=1$; see \cite{CGG1}.

\item {$\PP^n \times \PP^n\times \PP^n$} is never defective for $n \geq
3$; see \cite{Li}; 

\item $\PP^n \times \PP^n\times \PP^n\times \PP^n$ for $2 \leq n \leq 10$ is never defective except at most for $s = 100$ and $n = 8$ or for $s = 357$
and $n = 10$; see \cite{Ges}. 
\end{itemize}
%%%%%%%%%%%%%%%%%%%%%%%%%%%%%%%%%%%%%%%%%%%%%%%%%%%%%%%%%%%%%%%%%%%%%%%%%%

\subsection{The General Case}
When we drop the balanced dimensions request, not many defective
cases are actually known. For example:
\begin{itemize}
\item $\PP^2 \times \PP^3\times \PP^3$ is defective only for $s=5$, with defect $\delta_4=1$; see \cite{AOP};
\item $\PP^2 \times \PP^n \times \PP^n$ with $n$ even is defective
only for $s={\frac {3n}2}+1$ with defect $\delta_s=1$; see \cite{CGG2,CGG2err}.
\item $\PP^1 \times \PP^1\times \PP^n\times \PP^n$ is defective only
for $s=2n+1$ with defect $\delta_{2n+1}=1$; see \cite{AOP}.
\end{itemize}

However, when taking into consideration cases where the $n_i$'s are {far}
from being equal, we run into another ``defectivity phenomenon'', known as
``the unbalanced case''; see \cite{CGG2, CGG7}.

\begin{thm} \label{unbalanced}
Let $X_\bfn$ be the Segre embedding of $\Bbb P^{n_1} \times \cdots \times \Bbb P^{n_t}
\times \Bbb P^n \subset \PP ^M$, with $M=(n+1)\Pi _{i=1} ^t ( n_i +1)-1$.
Let $N=\Pi _{i=1} ^t ( n_i +1)-1$, and assume $n > N - \sum _{i=1}^tn_i + 1$. Then, $\sigma_s(X)$ is defective for:
$$N - \sum _{i=1} ^{t} n_i +1 < s \leq \min \{ n;N\},$$
with defect equal to $\delta_{s} (X) = s^2-s(N-\sum_{i=1}^tn_i +1)$.
\end{thm}

The examples described above are the few ones for which
defectivities of Segre varieties are known. Therefore, the following
conjecture has been stated in \cite{AOP}, where, for $n_1\leq n_2 \leq \ldots \leq n_t$, it is proven for $s \leq 6$.

\begin{conj}\label{ConjOnSegres}
The Segre embeddings of $\PP^{n_1} \times \cdots. \times \PP^{n_t}$,
$t\geq 3$, are never defective, except for:
\begin{itemize}
\item $\PP^1 \times \PP^1 \times \PP^1\times \PP^1$, for $s=3$, with
$\delta_3=1$;
\item $\PP^2 \times \PP^2 \times \PP^2$, for $s=4$, with
$\delta_4=1$;
\item the ``unbalanced case'';
\item $\PP^2 \times \PP^3 \times \PP^3$, for $s=5$, with
$\delta_5=1$;
\item $\PP^2 \times \PP^n \times \PP^n$, with $n$ even, for $s={\frac {3n}2}+1$, with
 $\delta_s=1)$;
\item $\PP^1 \times \PP^1 \times \PP^n \times \PP^n$, for $s=2n+1$, with $\delta_{2n+1}=1)$.
\end{itemize}
\end{conj}

%%%%%%%%%%%%%%%%%%%%%%%%%%%%%%%%%%%%%%%%%%%%%%%%%%%%%%%%%%%%%%%%%%%%%%%%%%

\section{Other Structured Tensors}\label{sec:OtherDec}

There are other varieties of interest, parametrizing other ``structured tensors'', i.e., tensors that have determined properties. In all these cases, there exists an {additive decomposition problem}, which can be geometrically studied similarly as we did in the previous sections. In this section, we want to present some of these cases. 

In particular, we consider the following structured tensors:
\begin{enumerate}
	\item skew-symmetric tensors, i.e., $v_1\wedge\ldots\wedge v_k \in \bigwedge^k V$;
	\item decomposable partially-symmetric tensors, i.e., $L_1^{d_1}\otimes\ldots\otimes L_t^{d_t} \in S^{d_1}V_1^*\otimes\cdots\otimes S^{d_t}V_t^*$;
	\item ${d}$-powers of linear forms, i.e., homogeneous polynomials $L_1^{d_1}\cdots L_t^{d_t} \in S^d V^*$, for any partition ${d} = (d_1,\ldots,d_t) \vdash d$;
	\item reducible forms, i.e., $F_1\cdots F_t \in S^{d}V^*$, where $\deg(F_i) = d_i$, for any partition $\bfd = (d_1,\ldots,d_t) \vdash d$;
	\item powers of homogeneous polynomials, i.e., $G^k \in S^{d} V^*$, for any $k | d$;
\end{enumerate}

%%%%%%%%%%%%%%%%%%%%%%%%%%%%%%%%%%%%%%%%%%%%%%%%%%%%%%%%%%%%%%%%%%%%%%%%%%
\subsection{Exterior Powers and Grassmannians}\label{sec:ExteriorGrassmannians}

Denote by $\bbG(k,n)$ the Grassmannian of $k$-dimensional linear subspaces of
$\PP ^n \cong \PP V$, for a fixed $n+1$ dimensional vector space
$V$. We consider it with the embedding given by its Pl\"ucker coordinates as embedded in $\bbP^{N_k}$, where $N_k = {n+1\choose k} -1$.

The dimensions of the higher secant varieties to the Grassmannians
of lines, i.e., for $k=1$, are well known; e.g., see \cite{Z},
\cite{ER} or \cite{CGG3}. The secant variety $\sigma_s(Gr(1,n))$
parametrizes all $(n+1)\times (n+1)$ skew-symmetric matrices of rank~at most $2s$.

\begin{thm}\label{GrassLines} 
We have that $\sigma_s(\bbG(1,n))$ is
defective for $s < \lfloor {n+1 \over 2}\rfloor$ with defect equal to $\delta_s =
2s(s-1)$.
\end{thm}

For $k\geq 2$, not many results can be found in the classical or
contemporary literature about this problem; e.g., see \cite{ER, CGG3, AbOtPetGRASS}. However, they
are sufficient to have a picture of the whole situation.
Namely, there are only four other cases that are known to be
defective (e.g., see \cite{CGG3}), and it is conjectured in
\cite{BaDraGraafGRASS} that these are the only ones. This is summarized in the following conjecture:

\begin{conj}[Baur--Draisma--de Graaf, \cite{BaDraGraafGRASS}] \label{BDdG}
Let $k\ge 2$. Then, the secant variety $\sigma_s(\bbG(k,n))$ has the expected dimension
except for the following cases:
 $$\begin{array}{c|c|c}
 &\textrm{actual codimension}&\textrm{expected codimension}\\
 \hline
 \sigma_3(\bbG(2,6))&1&0\\
 \hline \sigma_3(\bbG(3,7))&20&19\\
 \hline \sigma_4(\bbG(3,7))&6&2\\
 \hline \sigma_4(\bbG(2,8))&10&8\\
 \end{array}$$
\end{conj}

\bigskip
In \cite{BaDraGraafGRASS}, they proved the conjecture for $n\leq 15$ (the case $n\leq 14$ can be found in \cite{McG06}). The conjecture has been proven to hold for $s\leq 6$ (see
\cite{AbOtPetGRASS}) and later for $s\leq 12$ in \cite{Bor}.

A few more results on non-defectivity are proven in
\cite{AbOtPetGRASS,CGG3}. We summarize them
in the~following.

\begin{thm} The secant variety $\sigma_s(\bbG(k,n))$ has the
expected dimension when:
\begin{itemize}
\item $k\geq 3$ and $ks\leq n+1$,

\item $k=2$, $n\ge 9$ and $s\le s_1(n)$ or $s\ge s_2(n)$, 
\end{itemize}
where:

$$s_1(n)=\left\lfloor\frac{n^2}{18}-\frac{20n}{27}+\frac{287}{81}\right\rfloor+
\left\lfloor\frac{6n-13}{9}\right\rfloor \text{ and } 
 s_2(n)=\left\lceil\frac{n^2}{18}-\frac{11n}{27}+\frac{44}{81}\right\rceil+
\left\lceil\frac{6n-13}{9}\right\rceil;$$

\noindent in the second case, $\sigma_s(\bbG(2,n))$ fills the ambient space.
\end{thm}
Other partial results can be found in \cite{Bor}, while in \cite{RR}, the following theorem can be found.

\begin{dfn}\label{h}
Given an integer $m \geq 2$, we define a function
$h_m : \mathbb{N} \rightarrow \mathbb{N}$
as follows: 
\begin{itemize}
\item $h_m(0) = 0$;
\item for any $k \geq 1$, write
$k + 1 = 2^{\lambda_1} + 2^{\lambda_2} + \ldots + 2^{\lambda_\ell} + \epsilon$, for a suitable choice of
$\lambda_1 > \lambda_2 > \ldots > \lambda_\ell \geq 1$, $\epsilon \in \{0, 1\}$, and define:
$$h_m(k) := m^{\lambda_1-1} + m^{\lambda_2-1} + \ldots + m^{\lambda_\ell-1}.$$
In particular, we get $h_m(2k) = h_m(2k - 1)$, and $h_2(k) = \lfloor \frac{k+1}{2}\rfloor$.
\end{itemize}
\end{dfn}
\begin{thm}[Theorem 3.5.1 of {\cite[]{RR}}]
Assume that $k \geq 2$, and set:
$$\alpha := \left\lfloor \frac{n+1}{k+1} \right\rfloor.$$
If either:
\begin{itemize}
\item $n \geq k^2 + 3k + 1$ and $h \leq \alpha h_\alpha(k - 1)$;
\end{itemize}
or:
\begin{itemize}
\item $n < k^2 + 3k + 1$, $k$ is even, and $h \leq (\alpha - 1)h_\alpha(k - 1) + h_\alpha(n - 2 - \alpha r)$;
\end{itemize}
or:
\begin{itemize}
\item $n < k^2+3k+1$, $r$ is odd, and: $$h \leq (\alpha - 1)h_\alpha(k-2)+h_\alpha(\min\{n-3-\alpha (k-1), r-2\}),$$
\end{itemize} 
then, $\bbG(k, n)$ is not $(h + 1)$-defective.
\end{thm}

This results strictly improves the results in \cite{AbOtPetGRASS} for $k \geq 4$, whenever $(k, n)\neq (4, 10), (5,11)$; see \cite{RR}.

\bigskip
Notice that, if we let $\rk_{\wedge}(k,n)$ be the generic rank with respect to $\bbG(k,n)$, i.e., the minimum $s$ such that $\sigma_s(\bbG(k,n))=\PP^{N_k}$, we have that the results above give
that, asymptotically: 
$$\rk_{\wedge}(2,n) \sim {n^2 \over 18}.$$
A better asymptotic result can be found in \cite{RR}.

 Finally, we give some words concerning the methods involved.
The approach in \cite{CGG3} uses Terracini's lemma and an
exterior algebra version of apolarity. The main idea there is to consider the analog of the perfect pairing induced by the apolarity action that we have seen for the symmetric case in the skew-symmetric situation; see Section \ref{Apolaritysection}. In fact, the pairing considered here is: $$\bigwedge^kV \times \bigwedge^{n-k}V \to \bigwedge^n V \simeq \Bbbk,$$ induced by the multiplication in $\bigwedge V$, and it defines the apolarity of a subspace $Y\subset V$ of dimension $k$ to be $Y^{\perp}:=\{ w \in \bigwedge^{n-k}V \, | \, w\wedge v =0 \, \forall \, v \in Y\}$. Now, one can proceed in the same way as the symmetric case, namely by considering a generic element of the Grassmannian $\bbG(k,n)$ and by computing the tangent space at that point. Then, its orthogonal, via the above perfect pairing, turns out to be, as in the symmetric case, the degree $n-k$ part of an ideal, which is a double fat point. Hence, in \cite{CGG3}, the authors apply Terracini's lemma to this situation in order to study all the known defective cases and in other various cases. Notice that the above definition of skew-symmetric apolarity works well for computing the dimension of secant varieties to Grassmannians since it defines the apolar of a subspace that is exactly what is needed for Terracini's lemma, but if one would like to have an analogous definition of apolarity for skew symmetric tensors, then there are a few things that have to be done. Firstly, one needs to extend by linearity the above definition to all the elements of $ \bigwedge^{k}V$. Secondly, in order to get the equivalent notion of the apolar ideal in the skew symmetric setting, one has to define the skew-symmetric apolarity in any degree $\leq d$. This is done in \cite{ABMM}, where also the skew-symmetric version of the apolarity lemma is given. Moreover, in \cite{ABMM}, one can find the classification of all the skew-symmetric-ranks of any skew-symmetric tensor in $ \bigwedge^{3}\mathbb{C}^n$ for $n\leq n$ (the same classification can actually be found also in \cite{Segre, West}), together with algorithms to get the skew-symmetric-rank and the skew-symmetric decompositions for any for those tensors (as far as we know, this is~new).

Back to the results on dimensions of secant varieties of Grassmannians: in \cite{BaDraGraafGRASS}, a tropical geometry approach is involved. In \cite{AbOtPetGRASS}, as was done by Alexander and Hirshowitz for the symmetric case, the authors needed to introduce a specialization technique, by placing a certain number of
points on sub-Grassmannians and by using induction. In this way, they could prove several non-defective cases. Moreover, in the same work, invariant theory was used to describe the equation of $\sigma_3(\bbG(2,6))$, confirming the work of Schouten \cite{Schou}, who firstly proved that it was defective by showing that it is a hypersurface (note that by parameter count, it is expected to fill the ambient space). Lascoux \cite{Las} proved that the degree of Schouten's hypersurface is seven. In \cite{AOP}, with a very clever idea, an explicit description of this degree seven invariant was found by relating its cube to the determinant of a $21 \times 21$ symmetric~matrix.

Eventually, in \cite{RR}, the author employed a new method for studying the defectivity of varieties based on the study of osculating spaces. 

%%%%%%%%%%%%%%%%%%%%%%%%%%%%%%%%%%%%%%%%%%%%%%%%%%%%%%%%%%%%%%%%%%%%%%%%%%

%%%%%%%%%%%%%%%%%%%%%%%%%%%%%%%%%%%%%%%%%
\subsection{Segre--Veronese Varieties}\label{sec:SegreVeronese}
Now, we consider a generalization of the apolarity action that we have seen in both Section \ref{sec:SymmetricDec_Veronese} and Section \ref{sec:TensorDec_Segre} to the multi-homogeneous setting; see \cite{CGG5, Galazka, GRV18, BBG}. More precisely, fixing a set of vector spaces $V_1,\ldots,V_t$ of dimensions $n_1+1,\ldots,n_t+1$, respectively, and positive integers $d_1,\ldots,d_t$, we consider the space of partially-symmetric tensors: $$S^{d_1}V^*_1\otimes\ldots\otimes S^{d_t}V^*_t.$$
 The {Segre--Veronese variety} parametrizes {decomposable partially-symmetric tensors}, i.e., it is the image of the~embedding:
$$\begin{array}{c c c c}
	\nu_{\bfn,\bfd} : & \mathbb{P}(S^1V_1^*) \times \ldots \times \mathbb{P}(S^1V_t^*) & \longrightarrow & \mathbb{P}(S^{d_1}V_1^*\otimes\ldots\otimes S^{d_t}V_t^*) \\
	& ([L_1],\ldots,[L_t]) & \mapsto & \left[L_1^{d_1}\otimes\ldots\otimes L_t^{d_t}\right],
\end{array}$$
where, for short, we denote $\bfn = (n_1,\ldots,n_t),~\bfd = (d_1,\ldots,d_t)$.

More geometrically, the {Segre--Veronese variety} is the image of the Segre--Veronese embedding: 
$$
	\nu_{\bfn,\bfd} : \mathbb P^{\bfn} := \mathbb P^{n_1}\times\ldots\times\mathbb{P}^{n_t} \longrightarrow \mathbb{P}^N, \quad \text{ where } N = \prod_{i=1}^t {d_i+n_i \choose n_i} - 1,
$$
given by $\mathcal{O}_{\mathbb{P}^{\bfn}}(\bfd) := \mathcal{O}_{\mathbb{P}^{n_1}}(d_1)\otimes \ldots \otimes \mathcal{O}_{\mathbb{P}^{n_t}}(d_t)$, that is via the forms of 
 multidegree $(d_1,\dots ,d_t)$ of the multigraded homogeneous coordinate ring: 
$$R = \Bbbk[x_{1,0},\ldots ,x_{1,n_1};x_{2,0},\ldots ,x_{2,n_2};\ldots ;x_{t,0},\ldots ,x_{t,n_t}].$$
For instance, if $\bfn = (2,1)$, $\bfd =(1,2) $ and $P= ([a_0:a_1:a_2],[b_0:b_1]) \in \PP^{2} \times \PP^{1}$, we have 
 $\nu_{\bfn,\bfd}(P)=
[a_0b_0^2:a_0b_0b_1:a_0b_1^2:a_1b_0^2: a_1b_0b_1: a_1b_1^2: a_2b_0^2: a_2b_0b_1: a_2b_1^2]$, where the products are taken in lexicographical order.
We denote the embedded variety ${\nu}_{\bfn,\bfd}({\mathbb P}^{\bfn})$ by $X_{\bfn,\bfd}$. Clearly, for $t=1$, we recover Veronese varieties, while for $(d_1,\dots ,d_t)= (1,\ldots,1)$, we get the Segre varieties.

The corresponding additive decomposition problem is as follows.
\begin{problem}
	Given a partially-symmetric tensor $T \in S^{d_1} V_1^* \times \ldots \times \mathbb{P}S^{d_t} V_t^*$ or, equivalently, a multihomogeneous polynomial $T \in R$ of multidegree $\bfd$, find the smallest possible length $r$ of an expression $T = \sum_{i=1}^r L_{i,1}^{d_1}\otimes\ldots\otimes L_{i,t}^{d_t}$.
\end{problem}

As regards the generic tensor, a possible approach to this problem mimics what has been done for
Segre and Veronese varieties. One can use Terracini's lemma (Lemma \ref{Terracini} and Theorem \ref{thm:SecantsViaFatPoints}), as in \cite{CGG2,CGG4}, to translate the problem of determining the dimensions of the higher secant varieties of $X_{\bfn,\bfd}$ into that of calculating the value, at $\bfd=(d_1,\dots ,d_t)$, of the Hilbert function of generic sets of two-fat points in $\mathbb P^\bfn$.
Then, by using the multiprojective affine projective method introduced in Section \ref{MAP}, i.e., by passing to an affine chart in $\mathbb P^{\bfn}$
and then homogenizing in order to pass to $\mathbb P^n$, with
$n=n_1+\cdots+n_t$, this last calculation amounts to computing the
Hilbert function in degree $d = d_1 + \cdots + d_n$ for the
subscheme $W\subset \mathbb P^n$; see Theorem \ref{thm:multiprojaffproj}.

There are many scattered results on the dimension of
$\sigma_s(X_{\bfn,\bfd})$, by many authors, and
very few general results. One is the following, which generalizes
the ``unbalanced'' case considered for Segre varieties; see
\cite{CGG7,AOP}.

\begin{thm} \label{unbalancedSV} Let $X = X_{(n_1, \ldots , n_t, n),(d_1, \ldots , d_t,1)}$ be the
Segre--Veronese embedding:
$$
 \PP^{n_1}\times \ldots \times \PP^{n_t} \times \PP^n \buildrel{(d_1,\ldots,d_t,1)}\over{\longrightarrow} X
 \subset \PP^M, \quad \text{ with } M =(n+1)\left(\prod_{i=1}^t{n_i + d_i\choose d_i}\right) - 1.
$$
Let $ N = \Pi_{i=1}^t{n_i + d_i\choose d_i} - 1, $ then, for
$N-\sum_{i=1}^tn_i +1 < s \leq \min\{n,N\}$, the secant variety
 $\sigma_s (X)$ is defective with $\delta_{s} (X) =
s^2-s(N-\sum_{i=1}^tn_i +1)$.
\end{thm}

When it comes to Segre--Veronese varieties with only two factors, there are
many results by many authors, which allow us to get a quite complete
picture, described by the following conjectures, as stated in
\cite{AboBra}.
\begin{conj}
 Let $X = X_{(m,n),(a,b)}$, then $X$ is never defective, except for:
\begin{itemize}
\item $b=1$, $m\geq2$, and it is unbalanced (as in the theorem above); \par
\item $(m,n)= (1,n)$, $(a,b)=(2d,2)$;
\item $(m,n)= (3,4) $, $(a,b)=(2,1)$;
\item $(m,n)= (2,n)$, $(a,b)=(2,2)$;
\item $(m,n)= (2,2k+1) $, $k\geq 1$, $(a,b)=(1,2)$;
\item $(m,n)= (1,2) $, $(a,b)=(1.3)$;
\item $(m,n)= (2,2) $, $(a,b)=(2,2)$;
\item $(m,n)= (3,3) $, $(a,b)=(2,2)$;
\item $(m,n)= (3,4) $, $(a,b)=(2,2)$.
\end{itemize}
\end{conj}
\begin{conj}
 Let $X = X_{(m,n),(a,b)}$, then for $(a,b) \geq (3,3)$, $X$ is never defective.
\end{conj}

The above conjectures are based on results that can be found in \cite{Lon890,DF,ChCi,CaCh,CGG5,Boc05,Bal06,CGG6,BaDraGraafGRASS,Abrescia,Ott08,AboBra09,Abo,
BeCC,AboBra2,BaBeC,AboBra}.

For the Segre embeddings of many copies of $\Bbb P^1$, we have a complete result. First, in \cite{CGG5} and in~\cite{CGG6}, the cases of two and three copies of $\Bbb P^1$, respectively, were completely solved.

\begin{thm}\label{SV2factors}
Let $X = X_{(1,1),(a,b)}$, $a\leq b $. Then, $X$ is never defective, except for
$(a,b) = (2,2d)$; in this case, $\sigma_{2d+1}(X)$ is defective
with $\delta_{2d+1}=1$.
\end{thm}

\begin{thm}\label{SV3factors}
Let $X = X_{(1,1,1),(a,b,c)}$, $a\leq b \leq c$; then $X$ is never defective, except for:
\begin{itemize}
\item $(a,b,c) = (1,1,2d)$; in this case, $\sigma_{2d+1}(X)$ is defective
with $\delta_{2d+1}=1$;
\item $(a,b,c) =(2,2,2)$; here, $\sigma_7(X)$ is defective, and
$\delta_7=1$.
\end{itemize}
\end{thm}
In \cite{LP13} the authors, by using an induction approach, whose basic step was Theorem \ref{P1P1P1...} about the Segre varieties $X_{\bf1}$ \cite{CGG8}, concluded that there are no other defective cases except for the previously-known ones.

\begin{thm}\label{SVrfactors}
Let $X = X_{(1,\ldots,1),(a_1,\ldots ,a_r)}$. Then, $X$ is never defective, except for:
\begin{itemize}
\item $X_{(1,1),(2,2d)}$ (Theorem \ref{SV2factors});
\item $X_{(1,1,1),(1,1,2d)}$ and $X_{(1,1,1),(2,2,2)}$ (Theorem \ref{SV3factors});
\item $X_{(1,1,1,1),(1,1,1,1)}$ (Theorem \ref{P1P1P1...}).
\end{itemize}
\end{thm}

For several other partial results on the defectivity of certain Segre--Veronese varieties, see, e.g., \cite{CGG5,CaCat09,AboBra2,RR}, and for an asymptotic result about non-defective Segre--Veronese varieties, see \cite{AMR18,RR}.

%%%%%%%%%%%%%%%%%%%%%%%%%%%%%%%%%%%%%%%%%%%%%%%%%%%%%%%
\subsection{Tangential and Osculating Varieties to Veronese Varieties}\label{sec: tangential}
Another way of generalizing what we saw in Section \ref{sec:SymmetricDec_Veronese} for secants of Veronese Varieties $X_{n,d}$ is to work with their tangential and osculating varieties. 

\begin{dfn} \label{osculating}
Let $X_{n,d}\subset \PP ^N$ be a Veronese variety. We denote by
 $\tau (X_{n,d})$ the \emph{tangential variety} of $X_{n,d}$, i.e., the closure in $\PP^N$ of the union of all tangent spaces: 
 $$\tau(X_{n,d}) = \overline{\bigcup _{P\in X_{n,d}} T_P(X_{n,d})}\subset \bbP^N.$$
 More in general, we denote by $O^k(X_{n.d})$ the \emph{$k$-th osculating variety} of $X_{n,d}$, i.e., the closure in $\bbP^N$ of the union of all $k$-th osculating spaces:
 $$O^k(X_{n,d}) = \overline{\bigcup _{P\in X_{n,d}}
O^k_P(X_{n,d})}\subset \bbP^N.$$
Hence, $\tau (X_{n,d}) = O^1(X_{n,d})$.
\end{dfn}
These varieties are of interest also because the space $O^k(X_{n,d})$ parametrizes a particular kind of form. Indeed, if the point $P = [L^d] \in X_{n,d} \subset \bbP S_d$, then the $k$-th osculating space $O^k_P(X_{n,d})$ correspond to linear space $\left\{[L^{d-k}G] ~|~ G \in S_k\right\}$. Therefore, the corresponding additive decomposition problem asks the following.
\begin{problem}\label{prob:TangentialOsculating}
	Given a homogeneous polynomial $F \in \Bbbk[x_0,\ldots,x_n]$, find the smallest length of an expression $F = \sum_{i=1}^r L_i^{d-k}G_i$, where the $L_i$'s are linear forms and the $G_i$'s are forms of degree $k$.
\end{problem}
The type of decompositions mentioned in the latter problem have been called {generalized additive decomposition} in \cite{IaKa} and in \cite{BCMT}. In the special case of $k = 1$, they are a particular case of the so-called {Chow--Waring decompositions} that we treat in full generality in Section \ref{sec:ChowVeronese}. In this case, the answer to Problem \ref{prob:TangentialOsculating} is called {$(d-1,1)$-rank}, and we denote it by $\rk_{(d-1,1)}(F)$.
\begin{rmk}\label{rmk:simul}
	Given a family of homogeneous polynomials $\calF = \{F_1,\ldots,F_m\}$, we define the {simultaneous rank} of $\calF$ the smallest number of linear forms that can be used to write a Waring decomposition of all polynomials of~$\calF$. 
	
	Now, a homogeneous polynomial $F \in S^dV$ can be seen as a partially-symmetric tensor in $S^1V \otimes S^{d-1}V$ via the equality: 
	$$F = \frac{1}{d}\sum_{i=0}^n x_i \otimes \frac{\partial F}{\partial x_i}.$$
	From this expression, it is clear that a list of linear forms that decompose simultaneously all partial derivatives of $F$ also decompose $F$, i.e., the simultaneous rank of the first partial derivatives is an upper bound of the symmetric-rank of $F$. Actually, it is possible to prove that for every homogeneous polynomial, this is an equality (e.g., see \cite{TEI14} (Section 1.3) or \cite{GOV18} (Lemma 2.4)). For more details on relations between simultaneous ranks of higher order partial derivatives and partially-symmetric-ranks, we refer to \cite{GOV18}.
\end{rmk}
Once again, in order to answer the latter question in the case of the generic polynomial, we study the secant varieties to the $k$-th osculating variety of $X_{n,d}$. In \cite{CGG1}, the dimension of $\sigma_s(\tau (X_{n,d}))$ is studied ($k=1$ case). Via apolarity and inverse systems, with an analog of Theorem \ref{thm:SecantsViaFatPoints}, the problem is again reduced to the computation of the Hilbert function of some particular zero-dimensional subschemes of $\PP^n$; namely,
\begin{align*}
\dim \sigma_s (\tau(X_{n,d})) & = \dim_\Bbbk
(L_1^{d-1},\ldots,L_s^{d-1},L_1^{d-2}M_1,\ldots,L_s^{d-2}M_s)_d -1 = \\
& = \HF(Z,d)-1,
\end{align*}
where $L_1,\ldots,L_s,M_1,\ldots,M_s$ are $2s$ generic linear forms in
$\Bbbk[x_0,\ldots,x_n]$, while $\HF(Z,d)$ is the Hilbert function of a scheme
$Z$, which is the union of $s$ generic {$(2,3)$-points} in $\PP^n$, which are defined as follows.
\begin{dfn}
A \emph{$(2,3)$-point} is a zero-dimensional scheme in $\PP^n$ with support
at a point $P$ and whose ideal is of type $I(P)^3 +
I(\ell)^2$, where $I(P)$ is the homogeneous ideal of $P$ and $\ell \subset \PP^n$ is a line through $P$ defining ideal~$I(\ell)$. 
\end{dfn}
Note that when we say that $Z$ is a scheme of $s$ generic $(2,3)$-points in $\PP^n$, we mean that $I(Z) = I(Q_1) \cap \ldots \cap I(Q_s)$, where the $Q_i$'s are $(2,3)$-points, i.e., $I(Q_i) = I(P_i)^3 + I(\ell_i)^2$, such that $P_1,\ldots P_s$ are generic points in $\PP^n$, while $\ell_1,\ldots,\ell_s$ are generic lines passing though $P_1,\ldots,P_s$, respectively.

By using the above fact, in \cite{CGG1}, several cases where $\sigma_s (\tau(X_{n,d}))$ is defective were found, and it was conjectured that these exceptions were the only ones. The conjecture has been proven in a few cases in \cite{CGG2} ($s\leq 5$ and $n \geq s+1$) and in~\cite{Bal05:SecantTangentialVeronese} ($n = 2,3$). In \cite{BCGI2}, it was proven for $n \leq 9$, and moreover, it was proven that if the conjecture holds for $d = 3$, then it holds in every case. Finally, by using this latter fact, Abo and Vannieuwenhoven completed the proof of the following theorem \cite{AboVan:SecantTangentialVeronese}.
\begin{thm}\label{thm:SecantTangentialVeronese}
The $s$-th secant variety $\sigma_s(\tau(X_{n,d}))$ of the tangential variety to the Veronese variety has dimension as expected, except in the following cases:
\begin{enumerate}
\item $d = 2$ and $2 \leq 2s < n$;
\item $d = 3$ and $n = 2,3,4$.
\end{enumerate}
\end{thm}
As a direct corollary of the latter result, we obtain the following answer to Problem \ref{prob:TangentialOsculating} in the case of generic forms.
\begin{corollary}\label{cor: osculating ranks}
	Let $F\in S_d$ be a generic form. Then,
	$$
		\rk_{(d-1,1)}(F) = \left\lceil \frac{{n+d \choose n}}{2n+1} \right\rceil,
	$$
	except for: 
	\begin{enumerate}
		\item $d = 2$, where $\rk_{(d-1,1)}(F) = \left\lfloor \frac{n}{2} \right\rfloor + 1$;
		\item $d = 3$ and $n = 2,3,4$, where $\rk_{(d-1,1)}(F) = \left\lceil \frac{{n+d \choose n}}{2n+1} \right\rceil + 1$.
	\end{enumerate}
\end{corollary}

The general case of $\sigma_s(O^k(X_{n.d}))$ is studied in
\cite{BCGI1, BCGI2, BeC, BF}. Working in
analogy with the case $k=1$, the dimension of
$\sigma_s(O^k(X_{n,d})$ is related to the Hilbert function of a
certain zero-dimensional scheme $Z = Z_1\cup\cdots\cup Z_s$, whose support is a generic
union of points $P_1,\ldots,P_s\in \PP ^n$, respectively, and such that, for each $i = 1,\ldots,s$, we have that $(k+1)P_i \subset Z_i \subset (k+2)P_i$.

As one of the manifestations of the {ubiquity of fat points}, the following conjecture describes the conditions for the defect of this secant variety in terms of the Hilbert function of fat points:

\begin{conj}[{\cite{BCGI2}} (Conjecture 2a)]
 The secant variety $\sigma_s(O^k(X_{n,d}))$ is defective if and only if either:
\begin{enumerate}
\item $h^1({\mathcal I}_X(d)) > \max \{0, \deg Z - {d+n \choose
n}\}$, or
\item $h^0({\mathcal I}_T(d)) > \max \{0, {d+n \choose n} - \deg
Z\}$,
\end{enumerate} 
where $X$ is a generic union of $s$ $(k+1)$-fat points and $T$ is a generic union of $(k+2)$-fat points.
\end{conj}

In \cite{BeC,BF}, the conjecture is proven for $n=2, s\leq 9$, and in \cite{BCGI2} for $n=2$ and any $s$.

%%%%%%%%%%%%%%%%%%%%%%%%%%%%%%%%%%%%%%%%%%%%%%%%%%%%%%%%%%%%%%%%%%%%%%%%%%

\subsection{Chow--Veronese Varieties}\label{sec:ChowVeronese}
Let $\bfd = (d_1,\ldots,d_t)$ be a partition of a positive integer $d$, i.e., $d_1 \geq \ldots \geq d_t$ are positive integers, which sum to $d$. Then, we consider the following problem.

Let $S = \bigoplus_{d \in \bbN} S_d$ be a polynomial ring in $n+1$ variables.
\begin{problem}\label{prob: ChowVeronese}
	Given a homogeneous polynomial $F \in S_d$, find the smallest length of an expression $F = \sum_{i=1}^r L_{i,1}^{d_1}\cdots L_{i,t}^{d_t}$, where $L_{i,j}$'s are linear forms.
\end{problem}
The decompositions considered in the latter question are often referred to as {Chow--Waring decompositions}. We call the answer to Problem \ref{prob: ChowVeronese} as the {$\bfd$-rank} of $F$, and we denote it by $\rk_\bfd(F)$.

In this case, the summands are parametrized by the so-called {Chow--Veronese variety}, which is given by the image of the embedding:
$$
	\begin{array}{c c c c}
		\nu_{\bfd} : & \mathbb{P} S_1 \times \ldots \times \mathbb{P} S_1 & \longrightarrow & \bbP S_d, \\ 
		& ([L_1],\ldots,[L_t]) & \mapsto & \left[L_1^{d_1}\cdots L_t^{d_t}\right].
	\end{array}
$$
We denote by $X_\bfd$ the image $\nu_{\bfd}(\bbP^\bfn)$. Notice that this map can be seen as a linear projection of the Segre--Veronese variety $X_{\bfn,\bfd} \subset \bbP(S_{d_1}\otimes\ldots\otimes S_{d_t})$, for $\bfn = (n,\ldots,n)$, under the map induced by the linear projection of the space of partially-symmetric tensors $S_{d_1} \otimes\ldots\otimes S_{d_t}$ on the totally symmetric component $S_d$. Once again, we focus on the question posed in Problem \ref{prob: ChowVeronese} in the case of a generic polynomial, for which we study dimensions of secant varieties to $X_{\bfd}$.

In the case of $\bfd = (d-1,1)$, we have that $X_\bfd$ coincides with the tangential variety of the Veronese variety $\tau(X_{n,d})$, for which the problem has been completely solved, as we have seen in the previous section (Theorem \ref{thm:SecantTangentialVeronese}).

The other special case is given by $\bfd = (1,\ldots,1)$, for $d \geq 3$. In this case, $X_{d}$ has been also referred to as the {Chow variety} or as the variety of {split forms} or {completely decomposable forms}. After the first work by Arrondo and Bernardi \cite{AB}, Shin found the dimension of the second secant variety in the ternary case ($n=2$) \cite{Shin12:CompletelyRed}, and Abo determined the dimensions of higher secant varieties \cite{Abo14:CompletelyDec}. All these cases are non-defective. It is conjectured that varieties of split forms of degree $d \geq 3$ are never defective. New cases have been recently proven in \cite{7:ReducibleForms, Tor17:GenericChow}.

The problem for any arbitrary partition $\bfd$ has been considered in \cite{CaChGeOn:WaringLike}. Dimensions of all $s$-th secant varieties for any partition haves been computed in the case of binary forms ($n=1$). In a higher number of variables, the dimensions of secant line varieties ($s=2$) and of higher secant varieties with $s \leq 2 \left\lfloor \frac{n}{3} \right\rfloor$ have been computed. This was done by using the classical Terracini's lemma (Lemma \ref{Terracini}) in order to obtain a nice description of the generic tangent space of the $s$-th secant variety. In the following example, we explain how the binary case could be treated.
 
\begin{example}
	If $P = [L_1^{d_1}\cdots L_t^{d_t}] \in X_\bfd$, then it is not difficult to prove (see Proposition 2.2 in \cite[]{CaChGeOn:WaringLike}) that: 
	$$
		T_P X_\bfd = \bbP \left((I_P)_d\right), \quad \text{ with }d = d_1+\ldots+d_t
	$$
	where $I_P = (L_1^{d_1-1}L_2^{d_2}\cdots L_t^{d_t},L_1^{d_1}L_2^{d_2-1}\cdots L_t^{d_t},\ldots,L_1^{d_1}L_2^{d_2}\cdots L_t^{d_t-1})$.
	
	In the particular case of binary forms, some more computations show that actually,
	$T_P X_\bfd = \bbP \left((I'_P)_d\right)$, where $I'_P$ is the principal ideal $(L_1^{d_1-1}\cdots L_t^{d_t-1})$. In this way, by using Terracini's lemma, we obtain that, if $Q$ is a generic point on the linear span of $s$ generic points on $X_\bfd$, then:
	$$
		T_Q\sigma_s(X_\bfd) = \bbP \left((L_{1,1}^{d_1-1}\cdots L_{1,t}^{d_t-1},\ldots,L_{s,1}^{d_1-1}\cdots L_{s,t}^{d_t-1})_d\right),
	$$
	where $L_{i,j}$'s are generic linear forms. Now, in order to compute the dimension of this tangent space, we can study the Hilbert function of the ideal on the right-hand side. By semicontinuity, we may specialize to the case $L_{i,1} = \ldots = L_{i,t}$, for any $i = 1,\ldots,s$. In this way, we obtain a power ideal, i.e., an ideal generated by powers of linear forms, whose Hilbert function is prescribed by Fr\"oberg--Iarrobino's conjecture; see Remark \ref{conj:FI}. Now, since in \cite{GeSc}, the authors proved that the latter conjecture holds in the case of binary forms, i.e., the Hilbert function of a generic power ideal in two variables is equal to the right-hand side of \eqref{eq: Froberg inequality}, we can conclude our computation of the dimension of the secant variety of $X_\bfd$ in the binary case. This is the way Theorem 3.1 in \cite[]{CaChGeOn:WaringLike} was proven.
\end{example}

In the following table, we resume the current state-of-the-art regarding secant varieties and Chow--Veronese varieties.

\begin{center}
\begin{tabular}{c | c | c | c | c | l}
	{$\bfd$} & $s$ & $d$ & $n$ & {References} & {$\dim\sigma_s(X_{\bfd})$} \\
	\hline
	$(d-1,1)$ & any & any & any & \cite{AboVan:SecantTangentialVeronese} & non-defective, except for \\	
	& & & & & (1) $d = 2$ and $2 \leq 2s < n$;\\
	& & & & & (2) $d = 3$ and $n = 2,3,4$. \\
	\hline
	$(1,\ldots,1)$ & &$d>2$ & {\small $3(s-1)<n$} & \cite{AB} & non-defective\\
	~& any & any & $2$ & \cite{Abo14:CompletelyDec} & non-defective, except for cases above\\
	~ & \multicolumn{3}{c | }{\footnotesize {some numerical constraints}} & \cite{Abo14:CompletelyDec, Tor17:GenericChow} &  \\
	\hline
	any & any & any & $1$ & \cite{CaChGeOn:WaringLike} & non-defective, except for cases above \\
	& $2$ & any & any & & \\
	& $\leq 2 \left\lfloor \frac{n}{3} \right\rfloor$ & any & any & \\
\end{tabular}
\end{center}

\subsection{Varieties of Reducible Forms}
In 1954, Mammana \cite{Mamma} considered the variety of reducible plane curves and tried to generalize previous works by, among many others, C. Segre, Spampinato and Bordiga. More recently, in \cite{7:ReducibleForms}, the authors considered the {varieties of reducible forms} in full generality. 

Let $\bfd = (d_1,\ldots,d_t) \vdash d$ be a partition of a positive integer $d$, i.e., $d_1 \geq \ldots \geq d_t$ are positive integers, which sum up to $d$ and $t \geq 2$. Inside the space of homogeneous polynomials of degree $d$, we define the {variety of $\bfd$-reducible forms} as:
$$
	Y_\bfd = \{[F] \in \bbP S_d ~|~ F = G_1\cdots G_t, \text{ where } \deg(G_i) = d_i\},
$$
i.e., 
the image of the embedding:
$$
	\begin{array}{c c c c}
		\psi_{\bfd} : & \mathbb{P} S_{d_1} \times \ldots \times \mathbb{P} S_{d_t} & \longrightarrow & \bbP S_d, \\ 
		& (G_1,\ldots,G_t) & \mapsto & G_1^{d_1}\cdots G_t^{d_t}.
	\end{array}
$$
Clearly, if $d = 2$, then $\bfd = (1,1)$, and $Y_{(1,1)}$ is just the Chow variety $X_{(1,1)}$. In general, we may see $Y_\bfd$ as the linear projection of the Segre variety $X_{m}$ inside $\bbP(S_{d_1}\otimes\cdots\otimes S_{d_t})$, where ${m} = (m_1,\ldots,m_t)$ with $m_i = {d_i+n \choose n} - 1$. Note that, if $\bfd, \bfd'$ are two partitions of $d$ such that $\bfd$ can be recovered from $\bfd'$ by grouping and summing some of entries, then we have the obvious inclusion $Y_{\bfd'} \subset Y_\bfd$. Therefore, if we define the {variety of reducible forms} as the union over all possible partitions $\bfd \vdash d$ of the varieties $Y_\bfd$, we can actually write:
$$
	Y = \bigcup_{k = 1}^{\left\lfloor \frac{d}{2} \right\rfloor} Y_{(d-k,k)} \subset \bbP S_d.
$$
In terms of additive decompositions, the study of varieties of reduced forms and their secant varieties is related to the notion of the {strength} of a polynomial, which was recently introduced by T. Ananyan and M. Hochster \cite{AnHo:StillmanConj} and then generalized to any tensor in \cite{BiDrEg:BoundedStrength}. 
\begin{problem}\label{prob:Strength}
	Given a homogeneous polynomial $F \in S_d$, find the smallest length of an expression $F = \sum_{i=1}^r G_{i,1}G_{i,2}$, where $1 \leq \deg(G_{i,j}) \leq d-1$.
\end{problem}
The answer to Problem \ref{prob:Strength} is called the {strength} of $F$, and we denote it by ${\rm S}(F)$. 

In \cite{7:ReducibleForms}, the authors gave a conjectural formula for the dimensions of all secant varieties $\sigma_s(Y_\bfd)$ of the variety of $\bfd$-reducible forms for any partition $\bfd$ (see Conjecture 1.1 in \cite[]{7:ReducibleForms}), and they proved it under certain numerical conditions (see Theorem 1.2 in \cite[]{7:ReducibleForms}). These computations have been made by using the classical Terracini's lemma and relating the dimensions of these secants to the famous Fr\"oberg's conjecture on the Hilbert series of generic forms. 

The variety of reducible forms is not irreducible and the irreducible component with biggest dimension is the one corresponding to the partition $(d-1,1)$, i.e., $\dim Y = \dim Y_{(d-1,1)}$. Higher secant varieties of the variety of reducible forms are still reduced, but understanding which is the irreducible component with the biggest dimension is not an easy task. In Theorem 1.5 of \cite[]{7:ReducibleForms}, the authors proved that, if $2s \leq n-1$, then the biggest irreducible component of $\sigma_s(Y)$ is $\sigma_s(Y_{(d-1,1)}$, i.e., $\dim \sigma_s(Y) = \dim \sigma_s(Y_{(d-1,1)})$, and together with the aforementioned Theorem 1.2 of \cite[]{7:ReducibleForms}, this allows us to compute the dimensions of secant varieties of varieties of reducible forms and answer Problem \ref{prob:Strength} under certain numerical restrictions (see Theorem 7.4 \cite[]{7:ReducibleForms}). 

In conclusion, we have that Problem \ref{prob:Strength} is answered in the following cases:
\begin{enumerate}
	\item any binary form ($n=2$), where $S(F) = 1$,
	
	 since every binary form is a product of linear forms;
	\item generic quadric ($d=2$), where $S(F) = \left\lfloor \frac{n}{2} \right\rfloor +1$, 
	
	since it forces $\bfd = (1,1)$, which is solved by Corollary \ref{cor: osculating ranks} (2);
	\item generic ternary cubic ($n = 2,~d=3$), where $S(F) = 2$, 
	
	since $Y_{(2,1)}$ is seven-dimensional and non-degenerate inside $\bbP^9 = \bbP S_3$, then $\sigma_2(Y_{(2,1)})$ cannot be eight-dimensional; otherwise, we get a contradiction by one of the classical Palatini's lemmas, which states that if $\dim\sigma_{s+1}(X) = \dim\sigma_s(X) + 1$, then $\sigma_{s+1}(X)$ must be a linear space \cite{Pa}.
\end{enumerate}

\subsection{Varieties of Powers}
Another possible generalization of the classical Waring problem for forms is given by the~following.
\begin{problem}\label{prob: generalized Waring}
	Given a homogeneous polynomial $F \in S_d$ and a positive divisor $k > 1$ of $d$, find the smallest length of an expression $F = \sum_{i=1}^r G_i^k$.
\end{problem}
The answer to Problem \ref{prob: generalized Waring} is called the {$k$-th Waring rank}, or simply $k$-th rank, of $F$, and we denote it $\rk^k_{d}(F)$. In this case, we need to consider the {variety of $k$-th powers}, i.e., 
$$V_{k,d} = \{[G^k] \in \bbP S_d ~|~ G \in S_{d/k}\}.$$
That is, the variety obtained by considering the composition:
\begin{equation}\label{geometric FOS}
	\pi \circ \nu_k : \bbP S_{d/k} \rightarrow \bbP S^k(S_{d/k}) \dashrightarrow \bbP S_d,
\end{equation}
where:
\begin{enumerate}
	\item if $W = S_{d/k}$, then $\nu_k$ is the $k$-th Veronese embedding of $\bbP W$ in $\bbP S^kW$;
	\item if we consider the standard monomial basis $w_\alpha = \bfx^\alpha$ of $W$, i.e., $|\alpha| = d/k$, then $\pi$ is the linear projection from $\bbP S^kW$ to $\bbP S_d$ induced by the substitution $w_\alpha \mapsto \bfx^\alpha$. In particular, we have that the center of the projection $\pi$ is given by the homogeneous part of degree $k$ of the ideal of the Veronese variety $\nu_d(\bbP^n)$.
\end{enumerate}
Problem \ref{prob: generalized Waring} was considered by Fr\"oberg, Shapiro and Ottaviani \cite{FrOtSh:GeneralizedWaring}. Their main result was that, if $F$ is generic, then:
\begin{equation}\label{eq: FOS upperbound}
	\rk_d^k(F) \leq k^n,
\end{equation}
i.e., the $k^n$-th secant variety of $V_{k,d}$ fills the ambient space. This was proven by Terracini's lemma. Indeed, for any $G,H \in S_{d/k}$, we have that:
$$
	\left.\frac{\rm d}{{\rm d}t}\right|_{t = 0}(G+tH)^k = kG^{k-1}H;
$$
therefore, we obtain that:
$$
	T_{[G^k]} V_{k,d} = \bbP \left(\langle [G^{k-1}H] ~|~ H \in S_{d/k} \rangle\right),
$$
and, by Terracini's lemma (Lemma \ref{Terracini}), if $Q$ is a generic point on $\langle [G_1^k],\ldots,[G_s^k]\rangle$, where the $G_i$'s are generic forms of degree $d/k$, then: 
\begin{equation}\label{eq: Terracini FOS}
	T_Q \sigma_sV_{k,d} = \bbP \left((G_1^{k-1},\ldots,G_s^{k-1})_d\right).
\end{equation}
In \cite[Theorem 9]{FrOtSh:GeneralizedWaring}, the authors showed that the family: 
$$
	G_{i_1,\ldots,i_n} = (x_0 + \xi^{i_1}x_1 + \ldots + \xi^{i_n}x_n)^{d/k} \in S_{d/k}, \quad \text{ for } i_1,\ldots,i_n \in \{0,\ldots,k-1\},
$$
where $\xi^k = 1$, is such that: 
$$
	(G_1^{k-1},\ldots,G_s^{k-1})_d = S_{d}.
$$
In this way, they showed that $\sigma_{k^n}(V_{k,d})$ fills the ambient space. A remarkable fact with the upperbound~\eqref{eq: FOS upperbound} is that it is independent of the degree of the polynomial, but it only depends on the power $k$. Now, the naive lower bound due to parameter counting is $\left\lceil \frac{\dim S_d}{\dim S_{d/k}} \right\rceil = \left\lceil \frac{{n+d \choose n}}{{n+d/k \choose n}}\right\rceil$, which tends to $k^n$ when $d$ runs to infinity. 

In conclusion, we obtain that the main result of \cite{FrOtSh:GeneralizedWaring} is resumed as follows.
\begin{thm}[{\cite[Theorem 4]{FrOtSh:GeneralizedWaring}}]
	Let $F$ be a generic form of degree $d$ in $n+1$ variables. Then,
	$$
		R^k_d(F) \leq k^n.
	$$
	If $d \gg 0$, then the latter bound is sharp.
\end{thm}
This result gives an asymptotic answer to Problem \ref{prob: generalized Waring}, but, in general, it is not known for which degree $d$ the generic $k$-th Waring rank starts to be equal to $k^n$, and it is not known what happens in lower degrees. 

We have explained in \eqref{geometric FOS} how to explicitly see the variety of powers $V_{k,d}$ as a linear projection of a Veronese variety $X_{N,k} = \nu_k(\bbP^N)$, where $N = {n+d/k \choose n}-1$. It is possible to prove that $\sigma_2(X_{N,k})$ does not intersect the base of linear projection, and therefore, $V_{k,d}$ is actually isomorphic to $X_{N,k}$. Unfortunately, higher secant varieties intersect non-trivially the base of the projection, and therefore, their images, i.e., the secant varieties of the varieties of powers, are more difficult to understand. However, computer experiments suggest that the dimensions are preserved by the linear projection; see \cite{LORS18} (Section~4) for more details about these computations (a {Macaulay2} script with some examples is available in the ancillary files of the arXiv version of \cite{LORS18}). In other words, it seems that we can use the Alexander--Hirschowitz theorem to compute the dimensions of secant varieties of varieties of powers and provide an answer to Problem \ref{prob: generalized Waring}. More on this conjecture is explained in \cite{LORS18}.
\begin{conj}[{\cite{LORS18}} (Conjecture 1.2)]\label{conj: FOS}
	Let $F$ be a generic form of degree $d$ in $n+1$ variables. Then,
	$$
		\rk^k_d(F) = 
		\begin{cases}
			\min\{s \geq 1 ~|~ s{n+d/2 \choose n} - {s \choose 2} \geq {n+d \choose n}\} & \text{ for } k = 2; \\
			\min\{s \geq 1 ~|~ s{n+d/k \choose n} \geq {n+d \choose n}\} & \text{ for } k \geq 3.
		\end{cases}
	$$
\end{conj}
\begin{rmk}
	The latter conjecture claims that for $k \geq 3$, the correct answer is given by the direct parameter count. For $k = 2$, we have that secant varieties are always defective. This is analogous to the fact that secant varieties to the two-fold Veronese embeddings are defective. Geometrically, this is motivated by Terracini's lemma and by the fact that:
	$$
		T_{[G^2]}V_{2,d} \cap T_{[H^2]} V_{2,d} = [GH],
	$$
	and not empty, as expected.
\end{rmk}

\begin{example}
Here, we explain how the binary case can be treated; see \cite[Theorem 2.3]{LORS18}.
By \eqref{eq: Terracini FOS}, the computation of the dimension of secant varieties of varieties of powers reduces to the computation of dimensions of homogeneous parts of particular ideals, i.e., their Hilbert functions. This relates Problem \ref{prob: generalized Waring} to some variation of Fr\"oberg's conjecture, which claims that the ideal $(G_1^k,\ldots,G_s^k)$, where the $G_i$'s are generic forms of degrees at least two, has Hilbert series equal to the right-hand side of \eqref{eq: Froberg inequality}; see \cite{Nick:HS}. In the case of binary forms, by semicontinuity, we may specialize the $G_i$'s to be powers of linear forms. In this way, we may employ the result of~\cite{GeSc}, which claims that power ideals in two variables satisfy Fr\"oberg--Iarrobino's conjecture, i.e., \eqref{eq: Froberg inequality} is actually an equality, and we conclude the proof of Conjecture \ref{conj: FOS} in the case of binary forms.
\end{example}

By using an algebraic study on the Hilbert series of ideals generated by powers of forms, we have a complete answer to Problem \ref{prob: generalized Waring} in the following cases (see \cite{LORS18}):
\begin{enumerate}
	\item binary forms $(n = 1)$, where: 
	$$
		\rk^k_d(F) = \left\lceil \frac{d+1}{d/k+1} \right\rceil;
	$$
	\item ternary forms as sums of squares $(n = 2, k = 2)$, where:
	$$
		\rk^2_d(F) = \left\lceil \frac{{d+2 \choose 2}}{{d/2+2 \choose 2}} \right\rceil,
	$$
	except for $d = 1,3,4$, where $
		\rk^2_d(F) = \left\lceil \frac{{d+2 \choose 2}}{{d/2+2 \choose 2}} \right\rceil +1.
	$;
		\item quaternary forms as sums of squares $(n = 3, k = 2)$, where:
	$$
		\rk^2_d(F) = \left\lceil \frac{{d+3 \choose 3}}{{d/2+3 \choose 3}} \right\rceil,
	$$
	except for $d = 1,2$, where $
		\rk^2_d(F) = \left\lceil \frac{{d+3 \choose 3}}{{d/2+3 \choose 3}} \right\rceil +1.
	$
\end{enumerate}

%%%%%%%%%%%%%%%%%%%%%%%%%%%%%%%%%%%%%%%%%%%%%%%%%%%%%%%%%%%%%%%%%%%%%%%%%%

\section{Beyond Dimensions}\label{sec:ByondDim}

We want to present here, as a natural final part of this work, a list of problems about secant varieties and decomposition of tensors, which are different from merely trying to determine the dimensions of the varieties $\sigma_s(X)$ for the various $X$ we have considered before. We will consider problems such as determining maximal possible ranks, finding bounds or exact values on ranks of given tensor, understanding the set of all possible minimal decompositions of a given tensor, finding equations for the secant varieties or studying what happens when working over $\mathbb{R}$. The reader should be aware of the fact that there are many very difficult open problems around these questions.

\subsection{Maximum Rank} 
A very difficult and still open problem is the one that in the Introduction we have called the ``little Waring problem''. We recall it here.
\begin{quote} 
{Which is the minimum integer $r$ such that any form can be written as a sum of $r$ pure powers of linear forms?}
\end{quote}

This corresponds to finding the maximum rank of a form of certain degree $d$ in a certain number $n+1$ of variables.

To our knowledge, the best general achievement on this problem is due to Landsberg and Teitler, who in \cite[Proposition 5.1]{LT} proved that the rank of a degree $d$ form in $n+1$ variables is smaller than or equal to ${n+d \choose d}-n$. Unfortunately, this bound is sharp only for $n=1$ if $d\geq 2$ (binary forms); in fact, for example, if $n=2$ and $d=3,4$, then the maximum ranks are known to be $5<{2+3\choose 2}-3=7$ and $7<{2+4\choose 2}-4=11$, respectively; see \cite[Theorem 40 and Theorem 44]{bgi}. Another general bound has been obtained by Jelisiejew \cite{Jel14:AnUpperBound}, who proved that, for $F \in S^d\Bbbk^{n+1}$, we have $\rk_{\rm sym}(F) \leq {n+d-1 \choose d-1} - {n+d-5 \choose d-3}$. Again, this bound is not sharp for $n \geq 2$. Another remarkable result is the one due to Blekherman and Teitler, who proved in \cite[Theorem 1]{bt} that the maximum rank is always smaller than or equal to twice the generic rank. 
\begin{rmk}
The latter inequality, which has a very short and elegant proof, holds also between maximal and generic $X$-ranks with respect to any projective variety~$X$. 
\end{rmk}

In a few cases in small numbers of variables and small degrees, exact values of maximal ranks have been given. We resume them in the following table.
\begin{center}
\begin{tabular}{c | c | c | c | c}
& $d$ & $n$ & Maximal Rank & Ref. \\
\hline
binary forms & any & $1$ & $d$ & classical, \cite{Re13:OnLengthBinaryForms} \\
quadrics & $2$ & any & $n+1$ & classical \\
plane cubics & $3$ & $2$ & $5$ & \cite{Seg42:NonSingularCubic,LT} \\
plane quartics & $4$ & $2$ & $7$ & \cite{Kl99:RepresentingPolynomials,DeP15:MaximumRankTernaryQuartics} \\
plane quintics & $5$ & $2$ & $10$ & \cite{DeP15:TernaryQuintic,BT16:SomeExamplesHighRank} \\
cubic surfaces & $3$ & $3$ & $7$ & \cite{Seg42:NonSingularCubic} \\
cubic hypersurfaces & $3$ & any & $\leq \left\lfloor \frac{d^2 + 6d + 1}{4}\right\rfloor$ & \cite{deP16:AsympthoticTerm}
\end{tabular}
\end{center}
\medskip

We want to underline the fact that it is very difficult to find examples of forms having {high rank}, in the sense {higher than the generic rank}. Thanks to the complete result on monomials in \cite{CCG12:SolutionWaringMonomials} (see Theorem~\ref{thm:monomials}), we can easily see that in the case of binary and ternary forms, we can find monomials having rank higher than the generic one. However, for higher numbers of variables, monomials do not provide examples of forms of high rank. Some examples are given in \cite{BT16:SomeExamplesHighRank}, and the spaces of forms of high rank are studied from a geometric point of view in \cite{BHMT18:LocusHighRank}.

\subsection{Bounds on the Rank}
In the previous subsection, we discussed the problem of finding the maximal rank of a given family of tensors. However, for a given specific tensor $T$, it is more interesting, and relevant, to find explicit bounds on the rank of $T$ itself. For example, by finding {good} lower and upper bounds on the rank of $T$, one can try to compute actually the rank of $T$ itself, but usually, the maximal rank is going to be too large to be useful in this direction.

One typical approach to find upper bounds is very explicit: by finding a decomposition of $T$. In the case of symmetric tensors, that is in the case of homogeneous polynomials, the apolarity lemma (Lemma \ref{lemma:Apolarity}) is an effective tool to approach algebraically the study of upper bounds: by finding the ideal of a reduced set of points $\mathbb{X}$ inside $F^\perp$, we bound the rank of $F$ from above by the cardinality of $\mathbb{X}$.

\begin{example}\label{monomialboundexample}
For $F=x_0x_1x_2$, in standard notation, we have $F^\perp=(y_0^2,y_1^2,y_3^2)$, and we can consider the complete intersection set of four reduced points $\mathbb{X}$ whose defining ideal is $(y_1^2-y_0^2, y_2^2-y_0^2)$, and thus, the rank of $F$ is at most four. Analogously, if $F = x_0x_1^2x_2^3$, we have $F^\perp = (y_0^2,y_1^3,y_3^4)$, and we can consider the complete intersection of $12$ points defined by $(y_1^3-y_0^3,y_2^4-y_0^4)$. 
\end{example}

Other upper bounds have been given by using different notions of rank.
\begin{dfn}
We say that a scheme $Z\subset \mathbb{P}^N$ is {curvilinear} if it is a finite union of schemes of the form $\mathcal{O}_{\mathcal{C}_i,P_{i}}/{\mathfrak{m}}_{P_i}^{e_i}$, for smooth points $P_i$ on reduced curves $\mathcal{C}_i\subset \mathbb{P}^N$. Equivalently, the tangent space at each connected component of $Z$ supported at the $P_i$'s has Zariski dimension $\leq 1$. The {curvilinear rank} $\rk_{\rm curv}(F)$ of a degree $d$ form $F$ in $n+1$ variables is:
$$\rk_{\rm curv}(F):=\min\left\{\deg(Z)\; | \; Z\subset X_{n,d}, \; Z \text{ curvilinear, } [F]\in \langle Z \rangle\right\}.$$
\end{dfn}
 With this definition, in \cite[Theorem 1]{bb4}, it is proven that the rank of an $F\in S^{d}\Bbbk^{n+1}$ is bounded by $(\rk_{\rm curv}(F)-1)d+2-\rk_{\rm curv}(F)$. This result is sharp if $\rk_{\rm curv}(F) =2, 3$; see \cite[Theorem 32 and Theorem 37]{bgi}.

Another very related notion of rank is the following; see \cite{bubu,bbm}.
\begin{dfn}
We define the {smoothable rank} of a form $F \in S^d\Bbbk^{n+1}$ as:
$$\rk_{\rm smooth}(F) := \min \left\{\deg (Z) ~\Big|~ 
\let\scriptstyle\textstyle\substack{ Z \text{ is a limit of smooth schemes } Z_i \text{ such that} \\ 
Z,Z_i\subset X_{n,d}, \text{ are zero-dim schemes with }\deg(Z_i)=\deg(Z), \\ 
\text{ and } [F]\in \langle Z\rangle}\right\}.$$
\end{dfn}
In \cite{bb4} (Section 2), it is proven that if $F$ is a ternary form of degree $d$, then $\rk_{\rm sym}(F) \le (\rk_{\rm smooth}(F)-1)d$. We refer to \cite{bbm} for a complete analysis on the relations between different notions of ranks. 

The use of the apolarity lemma (Lemma \ref{lemma:Apolarity}) to obtain lower bounds to the symmetric-rank of a homogeneous polynomial was first given in \cite{RS11:RankSymmetricForm}.

\begin{thm}[{\cite[Proposition 1]{RS11:RankSymmetricForm}}]\label{RSlowerbound}
If the ideal $F^\perp$ is generated in degree $t$ and $\mathbb{X}$ is a finite scheme apolar to $F$, that is $I_\mathbb{X}\subset F^\perp$, then:
\[ 
	\frac{1}{t} \deg F^\perp \leq \deg \mathbb{X}.
\]
\end{thm}
This result is enough to compute the rank of the product of variables.
\begin{example}
For $F=x_0x_1x_2$, Theorem \ref{RSlowerbound} yields:
\[
	{1 \over 2} 8\leq \deg \mathbb{X}.
\]
If we assume $\bbX$ to be reduced, i.e., $\deg \bbX = |\bbX|$, by the apolarity lemma, we get $\rk_{\rm sym}(F) \geq 4$, and thus, by Example~\ref{monomialboundexample}, the rank of $F$ is equal to four. However, for the monomial $x_0x_1^2x_2^3$, we get the lower bound of six, which does not allow us to conclude the computation of the rank, since Example \ref{monomialboundexample} gives us $12$ as the upper bound.
\end{example}
To solve the latter case, we need a more effective use of the apolarity lemma in order to produce a better lower bound for the rank; see \cite{CCG12:SolutionWaringMonomials,CCCGW18}.

\begin{thm}[{\cite{CCCGW18}} (Corollary 3.4)]\label{colonperpbound}
Let $F$ be a degree $d$ form, and let $e > 0$ be an integer. Let $I$ be any ideal generated in degree $e$, and let $G$ be a general form in $I$. For $s \gg 0$, we have:
\[
	\rk_{\rm sym}(F) \geq 
	\frac{1}{e} \sum_{i=0}^s \HF\left(R/(F^\perp : I+(G)),i\right).
\]
\end{thm}
A form for which there exists a positive integer $e$ such that the latter lower bound is actually sharp is called {$e$-computable}; see \cite{CCCGW18}. Theorem \ref{colonperpbound} was first presented in \cite{CCG12:SolutionWaringMonomials} in the special case of $e = 1$: this was the key point to prove Theorem \ref{thm:monomials} on the rank of monomials, by showing that monomials are one-computable. In order to give an idea of the method, we give two examples: in the first one, we compute the rank of $x_0x_1^2x_2^3$ by using one-computability, while in the second one, we give an example in which it is necessary to use two-computability; see \cite{CCG12:SolutionWaringMonomials,CCCGW18}.

\begin{example}\label{example:lowerBound}
Consider again $F=x_0x_1^2x_2^3$. We use Theorem \ref{colonperpbound} with $e = 1, G=y_0$ and $I=(y_0)$. Note that:
$$
	F^\perp : I + (G) = (y_0^2,y_1^3,y_2^4) : (y_0) + (y_0) = (y_0,y_1^3,y_2^4).
$$
This yields to: 
\[
	\rk_{\rm sym}(F) \geq \sum_{i=0}^s \HF\left(R/(y_0,y_1^3,y_2^4),i\right)=12, 
\]
since $\HS(R/(y_0,y_1^3,y_2^4),z) = 1 + 2z + 3z^2 + 3z^3 + 2z^4 + z^5$. Hence, by using Example \ref{monomialboundexample}, we conclude that the rank of $F$ is actually $12$.
\end{example}
 
\begin{example}
Consider the polynomial:
\[
\begin{array}{c}
F = x_0^{11} -22x_0^{9}x_1^{2} + 33x_0^{7}x_1^{4} -22x_0^{9}x_2^{2} + 396x_0^{7}x_1^{2}x_2^{2}-462x_0^{5}x_1^{4}x_2^{2} + \\ 33x_0^{7}x_2^{4} -462x_0^{5}x_1^{2}x_2^{4} + 385x_0^{3}x_1^{4}x_2^{4},
\end{array},
\]
we show that $F$ is two-computable and $\rk_{\rm sym}(F)=25$.
By direct computation, we get: $$F^\perp = ((y_0^2 +y_1^2+y_2^2)^2, G_1, G_2),$$
where $G_1=y_1^5 +y_2(y_0^2 +y_1^2+y_2^2)^2$ and $G_2=y_2^5 + y_0(y_0^2 +y_1^2+y_2^2)^2$.

Hence, by \eqref{colonperpbound}, we get:
$$
 \rk_{\rm sym}(F)\geq \frac{1}{2} \sum_{i=0}^\infty HF(T/ (F^\perp : (y_0^2 +y_1^2+y_2^2) + (y_0^2 + y_1^2+y_2^2)),i)= 25.
 $$
Moreover, the ideal $(G_1,G_2) \subset F^\perp$ is the ideal of $25$ distinct points, and thus, the conclusion follows. It can be shown that $F$ is not one-computable; see \cite{CCCGW18} (Example 4.23).
\end{example}

Another way to find bounds on the rank of a form is by using the rank of its derivatives. A first easy bound on the symmetric-rank of a homogeneous polynomial $F\in \Bbbk[x_0,\ldots, x_n]$ (where $\Bbbk$ is a characteristic zero field) is directly given by the maximum between the symmetric-ranks of its derivatives; indeed, if $F = \sum_{i=1}^r L_i^d$, then, for any $j = 0,\ldots,n$,
\begin{equation}\label{eq:DecompDerivatives}
	\frac{\partial F}{\partial x_j} = \sum_{i=1}^r \frac{\partial L_i^d}{\partial x_j} = (d-1)\sum_{i=1}^r \frac{\partial L_i}{\partial x_j} L_i^{d-1}.
\end{equation}
A more interesting bound is given in \cite{CGV16:RealComplexRankReducibleCubics}.
\begin{thm}[{\cite[Theorem 3.2]{CGV16:RealComplexRankReducibleCubics}}]\label{thm:boundCGV}
Let $1\leq p\leq n$ be an integer, and let $F\in\Bbbk[x_0,\ldots, x_n]$ be a form, where $\Bbbk$ is a characteristic zero field. Set $F_k=\frac{\partial F}{\partial x_k}$, for $0\leq k\leq n$. If:
\[
	\mathrm{rk}_{\rm sym}(F_0 +\sum_{k=1}^n\lambda_k F_k)\geq m,
\]
for all $\lambda_k\in\Bbbk$, and if the forms $F_1, F_2, \ldots ,F_p$ are linearly independent, then:
\[
	\mathrm{rk}_{\rm sym}(F) \geq m + p.
\]
\end{thm}
\noindent The latter bound was lightly improved in \cite[Theorem 2.3]{Tei15:SufficientStrassen}.

%\begin{rmk}
%	Note that the latter bound works over non-algebraically closed fields, as the {\em real} numbers.
%\end{rmk}

Formula \eqref{eq:DecompDerivatives} can be generalized to higher order differentials. As a consequence, for any $G \in S^j V^*$, with $1 \leq j \leq d-1$, we have that $\rk_{\rm sym}(F) \geq \rk_{\rm sym}(G\circ F)$, and in particular, if $F \in \langle L_1^d,\ldots,L_r^d\rangle$, we have that
$
	G\circ F \in \langle L_1^{d-j},\ldots,L_r^{d-j}\rangle.
$
Since this holds for any $G \in S^jV^*$, we conclude that the image of the $(j,d-j)$-th catalecticant matrix is contained in $\langle L_1^{d-j},\ldots,L_r^{d-j}\rangle$. Therefore, 
\begin{equation}\label{eq:lowerBoundCatalecticant}
	\rk_{\rm sym}(F) \geq \dim_\Bbbk({\rm Imm}~Cat_{j,d-j}(F)) = \mathrm{rk}~Cat_{j,d-j}(F).
\end{equation}
The latter bound is very classical and goes back to Sylvester. By using the geometry of the hypersurface $V(F)$ in $\bbP^n$, it can be improved; see \cite{LT}.
\begin{thm}[{\cite[Theorem 1.3]{LT}}]\label{thm:boundLT}
Let $F$ be a degree $d$ form with $n+1$ essential variables. Let $1 \leq j \leq d-1$. Use the convention that $\dim\emptyset = -1$. Then, the symmetric-rank of $F$ is such that:
\[
	\rk_{\rm sym}(F) \geq \mathrm{rk}~Cat_{j,d-j}(F) + \dim \Sigma_j(F) + 1,\]
where $Cat_{j,d-j}(F) $ is the $(j,d-j)$-th catalecticant matrix of $F$ and:
\[
	\Sigma_j(F)=\left\lbrace P\in V(F) \subset \bbP V : {\partial^\alpha F\over \partial \bfx^\alpha}(P)=0, \quad \forall |\alpha| \leq j \right\rbrace.\]
\end{thm}
The latter result has been used to find lower bounds on the rank of the determinant and the permanent of the generic square matrix; see \cite{LT} (Corollary~1.4).

\smallskip
The bound \eqref{eq:lowerBoundCatalecticant} given by the ranks of catalecticant matrices is a particular case of a more general fact, which holds for general tensors. 

Given a tensor $T \in V_1\otimes\ldots\otimes V_d$, there are several ways to view it as a linear map. For example, we can ``reshape'' it as a linear map $V^*_i \rightarrow V_1\otimes\ldots\otimes\widehat{V_i}\otimes\ldots\otimes V_d$, for any $i$, or as $V^*_i\otimes V^*_j \rightarrow V_1\otimes\ldots\otimes\widehat{V_i}\otimes\ldots\otimes \widehat{V_j}\otimes\ldots\otimes V_d$, for any $i \neq j$, or more in general, as: 
\begin{equation}\label{eq:flattening}
	V_{i_1}^*\otimes V^*_{i_2}\otimes\ldots\otimes V_{i_s}^* \rightarrow V_1\otimes\ldots\otimes\widehat{V_{i_1}}\otimes\ldots\otimes \widehat{V_{i_s}}\otimes\ldots\otimes V_d,
\end{equation}
for any choice of $i_1,\ldots,i_s$. All these ways of reshaping the tensor are called {flattenings}. Now, if $T$ is a tensor of rank $r$, then all its flattenings have (as matrices) rank at most $r$. In this way, the ranks of the flattenings give lower bounds for the rank of a tensor, similarly as the ranks of catalecticant matrices gave lower bounds for the symmetric-rank of a homogeneous polynomial. 

We also point out that other notions of flattening, i.e., other ways to construct linear maps starting from a given tensor, have been introduced in the literature, such as {Young flattenings} (see \cite{LO}) and {Koszul flattenings} (see \cite{OO}). These were used to find equations of certain secant varieties of Veronese and other varieties and to provide algebraic algorithms to compute decompositions.

\smallskip
We conclude this section with a very powerful method to compute lower bounds on ranks of tensors: the so-called {substitution method}. In order to ease the notation, we report the result in the case $T \in V_1\otimes V_2 \otimes V_3$, with $\dim_\Bbbk V_i = n_i$. For a general result, see \cite{AFT:LowerUpperBounds} (Appendix B).

\begin{thm}[{The substitution method \cite{AFT:LowerUpperBounds} (Appendix B) or \cite{Lan17:GeometryComplexity}} (Section 5.3)]
	Let $T \in V_1\otimes V_2 \otimes V_3$. Write $T = \sum_{i=1}^{n_1} e_i \otimes T_i$, where the $e_i$'s form a basis of $V_1$ and the $T_i$'s are the corresponding ``slices'' of the tensor. Assume that $M_{n_1} \neq 0$. Then, there exist constants $\lambda_1,\ldots,\lambda_{n_1 - 1}$ such that the tensor:
	$$
		T' = \sum_{i=1}^{n_1-1} e_i \otimes (T_i - \lambda_i T_{n_1}) \in \Bbbk^{n_1-1}\otimes V_2 \otimes V_3,
	$$
	has rank at most $\rk (T) - 1$. If $T_{n_1}$ is a matrix of rank one, then equality holds.
\end{thm}
Roughly speaking, this method is applied in an iterative way, with each of the $V_i$'s playing the role of $V_1$ in the theorem, in order to reduce the tensor to a smaller one whose rank we are able to compute. Since, in the theorem above, $\rk(T) \geq \rk(T') + 1$, at each step, we get a plus one on the lower bound. For a complete description of this method and its uses, we refer to \cite{Lan17:GeometryComplexity} (Section 5.3).

A remarkable use of this method is due to Shitov, who recently gave counterexamples to very interesting conjectures such as {Comon's conjecture}, on the equality between the rank and symmetric-rank of a symmetric tensor, and {Strassen's conjecture}, on the additivity of the tensor rank for sums of tensors defined over disjoint subvector spaces of the tensor space.

We will come back with more details on Strassen's conjecture, and its symmetric version, in the next section. We spend a few words more here on Comon's conjecture.

Given a symmetric tensor $F \in S^dV \subset V^{\otimes d}$, we may regard it as a tensor, forgetting the symmetries, and we could ask for its tensor rank, or we can take into account its symmetries and consider its symmetric-rank. Clearly, 
\begin{equation}\label{ineq:comon}
	\rk(F) \leq \rk_{\rm sym}(F).
\end{equation}
The question raised by Comon asks if whether such an inequality is actually an equality. Affirmative answers were given in several cases (see \cite{CGLM08,BB13,ZHQ16,Fri16,Sei18:RankSymmRankCubic}). In \cite{Shi18:Comon}, Shitov found an example (a cubic in $800$ variables) where the inequality \eqref{ineq:comon} is strict. As the author says, unfortunately, no symmetric analogs of this substitution method are known. However, a possible formulation of such analogs, which might lead to a smaller case where \eqref{ineq:comon} is strict, was proposed.
\begin{conj}[{\cite{Shi18:Comon}} (Conjecture 7)]
	Let $F,G \in S = \Bbbk[x_0,\ldots,x_n]$ of degree $d,d-1$, respectively. Let $L$ be a linear form. Then,
	$$
		\rk_{\rm sym}(F+LG) \geq d + \min_{L' \in S_1} \rk_{\rm sym}(F+L'G).
	$$
\end{conj}
\begin{rmk}
	A symmetric tensor $F \in S^dV$ can be viewed as a partially-symmetric tensor in $S^{d_1}V\otimes\ldots\otimes S^{d_m}V$, for any $\underline{d} = (d_1,\ldots,d_m) \in \NN^m$ such that $d_1+\ldots+d_m = d$. Moreover, if $\underline{d}' = (d'_1,\ldots,d'_{m'}) \in \NN^{m'}$ is a {refinement} of $\underline{d}$, i.e., there is some grouping of the entries of $\underline{d}'$ to get $\underline{d}$, then we have:
	\begin{equation}\label{ineq:comon_part}
		\rk_{\rm sym}(F) \geq \rk_{\underline{d}}(F) \geq \rk_{\underline{d}'}(F),
	\end{equation}
	which is a particular case of \eqref{ineq:comon}. In the recent paper \cite{GOV18}, the authors investigated the partially-symmetric version of Comon's question, i.e., the question if, for a given $F$, \eqref{ineq:comon_part} is an equality or not. Their approach consisted of bounding from below the right-hand side of \eqref{ineq:comon_part} with the simultaneous rank of its partial derivatives of some given order and then studying the latter by using classical apolarity theory (see also Remark~\ref{rmk:simul}). If such a simultaneous rank coincides with the symmetric-rank of $F$, then also all intermediate ranks are the same. In particular, for each case in which Comon's conjecture is proven to be true, then also all partially-symmetric tensors coincide. For more details, we refer to \cite{GOV18}, where particular families of homogeneous polynomials are~considered. 
\end{rmk}

\subsection{Formulae for Symmetric Ranks}
In order to find exact values of the symmetric-rank of a given polynomial, we can use one of the available algorithms for rank computations; see Section \ref{sec:algorithms}. However, as we already mentioned, the algorithms will give an answer only if some special conditions are satisfied, and the answer will be only valid for that specific form. Thus, having exact formulae working for a family of forms is of the utmost interest.

Formulae for the rank are usually obtained by finding an explicit (a posteriori sharp) upper bound and then by showing that the rank cannot be less than the previously-found lower bound.

An interesting case is the one of monomials. The lower bound of \eqref{colonperpbound} is used to obtain a rank formula for the {complex} rank of any monomial, similarly as in Example \ref{example:lowerBound}; i.e., given a monomial $F = \bfx^\alpha$, whose exponents are increasingly ordered, we have that $F^\perp = (y_0^{\alpha_0+1},\dots,y_n^{\alpha_n+1})$, and then, one has to:
\begin{enumerate}
\item first, as in Example \ref{monomialboundexample}, exhibit the set of points apolar to $F$ given by the complete intersection $(y_1^{\alpha_1}-y_0^{\alpha_1},\ldots,y_n^{\alpha_n} - y_0^{\alpha_0})$; this proves that the right-hand side of \eqref{eq:MonomialsRank} is an upper bound for the rank;
\item second, as in Example \ref{example:lowerBound}, use Theorem \ref{colonperpbound} with $e = 1$ and $G = y_0$ to show that the right-hand side of \eqref{eq:MonomialsRank} is a lower bound for the rank.
\end{enumerate}
\begin{thm}[{\cite[Proposition 3.1]{CCG12:SolutionWaringMonomials}}]\label{thm:monomials}
Let $1 \leq \alpha_0 \leq \alpha_1 \ldots\leq \alpha_n$. Then,
\begin{equation}\label{eq:MonomialsRank}
	\rk_{\rm sym}(\bfx^\alpha) = \frac{1}{\alpha_0+1}\prod_{i=0}^n (\alpha_i+1).
\end{equation}
\end{thm}

\smallskip
Another relevant type of forms for which we know the rank is the one of reduced cubic forms. The reducible cubics, which are not equivalent to a monomial (up to change of variables), can be classified into three canonical forms. The symmetric complex rank for each one was computed, as the following result summarizes: the first two were first presented in \cite{LT}, while the last one is in \cite{CGV16:RealComplexRankReducibleCubics}. In particular, for all three cases, we have that the lower bound given by Theorem \ref{thm:boundCGV} is sharp.
\begin{thm}[{\cite[Theorem 4.5]{CGV16:RealComplexRankReducibleCubics}}]\label{CGV:ReducedCubicsCC}
Let $F\in\mathbb{C}[x_0,\ldots,x_n]$ be a form essentially involving $n +1$ variables, which is not equivalent to a monomial. If $F$ is a reducible cubic form, then one and only one of the following~holds:
\begin{enumerate}
\item $F$ is equivalent to:
\[x_0(x_0^2+ x_1^2+\ldots + x_n^2),\]
and $\rk_{\rm sym}(F)=2n$.
\item $F$ is equivalent to:
\[ x_0(x_1^2+ x_2^2+ \ldots + x_n^2),\]
and $\rk_{\rm sym}(F)=2n$.
\item $F$ equivalent to:
\[x_0(x_0x_1 + x_2x_3 + x_4^2+ \ldots + x_n^2),\]
and $\rk_{\rm sym}(F)=2n +1$.
\end{enumerate}
\end{thm}

%\begin{rmk}[Strassen's Conjecture]
%In 1973 \cite{Str73:Conj}, Strassen formulated a conjecture about additivity of tensor ranks for sums of tensors which are in disjoint spaces. After a series of positive results (see e.g., \cite{FW84:OnStrassen, JT86:ValidityStrassen, LM17:AbelianTensors}), in a recent paper \cite{Shitov:Strassen}, Shitov presented counterexamples to the conjecture. However, the conjecture remains open in its symmetric version, which is stated as follows: if $F = \sum_{i=1}^s F_i$, where the $F_i$'s are forms in disjoint sets of variables, then $\rk_{\rm sym}(F) = \sum_{i=1}^s \rk_{\rm sym}(F_i)$. Following the same ideas that we just presented (with more technical algebraic computations), in \cite{CCG12:SolutionWaringMonomials}, the authors proved that Strassen's conjecture holds when the $F_i$'s are monomials. Afterwards, in \cite{CCCGW18}, Theorem \ref{colonperpbound} has been employed to prove that Strassen's conjecture holds for other classes of polynomials; see the list given in \cite[Theorem 6.1]{CCCGW18}.
%\end{rmk}

Another way to find formulae for symmetric-ranks relies on a symmetric version of {Strassen's conjecture} on tensors. In 1973, Strassen formulated a conjecture about the additivity of the tensor ranks~\cite{Str73:StrassenConj}, i.e., given tensors $T_i,\ldots,T_s$ in $V^{\otimes d}$ defined over disjoint subvector spaces, then,
\[\rk (\sum_{i=1}^s T_i)=\sum_{i = 1}^s\rk (T_i).\]
After a series of positive results (see, e.g., \cite{FW84:OnStrassen, JT86:ValidityStrassen, LM17:AbelianTensors}), Shitov gave a proof of the existence of a counter-example to the general conjecture in the case of tensors of order three \cite{Shi17:CounterexampleStrassen}. Via a clever use of the substitution method we introduced in the previous section, the author described a way to construct a counter-example, but he did not give an explicit one.

However, as the author mentioned is his final remarks, no counter example is known for the symmetric version of the conjecture that goes as follows: given homogeneous polynomials $F_1,\ldots,F_s$ in different sets of variables, then:
$$
	\rk_{\rm sym}(\sum_{i = 1}^s F_i) = \sum_{i = 1}^s \rk_{\rm sym}(F_i).
$$ 
In this case, Strassen's conjecture is known to be true in a variety of situations. The case of sums of coprime monomials was proven in \cite[Theorem 3.2]{CCG12:SolutionWaringMonomials} via apolarity theory by studying the Hilbert function of the apolar ideal of $F = \sum_{i=1}^s F_i$. Indeed, it is not difficult to prove that: 
\begin{align}\label{eq:ApolarSum}
F^\perp = \bigcap_{i=1}^s F_i^\perp + (F_i + \lambda_{i,j} F_j & ~:~ i \neq j), \\
& \text{ where the $\lambda_{i,j}$'s are suitable coefficients}. \nonumber
\end{align}
In this way, since apolar ideals of monomials are easy to compute, it is possible to express explicitly also the apolar ideal of a sum of coprime monomials. Therefore, the authors applied an analogous strategy as the one used for Theorem \ref{thm:monomials} (by using more technical algebraic computations) to prove that Strassen's conjecture holds for sums of monomials. 

In \cite{CCC:ProgressStrassen}, the authors proved that Strassen's conjecture holds whenever the summands are in either one or two variables. In \cite{CCCGW18,Tei15:SufficientStrassen}, the authors provided conditions on the summands to guarantee that additivity of the symmetric-ranks holds. For example, in \cite{Tei15:SufficientStrassen}, the author showed that whenever the catalecticant bound \eqref{eq:lowerBoundCatalecticant} (or the lower bound given by Theorem \ref{thm:boundLT}) is sharp for all the $F_i$'s, then Strassen's conjecture holds, and the corresponding bound for $\sum_{i=1}^s F_i$ is also sharp. 

A nice list of cases in which Strassen's conjecture holds was presented in \cite{CCCGW18}. This was done again by studying the Hilbert function of the apolar ideal of $F = \sum_{i=1}^s F_i$, computed as described in~\eqref{eq:ApolarSum}, and employing the bound given by Theorem \ref{colonperpbound}.

\begin{thm}[{\cite[Theorem 6.1]{CCCGW18}}]
Let $F = F_1 +\ldots+ F_m $, where the degree $d$ forms $F_i$ are in different sets of variables. If, for $i = 1,\ldots,m$, each $F_i$ is of one of the following types:
\begin{itemize}
\item $F_i$ is a monomial;
\item $F_i$ is a form in one or two variables;
\item $Fi =x_0^a(x_1^b+\ldots +x_n^b)$ with $a+1\geq b$;
\item $Fi =x_0^a(x_1^b+x_2^b)$;
\item $Fi =x_0^a(x_0^b+\ldots +x_n^b)$ with $a+1\geq b$;
\item $Fi =x_0^a(x_0^b+x_1^b+x_2^b)$;
\item $F_i=x_0^aG(x_1,\ldots,x_n)$ where $G^\perp=(H_1,\ldots,H_m)$ is a complete intersection and $a\leq \deg H_i$;
\item $F_i$ is a Vandermonde determinant;
\end{itemize}
then Strassen's conjecture holds for $F$.
\end{thm}

%%%%%%%%%%%%%%%%%%%%%%%%%%
\subsection{Identifiability of Tensors}
For simplicity, in this section we work on the field $\mathbb{C}$ of the complex numbers. Let us consider tensors in $\bfV = \mathbb{C}^{n_1+1}\otimes \ldots \otimes \mathbb{C}^{n_d+1}$. A problem of particular interest when studying minimal decompositions of tensors is to count how many there are.

\begin{problem} 
Suppose a given tensor $T \in \bfV$ has rank $r$, i.e., it can be written as $T = \bigoplus_{i=1}^r {v}_i^1\otimes \ldots \otimes {v}_i^d$. When is it that such a decomposition is {unique} (up to permutation of the summands and scaling of the vectors)? 
\end{problem}

This problem has been studied quite a bit in the last two centuries (e.g., see \cite{Hil88,Kr,Str,Syl}), and it is also of interest with respect to many applied problems (e.g., see \cite{AGHKT,AD,AMR09}). Our main references for this brief exposition are \cite{ABC,Ae}.

Let us begin with a few definitions.

\begin{dfn} A rank-$r$ tensor $T \in \bfV$ is said to be \emph{identifiable} over $\mathbb{C}$ if its presentation 
$
T = \bigoplus_{i=1}^r {v}_i^1\otimes \ldots \otimes {v}_i^d
$
is unique (up to permutations of the summands and scaling of the vectors).
\end{dfn}
\noindent It is interesting to study the identifiability of a generic tensor of given shape and~rank.
\begin{dfn}
We say that tensors in $\bfV$ are \emph{$r$-generically identifiable} over $\mathbb{C}$ if identifiability over $\mathbb{C}$ holds in a Zariski dense open subset of the space of tensors of rank $r$. Moreover, we say that the tensors in $\bfV$ are \emph{generically identifiable} if they are {$r_g$-generically identifiable}, where $r_g$ denotes the generic rank in $\bfV$.
\end{dfn}
Let us recall that the generic rank for tensors $\bfV$ is the minimum value for which there is a Zariski open non-empty set $U$ of $\bfV$ where each point represents a tensor with rank $\leq r_g$; see Section \ref{sec:TensorDec_Segre}.
%The expected generic rank $r_g$ is:
%$$
%r_g = \left\lceil \frac{\Pi_{i=1}^{d}n_i}{1+\sum_{i=1}^d(n_i-1)} \right\rceil.
%$$
Considering $\mathbf{n}= (n_1,\ldots , n_d)$, let $X_{n}\subset {\mathbb{P}}\bfV$ be the Segre embedding of $\bbP^{n_1}\times\ldots\times\bbP^{n_d}$. As we already said previously, if $r_g$ is the generic rank for tensors in $\bfV$, then $\sigma_{r_g}(X_{n})$ is the first secant variety of $X_{\bfn}$, which fills the ambient space. Therefore, to say that the tensors in $\bfV$ are generically identifiable over $\mathbb{C}$ amounts to saying the following: let $r_g$ be the generic rank with respect to the Segre embedding $X_{\bfn}$ in $\bbP\bfV$; then, for the generic point $[T] \in \mathbb{P}\bfV$, there exists a unique ${\mathbb{P}^{r-1}}$, which is $r_g$-secant to $X_{n}$ in $r_g$ distinct points and passes through $[T]$. The $r_g$ points of $X_{n}$ gives (up to scalar) the $r_g$ summands in the unique (up to permutation of summands) minimal decomposition of the tensor $T$.

When $\sigma_r(X_{n})\neq {\mathbb{P}\bfV}$, i.e., the rank $r$ is smaller than the generic one (we can say that $r$ is {sub-generic}), then we have that the set of tensors $T \in \bfV$ with rank $r$ is {r%should it be italics?
-generically identifiable} over $\mathbb{C}$ if there is an open set $U$ of $\sigma_r(X_{n})$ such that for the points $[T]$ in $U$, there exists a unique ${\mathbb{P}^{r-1}}$, which is $r$-secant in $r$ distinct points to $X_{n}$ and passes through $[T]$.

Obviously, the same problem is interesting also when treating symmetric or skew-symmetric tensors, i.e., when $n_1=\ldots = n_d=n$ and $T \in S^d(\mathbb{C}^{n+1}) \subset \bfV$ or $T \in \bigwedge^d(\mathbb{C}^{n+1})$. From a geometric point of view, in these cases, we have to look at Veronese varieties or Grassmannians, respectively, and their secant varieties, as we have seen in the previous sections.

Generic identifiability is quite rare as a phenomenon, and it has been largely investigated; in particular, we refer to \cite{BCO,BC,BCV,CO,COV14,COV17,Kr,Str}. As an example of how generic identifiability seldom presents itself, we can consider the case of symmetric tensors.

It is classically known that there are three cases of generic identifiability, namely:
\begin{itemize}
\item binary forms of odd degree ($n = 1$ and $d = 2t+1$), where the generic rank is $t+1$ \cite{Syl};
\item ternary quintics ($n = 2$ and $d = 5$), where the generic rank is seven \cite{Hil88};
\item quaternary cubics ($n = 3$ and $d = 3$), where the generic rank is five \cite{Syl}.
\end{itemize}
Recently, Galuppi and Mella proved that these are the only generically identifiable cases when considering symmetric-ranks of symmetric tensors; see \cite{GM16:Identifiability}.

When we come to partially-symmetric tensors (which are related to Segre--Veronese varieties, as we described in Section \ref{sec:SegreVeronese}), a complete classification of generically-identifiable cases is not known, but it is known that it happens in the following cases; see \cite{AGMO}.
\begin{itemize}
\item $S^{d_1}\bbC^2 \otimes \ldots\otimes S^{d_t}\bbC^2$, with $d_1 \leq \ldots d_t$ and $d_1+1 \geq r_g$, where $r_g$ is the generic partially-symmetric-rank, i.e., forms of multidegree $(d_1,\ldots,d_t)$ in $t$ sets of two variables; here, the generic partially-symmetric-rank is $t+1$ \cite{CR06:MinimalSecantDegree};
\item $S^2\bbC^{n+1}\otimes S^2\bbC^{n+1}$, i.e., forms of multidegree $(2,2)$ in two sets of $n+1$ variables; her,e the generic partially-symmetric-rank is $n+1$ \cite{Wei67:ZurBilinearenFormen};
\item $S^2\bbC^3 \otimes S^2\bbC^3 \otimes S^2\bbC^3 \otimes S^2\bbC^3$, i.e., forms of multidegree $(2,2,2,2)$ in four sets of three variables; here, the generic partially-symmetric-rank is four (this is a classical result; see also \cite{AGMO});
\item $S^2\bbC^3 \otimes S^3\bbC^3$, i.e., forms of multidegree $(2,3)$ in two sets of three variables; here, the generic partially-symmetric-rank is four \cite{Ro89:NoteCubicConic};
\item $S^2\bbC^3 \otimes S^2\bbC^3 \otimes S^4\bbC^3$, i.e., forms of multidegree $(2,2,4)$ in three sets of three variables; here, the generic partially-symmetric-rank is seven \cite{AGMO}.
\end{itemize}

When considering $r$-generically identifiable tensors for sub-generic rank, i.e., for $r < r_g$, things change completely, in as much as we do expect $r$-generically-identifiability in this case. Again, the symmetric case is the best known; in \cite[Theorem 1.1]{COV17}, it was proven that every case where $r$ is a sub-generic rank and $\sigma_r(X_{n,d})$ has the expected dimension for the Veronese variety $X_{n,d}$, $r$-generically identifiability holds with the only following exceptions:
\begin{itemize}
\item $S^6\bbC^3$, i.e., forms of degree six in three variables, having rank nine;
\item $S^4\bbC^4$, i.e., forms of degree four in four variables, having rank eight;
\item $S^3\bbC^6$, i.e., forms of degree three in six variables, having rank nine.
\end{itemize}
In all the latter cases, the generic forms have exactly two decompositions.

Regarding generic identifiability for skew-symmetric tensors, there are not many studies, and we refer to \cite{BV}.

It is quite different when we are in the {defective} cases, namely, when we want to study $r$-generic identifiability and the $r$-th secant variety is defective. In this case, non-identifiability is expected; in particular, we will have that the number of decompositions for the generic tensor parametrized by a point of $\sigma_r(Y)$ is infinite.

\subsection{Varieties of Sums of Powers}
Identifiability deals with the case in which tensors have a unique (up to permutation of the summands) decomposition. When the decomposition is not unique, what can we say about all possible decompositions of the given tensor? In the case of symmetric tensors, that is homogeneous polynomial, an answer is given by studying {varieties of sums of powers}, the so-called $\VSP$, defined by Ranestad and Schreyer in \cite{RS00}.

\begin{dfn} 
Let $F$ be a form in $n+1$ variables having Waring rank $r$, and let $\mathrm{Hilb}_{r}\mathbb{P}^n$ be the Hilbert scheme of $r$ points in $\mathbb{P}^n$; we define:
\[
\mathrm{VSP}(F,r)=\overline{\lbrace \mathbb{X} = \{P_1,\ldots,P_r\} \in\mathrm{Hilb}_{r}\mathbb{P}^n:I_\mathbb{X} = \wp_1\cap \ldots \cap \wp_r\subset F^\perp\rbrace}.
\]
\end{dfn}

For example, when identifiability holds, $\VSP$ is just one single point. It is interesting to note that, even for forms having generic rank, the corresponding $\VSP$ might be quite big, as in the case of binary forms of even degree.

Using Sylvester's algorithm we have a complete description of $\VSP$ for binary forms, and it turns out to be always a linear space.

\begin{example} 
Consider the binary form $F=x_0^2x_1^2$. Since $F^\perp=(y_0^3,y_1^3)$, by Sylvester's algorithm, we have that the rank of $F$ is three. Moreover, by the apolarity lemma, we have that $\VSP(F,3)$ is the projectivization of the vector space $W = \langle y_0^3,y_1^3\rangle$ because the generic form in $W$ has three distinct roots.
\end{example}

In general, the study of VSPs%define if appropriate; italics or not? please check the conventions throughout
 is quite difficult, but rewarding: VSPs play an important role in classification work by Mukai see \cite{Mukai1,Mukai2,Mukai3}. For a review of the case of general plane curves of degree up to ten, that is for general ternary forms of degree up to ten, a complete description is given in \cite{RS00}, including results from Mukai and original results. We summarize them in the following. 

\begin{thm}[{\cite[Theorem 1.7]{RS00}}]
Let $F\in S^d\bbC^3$ be a general ternary cubic with $d = 2t - 2, 2 \leq t \leq 5$, then:
\[\VSP\left(F, {t+1\choose 2}\right)\simeq G(t, V, \eta) = \{E \in G(t,V) ~|~ \wedge^2E \subset \eta\}
\]
where $V$ is a $2t + 1$-dimensional vector space and $\eta$ is a net of alternating forms $\eta: \Lambda^2 V \rightarrow\mathbb{C}^3$ on $V$. Moreover:
\begin{itemize}
\item if $F$ is a smooth plane conic section, then $\VSP(F,3)$ is a Fano three-fold of index two and degree five in $\mathbb{P}^6$.
\item if $F$ is a general plane quartic curve, then $\VSP(F, 6)$ is a smooth Fano three-fold of index one and genus $12$ with anti-canonical embedding of degree $22$;
\item if $F$ is a general plane sextic curve, then $\VSP(F, 10)$ is isomorphic to the polarized $K3$-surface of genus $20$;
\item if $F$ is a general plane octic curve, then $\VSP(F,15)$ is finite of degree $16$, i.e., consists of $16$ points.
\end{itemize}
\end{thm}

Very often, for a given specific form, we do not have such a complete description, but at least, we can get some relevant information, for example about the dimension of the VSP: this is the case for~monomials.

\begin{thm}[{\cite[Theorem 2]{MR3017012}}]
Let $F\in \bbC[x_0, \ldots , x_n]$ be a monomial $F = \bfx^\alpha$ with exponents $0 < \alpha_0 \leq \ldots \leq \alpha_n$. Let $A = \bbC[y_0,\ldots, y_n]/(y_1^{\alpha_1+1},\ldots,y_n^{\alpha_n+1})$. Then, $\VSP(F,\rk_{\rm sym}(F))$ is irreducible and: $$\dim \VSP(F,{\rk}_{\rm sym}(F)) = \sum_{i=1}^n\HF(A;d_i - d_0).$$
\end{thm}

A complete knowledge of $\VSP(F,r)$ gives us a complete control on {all} sums of powers decompositions of $F$ involving $r$ summands. Such a complete knowledge comes at a price: a complete description of the variety of sums of powers might be very difficult to obtain. However, even less complete information might be useful to have and, possibly, easier to obtain. One viable option is given by {Waring loci} as defined in \cite{MR3658727}.

\begin{dfn}
The {Waring locus} of a degree $d$ form $F \in S^dV$ is:
\[
\mathcal{W}_F=\lbrace [L]\in\mathbb{P}V ~:~ F=L^d+L_2^d+\ldots+L_r^d,~ r = \rk_{\rm sym}(F)\rbrace,
\]
i.e., the space of linear form which appears in some {minimal} sums of powers decomposition of $F$. The {forbidden locus} of $F$ is defined as the complement of the Waring locus of $F$, and we denote it by $\mathcal{F}_F$.
\end{dfn}

\begin{rmk}
	In this definition, the notion of {essential variables} has a very important role; see Remark \ref{variables}. In particular, it is possible to prove that if $F \in \bbC[x_0,\ldots,x_n]$ has less than $n+1$ essential variables, say $x_0,\ldots,x_m$, then for any minimal decomposition $F = \sum_{i=1}^r L_i$, the $L_i$'s also involve only the variables $x_0,\ldots,x_m$. For this reason, if in general, we have $F \in S^dV$, which has less than $\dim_\bbC V$ essential variables, say that $W \subset V$ is the linear span of a set of essential variables, then $\mathcal{W}_F \subset \bbP W$. 
\end{rmk}

In \cite{MR3658727}, the forbidden locus, and thus the Waring locus, of several classes of polynomials was computed. For example, in the case of monomials, we have the following description.

\begin{thm}[{\cite[Theorem 3.3]{MR3658727}}] 
If $F = \bfx^\alpha$, such that the exponents are increasingly ordered and $m = \min_i\{\alpha_i = \alpha_0\}$, then:
\[
	\mathcal{F}_F = V(y_0\cdots y_m).
\]
\end{thm}

The study of Waring and forbidden loci can have a two-fold application. One is to construct {step-by-step} a minimal decomposition of a given form $F$: if $[L]$ belongs to $\mathcal{W}_F$, then there exists some coefficient $\lambda \in \bbC$ such that $F' = F+\lambda L^d$ has rank smaller than $F$; then, we can iterate the process by considering $\mathcal{W}_{F'}$. In \cite{MO18}, this idea is used to present an algorithm to find minimal decompositions of forms in any number of variables and of any degree of rank $r \leq 5$ (the analysis runs over all possible configuration of $r$ points in the space whose number of possibilities grows quickly with the rank); a {Macaulay2} package implementing this algorithm can be found in the ancillary files of the arXiv version of \cite{MO18}. A second possible application relates to the search for forms of high rank: if $[L]$ belongs to $\mathcal{W}_F$, then the rank of $F + \lambda L^d$ cannot increase as $\lambda$ varies, but conversely, if the rank of $F+\lambda L^d$ increases, then $[L]$ belongs to the forbidden locus of $F$. Unfortunately, it is not always possible to use elements in the forbidden locus to increase the rank of a given form; however, this idea can give a place to look for forms of high rank. For example, since $\mathcal{F}_{x_0x_1x_2}=V(y_0y_1y_2)$, that is the forbidden locus of the monomial $F=x_0x_1x_2$ is the union of the three coordinate lines of $\bbP^2$, the only {possible} way to make the rank of $F$ increase is to consider $F+\lambda L^3$ where $L$ is a linear form not containing at least one of the variables. However, as some computations can show, the rank of $F+\lambda L^3$ does {not} increase for any value of $\lambda$ and for any choice of $L$ in the forbidden locus.

Another family of polynomial for which we have a complete description of Waring and forbidden loci are binary forms.

\begin{thm}[{\cite[Theorem 3.5]{MR3658727}}]
Let $F$ be a degree $d$ binary form, and let $G \in F^\perp$ be an element of minimal degree. Then,
\begin{itemize}
\item if $\mathrm{rk}(F) <
{d+1\choose 2}$, then $\mathcal{W}_F=V(G)$;
\item if $\mathrm{rk}(F) >
{d+1 \choose 2}$, then $\mathcal{F}_F=V(G)$;
\item if $\mathrm{rk}(F) =
{d+1 \choose 2}$ and $d$ is even, then $\mathcal{F}_F$ is finite and not empty; \\
if $\mathrm{rk}(F) =
{d+1 \choose 2}$ and $d$ is odd, then $\mathcal{W}_F=V(G)$.
\end{itemize}
\end{thm}

\begin{example}
The result about Waring and forbidden loci of binary forms can be nicely interpreted in terms of rational normal curves. If $F=x_0x_1^2$, then:
\[\mathcal{F}_F=\lbrace [y_1^3] \rbrace,\]
and this means that any plane containing $[F]$ and $[y_1^3]$ is {not} intersecting the twisted cubic curve in three distinct points; indeed, the line spanned by $[F]$ and $[y_1^3]$ is tangent to the twisted cubic curve.
\end{example}
Other families of homogeneous polynomials for which we have a description of Waring and forbidden loci are quadrics \cite{MR3658727} (Corollary 3.2) (in which case, the forbidden locus is given by a quadric) and plane cubics \cite{MR3658727} (Section~3.4). We conclude with two remarks coming from the treatment of the latter case:
\begin{itemize}
\item in all the previous cases, the Waring locus of $F$ is always either closed or open. However, this is not true in general. In fact, for the cusp $F=x_0^3+x_1^2x_2$, we have that the Waring locus is: 
 \[
 	\mathcal{W}_F=\{[y_0^3]\}\cup\{[(ay_1+by_2)^3]~:~a,b\in\mathbb{C}\mbox{ and } b\neq 0\},\]
 that is the Waring locus is given by two disjoint components: a point and a line minus a point; therefore, $\mathcal{W}_F$ is neither open nor closed in $\mathbb{P}^2$;
\item since the space of minimal decompositions of forms of high rank is high dimensional, it is expected that the Waring locus is very large, and conversely, the forbidden locus is reasonably small. For example, the forbidden locus of the maximal rank cubic $F=x_0(x_1^2+x_0x_2)$ is just a point, i.e., 
  \[
  \mathcal{F}_F=\{[y_0^3]\}.
  \]
There is an open conjecture stating that, for any form $F$, the forbidden locus $\mathcal{F}_F$ is {not} empty.
\end{itemize}

%%%%%%%%%%%%%%%%%%%%%%%%%%%%%%%%%%%%%%%%%%%%%%%%%%%%%%%%

\subsection{Equations for the Secant Varieties}
A very crucial problem is to find equations for the secant varieties we have studied in the previous sections, mainly for Veronese, Segre and Grassmann varieties. Notice that having such equations (even equations defining only set-theoretically the secant varieties in question) would be crucial in having methods to find border ranks of tensors.

\subsubsection{Segre Varieties}
Let us consider first Segre varieties; see also \cite{CGG7}. In the case of two factors, i.e., $X_{\mathbf{n}}$ with $\mathbf{n} =(n_1,n_2)$, the Segre variety, which is the image of the embedding:
$$\mathbb{P}^{n_1} \times \mathbb{P}^{n_2} \rightarrow X_{\mathbf{n}} \subset \mathbb{P}^{N} ,\quad N = (n_1 + 1)(n_2 + 1)-1,$$
corresponds to the variety of rank one matrices, and $\sigma_s(X_{\mathbf{n}})$ corresponds to the variety of rank $s$ matrices, which is defined by the ideal generated by the $(s + 1)\times (s + 1)$ minors of the {generic} $(n_1 + 1)\times (n_2 + 1)$ matrix, whose entries are the homogenous coordinates of $\mathbb{P}^N$. In this case, the ideal is rather well understood; see, e.g., \cite{L} and also the extensive bibliography given in the book of Weyman \cite{W}.

We will only refer to a small part of this vast subject, and we recall that the ideal $I_{\sigma_s(X_{\mathbf{n}})}$ is a perfect ideal of height $(n_1+1-(s + 1)-1)\times (n_2+1-(s + 1)-1) = (n_1-s-1) \times (n_2-s-1)$ in the polynomial ring with $N + 1$ variables, with a very well-known resolution: the {Eagon--Northcott complex}. It follows from this description that all the secant varieties of the Segre embeddings of a product of two projective spaces are {arithmetically Cohen--Macaulay} varieties. Moreover, from the resolution, one can also deduce the degree, as well as other significant geometric invariants, of these varieties. A determinantal formula for the degree was first given by Giambelli. There is, however, a reformulation of this result, which we will use (see, e.g., \cite{Harris} (p. 244) or \cite[Theorem 6.5]{BrCo}), where this lovely reformulation of Giambelli's Formula is attributed to Herzog and Trung:
$$
 \deg(\sigma_s(X_{(n_1,n_2)})) = \prod_{i=0}^{n_1-s}\frac{{{n_2+1+i\choose s}}}{{{s+i\choose s}} }
.$$
Let us now pass to the case of the Segre varieties with more than two factors. Therefore, let $X_{\mathbf{n}}\subset \mathbb{P}^N$ with $\mathbf{n} = (n_1,\ldots ,n_t)$, $N = \Pi_{i=1}^t(n_i + 1)-1$ and $t \geq 3$, where we usually assume that $n_1 \geq \ldots \geq n_t$.

If we let $T$ be the {generic} $(n_1 +1)\times\ldots\times(n_t +1)$ tensor whose entries are the homogeneous coordinates in $\mathbb{P}^N$, then it is well known that the ideal of $X_{\mathbf{n}}$ has still a determinantal representation, namely it is generated by all the $2 \times 2$ ``minors'' of $T$, that is the $2\times 2$ minors of the flattenings of $T$. It is natural to ask if the flattenings can be used also to find equations of higher secant varieties of $X_{\mathbf{n}}$. If we split ${1,\ldots, t}$ into two subsets, for simplicity say ${1,\ldots, \ell}$ and ${\ell+1,\ldots, t}$, then we can form the~composition:
$$\nu_{(n_1,\ldots,n_\ell)} \times \nu_{(n_{\ell+1},\ldots,n_t)} : (\mathbb{P}^{n_1}\times \ldots \times \mathbb{P}^{n_\ell} ) \times (\mathbb{P}^{n_\ell +1} \times \ldots \times \mathbb{P}^{n_t}) \rightarrow \mathbb{P}^a \times \mathbb{P}^b,$$
where $a = \Pi_{i=1}^\ell(n_i + 1)-1$, $b=\Pi_{i=\ell+1}^t(n_i + 1)-1$, followed by:
$$ \nu_{1,1}: \mathbb{P}^a \times \mathbb{P}^b \rightarrow \mathbb{P}^N, \quad \quad N\quad {\rm as\ above}. $$
Clearly $X_{\mathbf{n}} \subset \nu_{1,1}(\mathbb{P}^a \times \mathbb{P}^b)$, and hence,
$\sigma_s(X_{\mathbf{n}}) \subset \sigma_s(\nu_{1,1}(\mathbb{P}^a \times \mathbb{P}^b))$.

Thus, the $(s + 1)\times (s + 1)$ minors of the matrix associated with the embedding $\nu_{1,1}$ will all vanish on $\sigma_s(X_{\mathbf{n}})$. That matrix, written in terms of the coordinates of the various $\mathbb{P}^{n_i}$, is what we have called a {flattening} of the tensor $T$.

As we have seen in \eqref{eq:flattening}, we can perform a flattening of $T$ for every partition of ${1,\ldots,t}$ into two subsets. The $(s + 1) \times (s + 1)$ minors of all of these flattenings will give us equations that vanish on $\sigma_s(X_{\mathbf{n}})$. In \cite{GSS}, it was conjectured that,
at least for $s = 2$, these equations are precisely the generators for the ideal $I_{\sigma_2(X_{\mathbf{n}})}$ of $\sigma_2(X_{\mathbf{n}})$. The conjecture was
proven in \cite{LM} for the special case of $t = 3$ (and set theoretically for all $t$'s). Then, Allman and Rhodes \cite{ar} proved the conjecture for up to five factors, while Landsberg and Weyman \cite{LaWe07} found the generators for the defining ideals of secant varieties for the Segre varieties in the following cases: all secant
varieties for $\mathbb{P}^1 \times \mathbb{P}^m \times \mathbb{P}^n$ for all $m, n$; the secant line varieties of the Segre varieties with four factors; the secant
plane varieties for any Segre variety with three factors. The proofs use representation theoretic methods.

Note that for $s > 2$, one cannot expect, in general, that the ideals $I_{\sigma_s(X_{\mathbf{n}})}$ are generated by the $(s+1)\times(s+1)$ minors of flattenings of $T$. Indeed, in many cases, there are no such minors, e.g., it is easy to check that if we consider $\mathbf{n}=(1,1,1,1,1)$ and $X_{\mathbf{n}}\subset \mathbb{P}^{31}$, we get that the flattenings can give only ten $4\times 8$ matrices and five $2\times 16$ matrices. Therefore, we get quadrics, which generate $I_{X_{\mathbf{n}}}$, and quartic forms, which are zero on $\sigma_3(X_{\mathbf{n}})$, but no equations for $\sigma_4(X_{\mathbf{n}})$ or $\sigma_5(X_{\mathbf{n}})$, which, by a simple dimension count, do not fill all of $\mathbb{P}^{31}$.

There is a particular case when we know that the minors of a single flattening are enough to generate the ideal $I(\sigma_s(X_{\mathbf{n}}))$, namely the {unbalanced} case we already met in Theorem \ref{unbalanced}, for which we have the following result.

\begin{thm}[\cite{CGG7}]
Let $X = X_{\mathbf{n}} \subset \PP ^M$ with $M=\Pi _{i=1} ^t ( n_i +1)-1$; let $N=\Pi _{i=1} ^{t-1} ( n_i +1)-1$; and let $Y_{(N,n_t)}$ be the Segre embedding of $\mathbb{P}^N\times \mathbb{P}^{n_t}$ into $\mathbb{P}^M$. Assume $n_t > N - \sum _{i=1}^{t-1}n_i +1$.
Then, for:
$$N - \sum _{i=1} ^{t-1} n_i +1 \leq s \leq \min \{ n_t,N\},$$
we have that $\sigma_s(X) = \sigma_s(Y) \neq \mathbb{P}^M$, and its ideal is generated by the $(s+1)\times (s + 1)$ minors of an $(n_t + 1) \times (N + 1)$ matrix of indeterminates, i.e., the flattening of the generic tensor with respect to the splitting $\{1,\ldots,t-1\} \cup \{t\}$.
\end{thm}

We can notice that in the case above, $X$ is defective for $N - \sum _{i=1} ^{t-1} n_i +1 < s$; see Theorem \ref{unbalanced}, while when equality holds, $\sigma_s(X)$ has the expected dimension. Moreover, in the cases covered by the theorem, $\sigma_s(X)$ is arithmetically Cohen--Macaulay, and a minimal free resolution of its defining ideal is given by the Eagon--Northcott complex.

\subsubsection{Veronese Varieties}
Now, let us consider the case of Veronese varieties. One case for which we have a rather complete information about the ideals of their higher secant varieties is the family of rational normal curves, i.e., the Veronese embeddings of $\PP^1$. In this case, the ideals in question are classically known; in particular, the ideal of $\sigma_s(X_{1,d})$ is generated by the $(s+1)\times (s+1)$ minors of catalecticant matrices associated with generic binary form of degree $d$, whose coefficients corresponds to the coordinates of the ambient coordinates. Moreover, we also know the entire minimal free resolution of these ideals, again given by the Eagon--Northcott complex.

Since the space of quadrics can be associated with the space of symmetric matrices, a similar analysis as the one done for the Segre varieties with two factors can be done in the case $d = 2$. In particular, the defining ideals for the higher secant varieties of the quadratic Veronese embeddings of $\PP ^n$, i.e., of $\sigma_s(X_{n,2})$, are defined by the $(s+1)\times (s+1)$ minors of the generic symmetric matrix of size $(n+1)\times(n+1)$.

For any $n,d$, the ideal of $\sigma_2(X_{n,d})$ is considered in \cite{K}, where it is proven that it is
generated by the $3\times 3$ minors of the first two catalecticant matrices of the generic polynomial of degree $d$ in $n+1$ variables; for the ideal of the $3\times 3$ minors, see also \cite{Rai}.

In general for all $\sigma_{s}(X_{n,d})$'s, these kinds of equations, given by $(s+1)\times (s+1)$ minors of catalecticant matrices (\cite{CatJ}), are known, but in most of the cases, they are not enough to generate the whole ideal.

Notice that those catalecticant matrices can also be viewed this way (see \cite{Pa}): consider a generic symmetric tensor (whose entries are indeterminates) $T$; perform a flattening of $T$ as we just did for generic tensors and Segre varieties; erase from the matrix that is thus obtained all the repeated rows or columns. What you get is a generic catalecticant matrix, and all of them are obtained in this way, i.e., those equations are the same as you get for generic tensors, symmetrized.

Only in a few cases, our knowledge about the equations of secant varieties of Veronese varieties is complete; see for example \cite{K, IaKa, CGG7, LO}. All recent approaches employ representation theory and the definition of {Young flattenings}. We borrow the following list of known results from \cite{LO}.

\begin{center}
{\begin{tabular}{c | l | c | c}
	\hline
	$\sigma_s(X_{n,2})$ & size $s+1$ minors generic symmetric matrix & ideal & classical \\
	\hline
	$\sigma_s(X_{1,d})$ & size $s+1$ minors of any generic catalecticant & ideal & classical \\
	\hline
	$\sigma_2(X_{n,d})$ & size $3$ minors of & ideal & \cite{K} \\
	& \quad generic $(1,d-1)$ and $(2,d-2)$-catalecticants & & \\
	\hline
	$\sigma_3(X_{n,3})$ & Aronhold equation + size $4$ minors of & ideal & Aronhold ($n=2$) \cite{IaKa} \\
	& \quad generic $(1,2)$-catalecticant & & \cite{LO} \\
	\hline
	$\sigma_3(X_{n,d})$ & size $4$ minors of & scheme & \cite{schreyer} ($n = 2, d = 4$) \\
	\quad ($d \geq 4$) & generic $(1,3)$ and $(2,2)$-catalecticant & & \cite{LO} \\
	\hline
	$\sigma_4(X_{2,d})$ & size $5$ minors of & scheme & \cite{schreyer} ($d = 4$) \\
	& \quad generic $(\left\lfloor \frac{d}{2} \right\rfloor,\left\lceil \frac{d}{2} \right\rceil)$-catalecticant & & \cite{LO} \\
	\hline
	$\sigma_5(X_{2,d})$ & size $6$ minors of & scheme & Clebsch ($d=4$) \cite{IaKa}\\
	($d \geq 6$, $d = 4$) & \quad generic $(\left\lfloor \frac{d}{2} \right\rfloor,\left\lceil \frac{d}{2} \right\rceil)$-catalecticant & & \cite{LO} \\
	\hline
	$\sigma_s(X_{2,5})$ & size $2s+2$ sub-Pfaffians of & irred.%define if appropriate
	comp.%define if appropriate
	& \cite{LO} \\
	$\quad s \leq 5$ & \quad generic Young $((31),(31))$-flattening & \\
	\hline
	$\sigma_6(X_{2,5})$ & size $14$ sub-Pfaffians of & scheme & \cite{LO} \\
	$\quad s \leq 5$ & \quad generic Young $((31),(31))$-flattening & \\	
	\hline
	$\sigma_6(X_{2,d})$ & size $7$ minors of & scheme & \cite{LO} \\
	& \quad generic $(\left\lfloor \frac{d}{2} \right\rfloor,\left\lceil \frac{d}{2} \right\rceil)$-catalecticant & \\
	\hline
	$\sigma_7(X_{2,6})$ & symmetric flattenings + Young flattenings & irred. comp. & \cite{LO} \\
	\hline
	$\sigma_8(X_{2,6})$ & symmetric flattenings + Young flattenings & irred. comp. & \cite{LO} \\
	\hline
	$\sigma_9(X_{2,6})$ & determinant of generic $(3,3)$-catalecticant & ideal & classical \\
	\hline
	$\sigma_s(X_{2,7})$ & size $(2s+2)$ sub-Pfaffians of & irred. comp. & \cite{LO} \\
	 ($s \leq 10$) & \quad generic $((4,1),(4,1))$-Young flattening & & \\
	\hline
	$\sigma_s(X_{2,2m})$ & rank of $(a,d-a)$-catalecticant $\leq \min\left\{s,{a+2 \choose 2}\right\}$ & scheme & \cite{IaKa,LO} \\
	 ($s \leq {m+1 \choose 2}$) & \quad for $1 \leq a \leq m$, open and closed & & \\
		\end{tabular}
}

{\begin{tabular}{c | l | c | c}
	$\sigma_s(X_{2,2m+1})$ & rank of $(a,d-a)$-catalecticant $\leq \min\left\{s,{a+2 \choose 2}\right\}$ & scheme & \cite{IaKa,LO} \\
	 ($s \leq {m+1 \choose 2}+1$) & \quad for $1 \leq a \leq m$, open and closed & & \\
	\hline
	$\sigma_s(X_{n,2m})$ & size $s+1$ minors of & irred. component & \cite{IaKa,LO} \\
	 ($s \leq {m+n-1 \choose n}$) & \quad generic $(m,m)$-catalecticant & & \\	
	\hline	
	$\sigma_s(X_{2,2m+1})$ & size ${n \choose \lfloor n/2 \rfloor}j+1$ minors of a Young flattening & irred. comp. & \cite{IaKa,LO} \\
	\quad ($s \leq {m+n \choose n}$) & & & \\
\end{tabular}
}
\end{center}
\medskip

 Note that the knowledge of equations that define the $\sigma_{s}(X_{n,d})$'s, also just
set-theoretically, would give the possibility to compute the symmetric border rank for any tensor in $S^{d}V$.

For the sake of completeness, we mention that equations for secant varieties in other cases can be found in \cite{LO} (Grassmannian and other homogeneous varieties), in \cite{CGG7} (Segre--Veronese varieties and Del Pezzo surfaces) and in \cite{BGL} (Veronese re-embeddings of varieties). 

\subsection{The Real World}
For many applications, it is very interesting to study tensor decompositions over the real numbers.

The first thing to observe here is that since $\bbR$ is not algebraically closed, the geometric picture is much more different. In particular, a first difference is in the definition of secant varieties, where, instead of considering the closure in the Zariski topology, we need to consider the Euclidean topology. In this way, we have that open sets are no longer dense, and there is not a definition of ``generic rank'': if $X$ is variety in $\bbP^N_{\bbR}$, the set $U_r(X) = \{P \in \bbP_\bbR^N ~|~ \rk_X(P) = r\}$ might be non-empty interior for several values of $r$. Such values are called the {typical ranks} of $X$. 

It is known that the minimal typical (real) rank of $X$ coincides with the generic (complex) rank of the complexification $X \otimes \bbC$; see \cite[Theorem 2]{bt}.

The kind of techniques that are used to treat this problem are sometimes very different from what we have seen in the case of algebraically-closed fields. 
%A classical reference for who is interested in start studying this topic is \cite{Rez}.

\smallskip
Now, we want to overview the few known cases on real symmetric-ranks. 

The typical ranks of binary forms are completely known. Comon and Ottaviani conjectured in~\cite{CO12:TypicalBinary} that typical ranks take all values between $\left\lfloor \frac{d+2}{2} \right\rfloor$ and $d$. The conjecture was proven by Blekherman in \cite{Ble15:Typical}.

In \cite{BBO18}, the authors showed that any value between the minimal and the maximal typical rank is also a typical rank. Regarding real symmetric-ranks, they proved that: the typical real rank of ternary cubics is four; the typical ranks of quaternary cubics are only five and six; and they gave bounds on typical ranks of ternary quartics and quintics.

Another family of symmetric tensors for which we have some results on real ranks are monomials. First of all, note that the apolarity lemma (Lemma \ref{lemma:Apolarity}) can still be employed, by making all algebraic computations over the complex number, but then looking for ideals of reduced points apolar to the given homogeneous polynomial and having only real coefficients. This was the method used in \cite{BCG:RealRankBinaryMonomials} to compute the real rank of binary monomials. 

Indeed, if $M = x_0^{\alpha_1}x_1^{\alpha_1}$, then, as we have already seen, $M^\perp = (y_0^{\alpha_0+1},y_1^{\alpha_1+1})$. Now, Waring decompositions of $M$ are in one-to-one relation with reduced sets of points (which are principal ideals since we are in $\bbP^1$), whose ideal is contained in $M^\perp$. Now, if we only look for sets of points that are also completely real, then we want to understand for which degree $d$ it is possible to find suitable polynomials $H_0$ and $H_1$ such that $G = y_0^{\alpha_0+1}H_0 + y_1^{\alpha_1+1}H_1$ is of degree $d$ and have only distinct real~roots.

The authors observe the following two elementary facts, which hold for any univariate $g(y) = c_dy^d+c_{d-1}y^{d-1}+\ldots+c_1y+c_0$, as a consequence of the classic Descartes' rule of signs:
\begin{itemize}
	\item if $c_i = c_{i-1} = 0$, for some $i = 1,\ldots,d$, then $f$ does not have $d$ real distinct roots; see \cite{BCG:RealRankBinaryMonomials} (Lemma~4.1);
	\item for any $j = 1,\ldots,d-1$, there exists $c_i$'s such that $f$ has $d$ real distinct roots and $c_j = 0$; see \cite{BCG:RealRankBinaryMonomials} (Lemma 4.2).
\end{itemize}
As a consequence of this, we obtain that:
\begin{itemize}
	\item if $d < \alpha_0+\alpha_1$, then $G$ (or rather, its dehomogenization) has two consecutive null coefficients; hence, it cannot have $d$ real distinct roots;
	\item if $d = \alpha_0+\alpha_1$, then only the coefficient corresponding to $y_0^{\alpha_0}y_1^{\alpha_1}$ of $G$ (or rather its dehomogenization) is equal to zero; hence, it is possible to find $H_0$ and $H_1$ such that $G$ has $d$ real distinct roots.
\end{itemize}
Therefore, we get the following result.
\begin{thm}[{\cite[Proposition 3.1]{BCG:RealRankBinaryMonomials}}]
If $\alpha_0,\alpha_1$ are not negative integers, then 
$\rk_{\rm sym}^{\mathbb{R}}(x_0^{\alpha_0}x_1^{\alpha_1}) =\alpha_0+\alpha_1.$
\end{thm}
Note that, comparing the latter result with Theorem \ref{thm:monomials}, we can see that for binary monomials, the real and the complex rank coincide if and only if the least exponent is one. However, this is true in full generality, as shown in \cite{CKOV:RealRankMonomials}.
\begin{thm}[{\cite[Theorem 3.5]{CKOV:RealRankMonomials}}]
Let $M = x_0^{\alpha_0}\cdots x_n^{\alpha_n}$ be a degree d%should it be italics?
 monomial with $\alpha_0 = \min_i\{\alpha_i\}$. Then,
\[\rk_{\rm sym}^\mathbb{R}(M) = \rk_{\rm sym}^\mathbb{C}(M) \quad \text{ if and only if } \quad \alpha_0 = 1.\]
\end{thm}
Note that the real rank of monomials is {not} known in general, and as far as we know, the first unknown case is the monomial $x_0^2x_1^2x_2^2$, whose real rank is bounded by $11 \leq \rk_{\rm sym}^{\mathbb{R}}(x_0^2x_1^2x_2^2) \leq 13$; here, the upper bound is given by \cite[Proposition 3.6 and Example 3.6]{CKOV:RealRankMonomials}, and the lower bound is given by~\cite[Example 6.7]{MMSV:RealRankTernary}.

\smallskip
We conclude with a result on real ranks of reducible real cubics, which gives a (partial) real counterpart to Theorem \ref{CGV:ReducedCubicsCC}.
\begin{thm}[{\cite[Theorem 5.6]{CGV16:RealComplexRankReducibleCubics}}]
If $F \in \bbR[x_0,\ldots,x_n]$ is a reducible cubic form essentially involving $n+1$ variables, then one and only one of the following holds:
\begin{itemize}
\item $F$ is equivalent to $x_0(\sum_{i=1}^n\epsilon_i x^2_i)$, where $\epsilon_i \in \{-1, +1\}$, for $1\leq i \leq n$, and:
$$2n \leq \rk_{\rm sym}^\bbR(F) \leq 2n + 1.$$
Moreover, if $\sum_i \epsilon_i = 0$, then $\rk_{\rm sym}^\bbR(F) = 2n$.
\item $F$ is equivalent to $x_0(\sum_{i=0}^n \epsilon_ix_i^2)$, where $\epsilon_i \in \{-1, +1\}$, for $1\leq i \leq n$, and:
$$2n \leq \rk_{\rm sym}^\bbR(F) \leq 2n + 1.$$
Moreover, if $\epsilon_0 = \ldots = \epsilon_n$, then $\rk_{\rm sym}^\bbR(F) = 2n$. If $\epsilon_0 \neq \epsilon_1$ and $\epsilon_1=\ldots=\epsilon_n$,
then $\rk_{\rm sym}^\bbR(F) = 2n + 1$.
\item $F$ is equivalent to $(\alpha x_0 + x_p)(\sum_{i=0}^n \epsilon_i x_i^2)$, for $\alpha \neq 0$, where $\epsilon_0=\ldots=\epsilon_{p-1}=1$ and $\epsilon_p=\ldots=\epsilon_n = -1$ for $1 \leq p \leq n$, and: $$2n \leq \rk_{\rm sym}^\bbR(F) \leq 2n + 3.$$
Moreover, if $\alpha = -1$ or $\alpha = 1$, then $2n + 1 \leq \rk_{\rm sym}^\bbR(F) \leq 2n + 3$.
\end{itemize}
\end{thm}%missing the Conclusions?

\bibliographystyle{alpha}

\end{document}